\documentclass[EJP]{ejpecp} 

\usepackage{enumerate}  

\usepackage{bbm}
\usepackage{bm}
\usepackage{float}

\SHORTTITLE{On the maximal offspring in a subcritical branching process} 

\TITLE{On the maximal offspring in a subcritical branching process}

\AUTHORS{%
  Benedikt~Stufler\footnote{Institute of Mathematics, University of Munich,  Theresienstr. 39, D-80333 Munich, Germany.
    \EMAIL{stufler@math.lmu.de}}}


\AMSSUBJ{Primary 60J80; 60F17; secondary 05C80; 05C0} 

\SUBMITTED{0} 
\ACCEPTED{0} 

\VOLUME{0}
\YEAR{0}
\PAPERNUM{0}
\DOI{0}

\ABSTRACT{We consider a subcritical Galton--Watson tree $\mT_n^\Omega$ conditioned on having $n$ vertices with outdegree in a fixed set $\Omega$. The offspring distribution is assumed to have a regularly varying density such that it lies in the domain of attraction of an $\alpha$-stable law for $1<\alpha \le2$. Our main results consist of a local limit theorem for the maximal degree of $\mT_n^\Omega$, and various limits describing the global shape of $\mT_n^\Omega$. In particular, we describe the joint behaviour of the fringe subtrees dangling from the vertex with maximal degree. We provide applications of our main results to establish limits of graph parameters, such as the height, the non-maximal vertex outdegrees, and fringe subtree  statistics.}

\newcommand{\ndN}{\mathbb{N}}
\newcommand{\ndZ}{\mathbb{Z}}

\newcommand{\ndR}{\mathbb{R}}

\renewcommand{\Pr}[1]{\mathbb{P}(#1)}

\newcommand{\Prb}[1]{\mathbb{P}\left(#1\right)}

\newcommand{\Ex}[1]{\mathbb{E}[#1]}

\newcommand{\Exb}[1]{\mathbb{E}\left[#1\right]}

\newcommand{\Va}[1]{\mathbb{V}[#1]}

\newcommand{\one}{{\mathbbm{1}}}

\newcommand{\convdis}{\,{\buildrel \mathrm{d} \over \longrightarrow}\,}

\newcommand{\convp}{\,{\buildrel \mathrm{p} \over \longrightarrow}\,}

\newcommand{\convas}{\,{\buildrel \mathrm{a.s.} \over \longrightarrow}\,}

\newcommand{\eqdist}{\,{\buildrel \mathrm{d} \over =}\,}

\newcommand{\atv}{\,{\buildrel \mathrm{d} \over \approx}\,}

\newcommand{\He}{\mathrm{H}}

\newcommand{\cB}{\mathcal{B}}

\newcommand{\cE}{\mathcal{E}}

\newcommand{\cN}{\mathcal{N}}

\newcommand{\cT}{\mathcal{T}}

\newcommand{\mA}{\mathsf{A}}

\newcommand{\mS}{\mathsf{S}}
\newcommand{\mT}{\mathsf{T}}

\begin{document}



\section{Introduction and main results}

\subsection{The largest degree}

Let $\mT$ be a Galton--Watson tree whose offspring distribution $\xi$ satisfies  $\Pr{\xi=0}>0$ and $\Pr{\xi \le 1}<1$. We assume that  $\Ex{\xi} < 1$ and that there is a slowly varying function $f$ and a constant $\alpha>1$ such that 
\begin{align}
\label{eq:xi}
\Pr{\xi=n} = f(n) n^{-1-\alpha}
\end{align} for all large enough integers $n$. For a fixed non-empty set $\Omega \subset \ndN_0$ of non-negative integers with $\Pr{\xi \in \Omega}>0$ we may consider the tree $\mT_n^\Omega$ obtained by conditioning $\mT$ on the event that the number of vertices with outdegree in $\Omega$ is equal to $n$. We assume that  either $\Omega$ or its complement $\Omega^c = \ndN_0 \setminus \Omega$ is finite. See Remark~\ref{re:extend1} for further comments on this restriction. Setting $\theta = \min(\alpha,2)$ we let $(X_t)_{t \ge 0}$ be the spectrally positive L\'evy process with Laplace exponent $\Ex{\exp(-\lambda X_t)} = \exp(t\lambda^\theta )$. Let $h$ be the density of $X_1$. Note that $h(z) = \frac{1}{\sqrt{4\pi}} \exp\left( -\frac{z^2}{4} \right)$ in the case $\Va{\xi}<\infty$. Our first main result is a local limit theorem for the maximal outdegree $\Delta(\mT_n^\Omega)$ of the tree~$\mT_n^\Omega$.

\begin{theorem}
	\label{te:main1}
	There is a slowly varying function $g_\Omega$ such that
	\[
	\Pr{\Delta(\mT_n^\Omega) = \ell} = \frac{1}{g_\Omega(n)n^{1/\theta}}\left(h\left(\frac{ \Pr{\xi \in \Omega}^{-1}(1-\Ex{\xi})n - \ell}{g_\Omega(n)n^{1/\theta}}   \right) + o(1)\right)
	\]
	uniformly for all $\ell \in \ndZ$.
\end{theorem}

The slowly varying function $g_\Omega$ appearing in Theorem~\ref{te:main1} admits an explicit description. 
Setting $K(x) = \Ex{\xi^2 \one_{\xi \le x}}$, define the function $g$ for sufficiently large integers $n \ge 1$ to equal
\begin{align}
\label{eq:defofg}
g(n) = \begin{cases}
\sqrt{\frac{\Va{\xi}}{2}}, \quad &\Va{\xi}<\infty \\
\sqrt{K\left(\sup\left\{x \ge 0 \mid \frac{K(x)}{x^2} \ge \frac{1}{n} \right\} \right)}, \quad &\Va{\xi}=\infty \,\,\mathrm{and}\,\, \theta= 2 \\
n^{-1/\theta}(\Gamma(1- \theta))^{1/\theta} \inf\left\{ x \ge 0 \mid \Pr{\xi > x} \le \frac{1}{n} \right\}, &\Va{\xi}=\infty \,\,\mathrm{and}\,\, 1<\theta<2. 
\end{cases}
\end{align}
The function $g_\Omega$ may be chosen to satisfy
\begin{align}
\label{eq:gslowly}
g_\Omega(n) = \begin{cases}
\frac{1}{\sqrt{2}}\left(\frac{\Ex{\xi^2}-1}{\Pr{\xi \in \Omega}} + \frac{(1- \Ex{\xi})(1 - \Ex{\xi} + 2 \Ex{\xi, \xi \in \Omega})}{\Pr{\xi \in \Omega}^2}\right)^{1/2}, \quad &\Va{\xi}<\infty \\
\frac{g(n)}{\Pr{\xi \in \Omega}^{1/\theta}}, \quad &\Va{\xi} = \infty.
\end{cases}
\end{align}

\begin{figure}[t]
	\centering
	\includegraphics[width=0.8\textwidth]{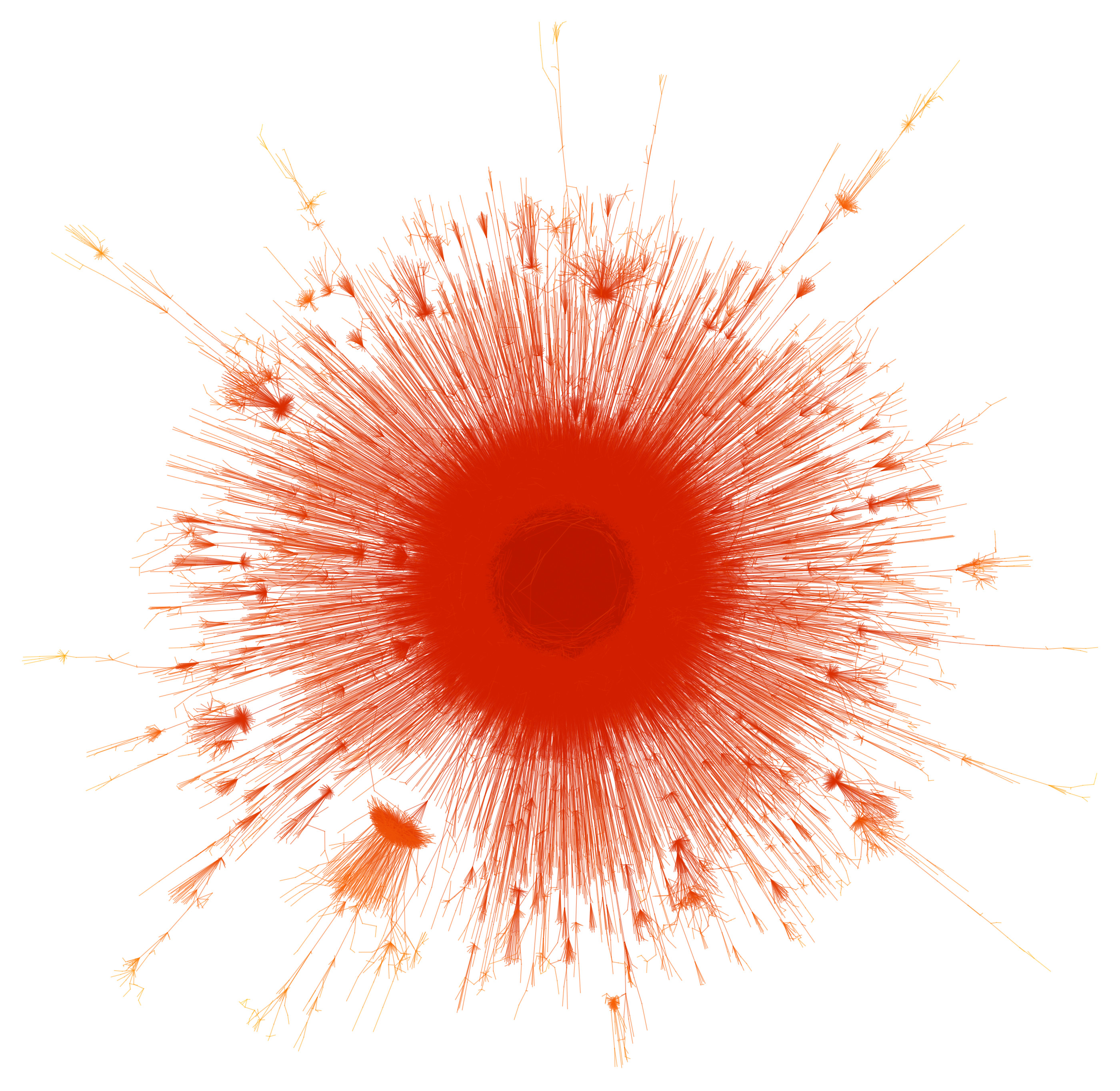}
	\caption{
		Simulation of a random simply generated tree with 100k vertices in the complete condensation phase.\protect\footnotemark The precise behaviour of the unique vertex with macroscopic degree in the center of the image is described by Theorems~\ref{te:main1} and~\ref{te:main2}. The offspring distribution was chosen to be of the form~\eqref{eq:xi} with $\alpha = 3/2$, $f(n)$ constant (except for the first term $f(0)$), and $\Ex{\xi} = 1/2$.
	}
	\label{fi:condensation}
\end{figure}

The behaviour of the maximal degree in this setting contrasts the case of a critical Galton--Watson tree, where the largest outdegree has order $o_p(n)$ ~\cite{MR2840318,MR3188595,MR2908619}, although condensation may still occur on a smaller scale, as shown in~\cite{2018arXiv180410183K}, if the offspring distribution lies in the domain of  attraction of a Cauchy law. In the subcritical regime the special case $\mT_n := \mT_n^{\ndN_0}$ was studied by Janson~\cite[Thm. 19.34]{MR2908619}, who established a central limit theorem for~$\Delta(\mT_n)$ if $\xi$ follows asymptotically a power law. This was extended to offspring laws with a regularly varying density by Kortchemski~\cite[Thm. 1]{MR3335012} using a different approach, which also inspired the present work. Hence Theorem~\ref{te:main1} generalizes these results to different kinds of conditionings and sharpens the form of convergence to a local limit theorem.  We note that the maximal displacement of subcritical branching random walks in the continuous regime also received attention in previous literature, see~\cite{zbMATH06708812,zbMATH06892193}.

\footnotetext{
	Visualization was done with Mathematica using a spring-electrical embedding algorithm. The simulation of the random tree was carried out with the author's open source software GRANT - Generate RANdom Trees - available at: \url{https://github.com/BenediktStufler/grant}. The code implements a multithreaded version of Devroy's~algorithm \cite{MR2888318} for simulating size-constrained Galton--Watson trees.
}

The case $\theta = 3/2$ is related to extremal component sizes in  uniform random planar structures. It was observed by Banderier, Flajolet, Schaeffer, and Soria~\cite[Thm. 7]{MR1871555} that the Airy distribution precisely quantifies the sizes of cores in various models of random planar maps. This phenomenon was also observed in  random graphs from planar-like classes by Gim\'enez, Noy, and Ru\'e \cite[Thm. 5.4]{MR3068033}. The local limit theorems established in these sources were obtained  using analytic methods. For uniform size-constrained  planar maps and related  models  Addario-Berry~\cite{AddBe} observed that the number of corners in the $2$-connected core is distributed like the largest outdegree in a simply generated tree. In a similar spirit S.~\cite[Cor. 6.42, Thm. 6.20]{2016arXiv161202580S} related the largest $2$-connected block in random graphs from planar-like classes and general tree-like structures to a Gibbs partition of the maximum outdegree in large Galton--Watson trees. These connections have been used in \cite{AddBe,2016arXiv161202580S} to recover the central limit theorem of the largest block in these models via a probabilistic approach. Theorem~\ref{te:main1} enables us to strengthen these alternative proofs to recover the stronger local limit theorem.

\subsection{The global shape}
We say a plane tree is marked if one of its vertices is distinguished. The path connecting the root to the marked vertex is called the spine.  The fringe subtree of a plane tree at a vertex  is the subtree consisting of the vertex and all its descendants.  Any plane tree $T$ may be fully described by 
the ordered list $(F_i(T))_{1 \le i \le \Delta(T)}$ of fringe subtrees dangling from the lexicographically first vertex $v$ with maximum outdegree, and the marked tree $F_0(T)$ obtained from $T$ by marking the vertex $v$ and  cutting away all its descendants.

We consider the size-biased random variable $\hat{\xi}$ with values in $\ndN \cup \{\infty\}$ and distribution given by
\begin{align}
\Pr{\hat{\xi} = k} = \begin{cases}k \Pr{\xi=k}, &k < \infty \\ 1 - \Ex{\xi}, &k = \infty.\end{cases}
\end{align}
Let $\mT^\bullet$ be the random marked plane tree constructed as follows. We start with a root vertex and sample an independent copy $\hat{\xi}_1$ of $\hat{\xi}$. If it is equal to infinity, then we mark the root vertex and stop the construction. Otherwise we add offspring according to $\hat{\xi}$ to the root vertex. We select one of the offspring vertices uniformly at random and declare it special. Each of the non-special offspring vertices gets identified with the root of an independent copy of the $\xi$-Galton--Watson tree $\mT$. We then iterate the construction with the special offspring vertex. In particular, the marked vertex of $\mT^\bullet$ is always a leaf. For any finite plane tree $T$ with a marked leaf $v$ it holds that
\begin{align}
\label{eq:tbullet}
\Pr{\mT^\bullet = (T,v)} = (1 - \Ex{\xi})\prod_{u \in T,u \ne v} \Pr{\xi = d^+_T(u)}.
\end{align}

Let $(\mT^i)_{i \ge 1}$ denote independent copies of the $\xi$-Galton--Watson tree $\mT$.
The following observation describes the asymptotic shape of the conditioned Galton--Watson tree $\mT_n^\Omega$. It tells us that $F_0(\mT_n)$ converges weakly to $\mT^\bullet$, and  that all but a very small number of the fringe subtrees dangling from the vertex with maximum outdegree in $\mT_n$ behave asymptotically like a list of independent copies of the unconditioned Galton--Watson tree $\mT$ that is stopped after accumulating sufficient mass:
\begin{theorem}
	\label{te:main2}
	There is a constant $C>0$ such that for any sequence of integers $(t_n)_{n \ge 1}$  with $t_n \to \infty$ and $t_n = o(n)$ it holds that
	\[
	\left(F_0(\mT_n^\Omega), (F_i(\mT_n^\Omega))_{1 \le i \le \Delta(\mT_n^\Omega) - t_n}, \one_{\sum_{i = \Delta(\mT_n^\Omega) - t_n }^{\Delta(\mT_n^\Omega)} |F_i(\mT_n^\Omega)| \ge C t_n } \right) \atv \left(\mT^\bullet, (\mT^i)_{1 \le i \le \Delta_{\langle n \rangle} - t_n},0\right)
	\]
	for 
	\[
	\Delta_{\langle n \rangle} :=  \sup \left\{ d \ge 1 \,\,\bigg\rvert\,\, L_\Omega(\mT^\bullet) +  \sum_{i=1}^d L_\Omega(\mT^i) \le n  \right\}.
	\]
\end{theorem}
Here $\atv$ denotes that the total variation distance of the two random vectors tends to zero as $n$ tends to infinity. We also let $|\cdot|$ denote the number of vertices, and $L_\Omega(\cdot)$ the number of vertices with outdegree in $\Omega$. 

Results similar to Theorem~\ref{te:main2} are known to hold for the tree $\mT_n = \mT_n^{\ndN_0}$.
Janson \cite[Thm. 20.2]{MR2908619} established convergence of $(F_0(\mT_n), F_i(\mT_n))_{1 \le i \le k})$ for $k$ a constant, assuming significantly weaker requirements on $\xi$. Specifically, assuming only that $\xi$ is heavy tailed and $\Ex{\xi}<1$ he showed that such a limit holds with respect to the lexicographically first vertex having ``large'' outdegree instead of maximum outdegree. The condition ensuring that ``large'' means maximum with high probability is $\Delta(\mT_n) = (1 - \Ex{\xi} + o_p(1))n$, which seems to be  more general than assumption~\eqref{eq:xi}.  Kortchemski ~\cite[Cor. 2.7]{MR3335012} proved a limit for the fringe subtrees $(F_i(\mT_n))_{1 \le i \le (1 - \Ex{\xi} - \epsilon)n}$ in the setting \eqref{eq:xi} for $\epsilon>0$ an arbitrarily small constant.  

Abraham and Delmas~\cite{MR3164755} established a local weak limit for the vicinity of the root vertex of $\mT_n^\Omega$ in the more general condensation regime. This implies that a vertex with large outdegree emerges close to the root. In general this vertex does not have to correspond with high probability to a vertex with maximum outdegree of $\mT_n^\Omega$. But it clearly does in the setting~\eqref{eq:xi}, yielding  weak convergence of $(F_0(\mT_n^\Omega), F_i(\mT_n^\Omega))_{1 \le i \le k})$ for any fixed constant $k$. (At least in  the case $\Omega=\ndN_0$, this is known to hold under the weaker assumption $\Delta(\mT_n^\Omega) = n(1 - \Ex{\xi})\Pr{\xi \in \Omega}^{-1} + o_p(n)$, see~\cite[Sec. 20]{MR2908619}. In the setting~\eqref{eq:xi} this is ensured by Theorem~\ref{te:main1}.) For conditioned Galton--Watson trees that encode certain types of Boltzmann planar maps  a result concerning the asymptotic behaviour of the forest of fringe subtrees dangling from a vertex with macroscopic degree was also established by Janson and Stef\'ansson~\cite{MR3342658}.

The strategy for proving the main theorems is to treat the case $\Omega=\ndN_0$ separately in Section~\ref{sec:cond1}, and then transfer the results to the general case in Sections~\ref{sec:cond} and~\ref{sec:0not}. The transfer is performed using a blow-up procedure due to Ehrenborg and M\'endez~\cite{MR1284403}, Minami~\cite{MR2135161}, and Rizzolo~\cite{MR3335013}, where every vertex is expanded at random into a vertebrate. The details are rather technical and use Gibbs partition  results from S.~\cite{doi:10.1002/rsa.20771} to describe the asymptotic behaviour of the vertebrates corresponding to vertices with large degree. The approach for the case $\Omega=\ndN_0$ is inspired by the work of Kortchemski~\cite{MR3335012}. For this case, the local limit result in Theorem~\ref{te:main1} is proved using a connection between random trees and random walks, combined with estimates of Denisov, Dieker and Shneer~\cite{MR2440928} on the big-jump domain for random walks. The asymptotic description of the fringe subtrees in  Theorem~\ref{te:main2} is proved by combining  approximation results by Armend{\'a}riz and Loulakis~\cite{MR2775110} with a path decomposition. 

Let us note a few implications of Theorem~\ref{te:main2}. First, if a sequence of positive integers $(r_n)_{n \ge 1}$ satisfies $r_n / (g(n) n^{1/\theta}) \to \infty$ and $r_n  = o(n)$, then $ \Pr{\Delta(\mT_n^\Omega) > \Pr{\xi \in\Omega}^{-1}(1-\Ex{\xi})n - r_n} \to 1$ by the central limit theorem for $\Delta(\mT_n^\Omega)$ implied by Theorem~\ref{te:main1}. Hence it follows by Theorem~\ref{te:main2}:

\begin{corollary}
	\label{co:m1}
	If a sequence of positive integers $(r_n)_{n \ge 1}$ satisfies $r_n / (g(n) n^{1/\theta}) \to \infty$ and $r_n  = o(n)$, then
	\begin{align}
	\label{eq:tsecond}
	\left(F_0(\mT_n^\Omega), (F_i(\mT_n^\Omega))_{1 \le i \le \Pr{\xi \in \Omega}^{-1}(1 - \Ex{\xi})n - r_n}  \right) \atv \left(\mT^\bullet, (\mT^i)_{1 \le i \le (1 - \Ex{\xi})\Pr{\xi \in \Omega}^{-1} n - r_n}\right).
	\end{align}
\end{corollary}

We may view $(\Delta_{\langle n \rangle})_{n \ge 1}$ as a renewal process with inter-arrival times distributed like $L_\Omega(\mT)$. The number $L_\Omega(\mT)$ of vertices with outdegree in $\Omega$ has a probability density that varies regularly with index $-1-\alpha$. Specifically,
\begin{align}
\label{eq:gendensity}
\Pr{L_\Omega(\mT) = n} \sim \frac{f(n) \Pr{\xi \in \Omega}^\alpha}{(1 - \Ex{\xi})^{1+\alpha}} n^{-1-\alpha}.
\end{align}
See Equations~\eqref{eq:lomegat1} and~\eqref{eq:lomegat2} below. 

In Appendix~\ref{sec:appendix} we collect a few results on general renewal processes with inter-arrival times having regularly varying probability densities. For example,  Lemma~\ref{le:doit2} implies that the first $o(n)$ fringe subtrees dangling from the vertex with maximal degree in $\mT_n^\Omega$ asymptotically become independent from each other~\emph{and} the maximal degree:

\begin{corollary}
	\label{co:m2}
	Let $\Delta_{[n]}$ denote an identically distributed copy of $\Delta(\mT_n^\Omega)$ that is independent from $\mT^\bullet$ and $(\mT_i)_{i \ge 1}$.  If a sequence of positive integers $(m_n)_{n \ge 1}$ satisfies $m_n = o(n)$, then 
	\begin{align}
	\label{eq:last}
	\left(F_0(\mT_n^\Omega), (F_i(\mT_n^\Omega))_{1 \le i \le m_n}, \Delta(\mT_n^\Omega)  \right) \atv \left(\mT^\bullet, (\mT^i)_{1 \le i \le m_n}, \Delta_{[n]}\right).
	\end{align}
\end{corollary}

Note the contrast between Corollaries~\ref{co:m1} and~\ref{co:m2}. Almost all fringe subtrees dangling from the vertex with maximal degree in $\mT_n^\Omega$ become asymptotically independent from each other, but only a small $o(n)$ number additionally becomes independent from the maximal degree itself. On an intermediary scale, it appears  that the best we can achieve is a a weaker \emph{contiguousness} relation that follows from~Lemma~\ref{le:doit}:

\begin{corollary}
	\label{cor:m3}
	Let $0<\delta < \Pr{\xi\in\Omega}^{-1}(1 - \Ex{\xi})$ be given.  For any $\epsilon>0$ there are constants $0<c<C$ and a number $N_0>0$ such that for all $n \ge N_0$ and all events $\cE$
	\begin{align}
	\label{eq:contig}
	c \Prb{\left(\mT^\bullet, (\mT^i)_{1 \le i \le \Delta_{[ n ]}- \delta n}\right) \in \cE} - \epsilon &\le \Prb{\left(F_0(\mT_n^\Omega), (F_i(\mT_n^\Omega))_{1 \le i \le \Delta(\mT_n^\Omega) - \delta n}  \right) \in \cE} \\&\le 	C \Prb{\left(\mT^\bullet, (\mT^i)_{1 \le i \le \Delta_{[ n ]}- \delta n}\right) \in \cE} + \epsilon. \nonumber
	\end{align}
	That is, $\cE$ may be any collection of finite sequences of finite plane trees, with the first tree additionally carrying a marked leaf.
\end{corollary}
Note that the family $\left(F_0(\mT_n^\Omega), (F_i(\mT_n^\Omega))_{1 \le i \le \Delta(\mT_n^\Omega) - \delta n}  \right)$ determines $\Delta(\mT_n^\Omega)$. Thus any property that holds with high probability for $\left(F_0(\mT_n^\Omega), (F_i(\mT_n^\Omega))_{1 \le i \le \Delta(\mT_n^\Omega) - \delta n}, \Delta(\mT_n^\Omega)  \right)$ also holds with high probability for $\left(\mT^\bullet, (\mT^i)_{1 \le i \le \Delta_{[ n ]}- \delta n}, \Delta_{[n]}\right)$, and vice versa.

\subsection{Limits of graph parameters}

We postpone the complex proofs of Theorems~\ref{te:main1} and \ref{te:main2} to Sections~\ref{sec:cond1},~\ref{sec:cond} and~\ref{sec:0not}. In the present section we collect and prove applications concerning limits of the height,  fringe subtree statistics, and the non-maximal vertex degrees.

\subsubsection{The height}
We let  $\He(\cdot)$ denote the height.  The unconditioned $\xi$-Galton--Watson tree  $\mT$ is known \cite[Thm. 2]{MR0217889} to satisfy
\begin{align}
\label{eq:h0}
\Pr{\He(\mT)=n} &\sim c_{\mathrm{H}} \Ex{\xi}^n
\end{align}
for some constant $c_{\mathrm{H}} > 0$. Theorems~\ref{te:main1} and \ref{te:main2} tell us that 
\begin{align}
\label{eq:h1}
\He(\mT_n^\Omega)  \atv \He(\mT^\bullet) + \max(\He(\mT^1), \ldots, \He(\mT^{\Delta_{\langle n \rangle}})).
\end{align}
Note that the height $\He(\mT^\bullet)$  follows a geometric geometric distribution
\begin{align}
\label{eq:h2}
\Pr{\He(\mT^\bullet) = k} = \Ex{\xi}^k (1 - \Ex{\xi}).
\end{align}
Using extreme value statistics it follows that for any sequence $k_n$ with $k_n \to \infty$ it holds uniformly for all integers $k \ge k_n$
\begin{multline}
\label{eq:h3}
\Prb{ \max_{1 \le i \le n} \He(\mT^i) \le k} = \\
\exp\left( -\exp\left( \log\left(n\right) -k \log\left(\frac{1}{\Ex{\xi}}\right) +  \log\left(\frac{c_{\mathrm{H}}\Ex{\xi}}{1 - \Ex{\xi}}\right)\right)(1 + o(1)) \right)
\end{multline}
and hence
\begin{align}
\max_{1 \le i \le n} \He(\mT^i) = \frac{\log n}{\log\left( \frac{1}{\Ex{\xi}} \right) } + O_p(1).
\end{align}
Using Corollary~\ref{co:m1} and a time-reversal argument, it follows that
\begin{align}
\label{eq:ccomb2}
\max_{1 \le i \le \Delta_{\langle n \rangle}} \He(\mT^i) \atv \max_{1 \le i \le \frac{1 - \Ex{\xi}}{\Pr{\xi \in \Omega}} n} \He(\mT^i) =  \frac{\log n}{\log\left( \frac{1}{\Ex{\xi}} \right) } + O_p(1).
\end{align}
Combining~\eqref{eq:h1} and~\eqref{eq:ccomb2}, we obtain:
\begin{corollary}
	\label{co:height}
	The height of the tree $\mT_n^\Omega$ satisfies
	\begin{align} 
	\label{eq:stbound}
	\He(\mT_n^\Omega) = \log(n) / \log(1/\Ex{\xi}) + O_p(1).
	\end{align}
\end{corollary}
The result agrees with the case $\Omega=\ndN_0$ established by Kortchemski in~\cite[Thm. 4]{MR3335012}. 
Convergence of moments may be verified as well:
\begin{proposition}
	\label{pro:mom}
	For any fixed constant $p \ge 1$ 
	\begin{align}
	\Exb{\left( \frac{\He(\mT_n^\Omega)}{\log n } \right)^{p}} \to \left(\log(1/\Ex{\xi}) \right)^p.
	\end{align}
\end{proposition}
The proof is by similar arguments as for~\cite[Prop. 2.11]{MR3335012}, where the case $\Omega = \ndN_0$ was treated:
\begin{proof}[Proof of Proposition~\ref{pro:mom}]
	It follows from~\eqref{eq:stbound} that
	\begin{align}
	\frac{\He(\mT_n^\Omega)}{\log n }  \convp  \log(1/\Ex{\xi}).
	\end{align}
	Hence it suffices to show that $\frac{\He(\mT_n^\Omega)}{\log n }$ is $p$-uniformly integrable. For this, it suffices to show that $\frac{\He(\mT_n^\Omega)}{\log n }$ is bounded in $\mathbb{L}^q$ for some fixed $q>p$. Let $K>0$ be a constant. Then
	\begin{align}
	\label{eq:k1}
	\Exb{\left( \frac{\He(\mT_n^\Omega)}{\log n } \right)^{q}} \le  K^q + \Exb{\He(\mT_n^\Omega)^{q}, \He(\mT_n^\Omega) > K \log n}.
	\end{align}
	Using the trivial bound $\He(\mT_n^\Omega) \le |\mT_n^\Omega|$, it follows that 
	\begin{multline}
	\label{eq:k2}
	\Exb{\He(\mT_n^\Omega)^{q}, \He(\mT_n^\Omega) > K \log n} \le \\ \frac{2n}{\Pr{\xi \in \Omega}} \Pr{ \He(\mT_n^\Omega) > K \log n} + \Exb{|\mT_n^\Omega|^{q},  |\mT_n^\Omega| >\frac{2n}{\Pr{\xi \in \Omega}}  }.
	\end{multline}
	Letting $(\xi_i)_{i \ge 1}$ denote independent copies of $\xi$, it follows from the Chernoff bounds, the Potter bounds, and Equation~\eqref{eq:gendensity} that uniformly for all integers $k > 2n / \Pr{\xi \in \Omega}$
	\begin{align}
	\Pr{|\mT_n^\Omega| = k} &\le \Pr{L_\Omega(\mT) = n}^{-1} \Prb{\sum_{i=1}^k \one_{\xi_i \in \Omega} = n} \\
	&\le O(n^{2+\alpha}) \exp\left( - \frac{|n - k \Pr{\xi \in \Omega}|^2}{2 k}\right) \nonumber \\
	&\le \exp( - \Theta(k))  .\nonumber
	\end{align}
	This entails
	\begin{align}
	\label{eq:k3}
	\Exb{|\mT_n^\Omega|^{q},  |\mT_n^\Omega| >\frac{2n}{\Pr{\xi \in \Omega}}  } = \exp(-\Theta(n)).
	\end{align}
	Furthermore, using~\eqref{eq:h0}, the Potter bounds, and Equation~\eqref{eq:gendensity} it follows that 
	\begin{align}
	\label{eq:k4}
	\Pr{ \He(\mT_n^\Omega) > K \log n} \le O\left(n^{2+\alpha- K\log\left(1/\Ex{\xi}\right)}\right).
	\end{align}
	Taking $K$ sufficiently large, it follows from Inequalities~\eqref{eq:k1},~\eqref{eq:k2},~\eqref{eq:k3}, and~\eqref{eq:k4}
	\begin{align}
	\Exb{\left( \frac{\He(\mT_n^\Omega)}{\log n } \right)^{q}} \le K^q + \exp(-\Theta(n)) + O\left(n^{3+\alpha- K\log\left(1/\Ex{\xi}\right)}\right) = K^q + o(1).
	\end{align}
	This verifies that~$\left( \frac{\He(\mT_n^\Omega)}{\log n } \right)^{q}$ is bounded in $\mathbb{L}^p$, and completes the proof.
\end{proof}

\subsubsection{Counting fringe subtrees}

For any finite plane tree $T$ we may consider the functional $N_T(\cdot )$ that takes a plane tree as input and returns the number of occurrences of $T$ at a fringe. It is easy to see that the unconditioned $\xi$-Galton--Watson tree $\mT$ satisfies
\begin{align}
\Ex{N_T(\mT)} = \Pr{\mT = T} / (1 -\Ex{\xi}).
\end{align}
By Corollary~\ref{co:m1} and a time-reversal argument it follows that
\begin{align}
\frac{N_T(\mT_n^\Omega)}{n \Pr{\xi \in \Omega}^{-1}} \convp \Pr{\mT = T}.
\end{align}
This agrees with known results for the case $\Omega=\ndN_0$, see~\cite[Thm. 7.12]{MR2908619}.



We may also derive exponential concentration inequalities: Let $F(\cdot)$ denote a function that takes a list of at least $|T|$ integers as input, extends it cyclically, and returns the number of occurrences of the depth-first-search ordered list of outdegrees of $T$ as a substring of the cyclically extended input. Note that such substrings cannot overlap. Hence changing a single coordinate of the input changes the value of $F$ by at most $1$. 

Recalling that $d_1, \ldots, d_{|\mT_n^\Omega|}$ denotes the outdegree list corresponding to $\mT_n^\Omega$, we may write
\begin{align}
\label{eq:nt}
N_T(\mT_n^\Omega) = F(d_1, \ldots, d_{|\mT_n^\Omega|}).
\end{align}
Letting  $(\xi)_{i \ge 1}$ denote independent copies of $\xi$, it follows by McDiarmid's inequality that for all $x>0$ and  $\ell \ge |T|$
\begin{align}
\label{eq:F}
\Pr{ |F(\xi_1, \ldots, \xi_\ell) - \Pr{\mT= T}\ell| \ge x } \le 2 \exp( - 2	x^2 / \ell).
\end{align}
Here we have used that $\Ex{F(\xi_1, \ldots, \xi_\ell)} = \ell \Pr{\mT=T}$, since $F$ cyclically extends the input list. The number $L_\Omega(\mT)$ of vertices of the unconditioned $\xi$-Galton--Watson tree $\mT$  with outdegree in $\Omega$ is known to equal the total of number of vertices in a Galton--Watson tree with a different offspring distribution~\cite{MR3335013} (that is critical/subcritical heavy-tailed/light-tailed if and only if $\xi$ is). See Sections~\ref{sec:cond} and~\ref{sec:0not}. Hence it follows by a general result of Janson~\cite{MR2908619}  that
\begin{align}
\label{eq:partition}
\Pr{L_\Omega(\mT) = n} = \exp(o(n)).
\end{align}
We may write
\begin{multline*}
\Pr{|\mT_n^\Omega| \notin (1 \pm \epsilon) n/\Pr{\xi \in \Omega} } \le \\ \Pr{L_\Omega(\mT)=n}^{-1} \Prb{L_\Omega(\mT) \notin \frac{1}{1 \pm \epsilon}|\mT|\Pr{\xi \in \Omega}, L_\Omega(\mT)=n}.
\end{multline*}
As plane trees correspond to cyclic shifts of balls-in-boxes configurations (see Equation~\eqref{eq:teq}), the Chernoff bounds and Equation~\eqref{eq:partition} imply that this bound simplifies to $\exp(-\Theta(n))$. Using  Equations~\eqref{eq:nt}  this implies
\begin{multline*}
\Prb{N_T(\mT_n^\Omega) \notin (1 \pm \epsilon)\Pr{\mT= T) n/ \Pr{\xi \in \Omega}}}  \le \\ \exp(-\Theta(n)) + \exp(o(n)) \sum_{\ell \in (1 \pm \epsilon) n/\Pr{\xi \in \Omega} }\Prb{F(\xi_1, \ldots, \xi_\ell) \notin (1 \pm \epsilon) \frac{\Pr{\mT= T}{\Pr{\xi \in \Omega}}} n}.
\end{multline*}
By Equation~\eqref{eq:F} this bound simplifies to $\exp(-\Theta(n))$. We obtain:
\begin{lemma}
	\label{le:free}
	For any $\epsilon>0$ it holds that
	\begin{align}
	\label{eq:fringesubtrees}
	\Prb{ \left| \frac{ N_T(\mT_n^\Omega)}{n\Pr{\xi \in \Omega}^{-1}}  - \Pr{\mT= T} \right| \ge \epsilon} &\le \exp(-\Theta(n)), \\
	\label{eq:ldev}
	\Prb{ | |\mT_n^\Omega| - n \Pr{\xi \in \Omega}^{-1}| \ge \epsilon } &\le \exp(-\Theta(n)).
	\end{align}
\end{lemma}
This agrees with the known case $\Omega=\ndN_0$, see for example~\cite[Thm. 6.5]{2016arXiv161202580S}.

\begin{remark}
	The proof of Lemma~\ref{le:free} does not use any of the assumptions on $\xi$. Hence Inequality~\eqref{eq:fringesubtrees} holds for any  offspring distributions $\xi$ (with $\Pr{\xi=0}>0$, $\Pr{\xi \ge 2} > 0$, and $\Pr{\xi \in \Omega} > 0$), that is either critical, or subcritical and heavy-tailed. Furthermore, it is well-known that if $\xi$ subcritical and light-tailed, then  the study of $\mT_n^\Omega$ may be reduced to one of these two cases by tilting the offspring distribution.
\end{remark}

Lemma~\ref{le:free} implies by the  Borel--Cantelli Lemma that
\begin{align}
\frac{ N_T(\mT_n^\Omega)}{|\mT_n^\Omega|} \convas \Pr{\mT= T}
\end{align}
This may be expressed equivalently in terms of random probability measures. Take the random tree $\mT_n^\Omega$, and let $\mu_n$ denote the (random) law of the fringe subtree at a uniformly selected vertex of $\mT_n^\Omega$. Then
\begin{align}
\mu_n \convas \mathcal{L}(\mT),
\end{align}
with $\mathcal{L}(\mT)$ denoting the law of the unconditioned Galton--Watson tree $\mT$.

\subsubsection{Sizes of subtrees dangling from the vertex with maximal degree}
It follows from~\cite[Cor. 2.1]{MR2440928} (see Equation~\eqref{eq:gwt} below for details) that
\begin{align}
\label{eq:uncond}
\Pr{|\mT|=n} \sim\Pr{\xi = \lfloor(n-1)(1- \Ex{\xi})\rfloor} \sim \frac{f(n)}{ (n(1-\Ex{\xi}))^{1+\alpha}}.
\end{align}
Similar as for the height in \eqref{eq:h1}, we obtain by extreme value statistics the following limit for the  maximal size of the fringe subtrees $F_i(\mT_n^\Omega)$, $1 \le i \le  \Delta(\mT_n^\Omega)$ dangling from the vertex with maximum degree.
\begin{corollary}
	\label{co:fringe}
	There is a slowly varying function $f_0(n)$ with
	\begin{align}
	\label{eq:fringe}
	f_0(n) n^{-1/\alpha} \max_{1 \le i \le \Delta(\mT_n^\Omega)} |F_i(\mT_n^\Omega)| \convdis \mathrm{\text{Fr\'echet}}(\alpha). 
	\end{align}
\end{corollary}
The slowly varying function $f_0$ may chosen to satisfy
\begin{align}
f_0(n) = n^{-1/\alpha}\inf\left\{m \ge 1 \,\,\Bigg\vert\,\, \Pr{|\mT| > m} \le \frac{\Pr{\xi \in \Omega}}{(1 - \Ex{\xi})n} \right\}.
\end{align}
By Karamata's theorem it holds that
\begin{align}
\Pr{|\mT| >n} \sim \frac{f(n)}{\alpha(1- \Ex{\xi})^{1 + \alpha}} n^{-\alpha}.
\end{align}
Thus, if for example  $f(n) \sim c$ for some constant $c>0$, then we may set
\begin{align}
f_0(n) = \left( \frac{c}{\alpha \Pr{\xi \in \Omega}} \right)^{1/\alpha} \frac{1}{1- \Ex{\xi}}.
\end{align}
Or, if $f(n)$ is any slowly varying function and $1<\alpha<2$, Equation~\eqref{eq:defofg} entails that we may set
\begin{align}
f_0(n) = \left( \frac{1}{\Pr{\xi \in\Omega} \Gamma(1-\alpha)} \right)^{1/\alpha} \frac{g(n)}{1- \Ex{\xi}}.
\end{align}

We let $\mathbb{D}([0,1], \ndR)$ denote the space of all c\`adl\`ag functions $[0,1] \to \ndR$, endowed with the Skorokhod $J_1$-topology. Equation~\eqref{eq:uncond} implies that the random walk $(Z_i)_{i \ge 1}$ with step-size distribution $|\mT|$ satisfies
\begin{align}
\left(\frac{Z_{\lfloor nt\rfloor} - nt / (1 - \Ex{\xi})}{f_1(n) n^{1/\theta}} , 0 \le t \le 1 \right) \convdis (X_t)_{0 \le t \le 1}
\end{align}
in the space $\mathbb{D}([0,1], \ndR)$ for some slowly varying function $f_1(n)$. By~\cite[Lem. 2.10]{MR3335012}, we may set
\begin{align}
f_1(n) = \frac{g(n)}{(1- \Ex{\xi})^{1 + 1/\theta}}
\end{align} 
for the function $g$  defined in Equation~\eqref{eq:defofg}. Using~\cite[Lem. 5.7]{MR2946438}, we obtain:
\begin{corollary}
	\label{co:yo}
	It holds that
	\begin{align}
	\left(\frac{\sum_{i=1}^{\lfloor \Delta(\mT_n^\Omega) t\rfloor}|F_i(\mT_n^\Omega)| - \Delta(\mT_n^\Omega)t / (1 - \Ex{\xi})}{\frac{g(n)}{\Pr{\xi \in \Omega}^{1/\theta}} n^{1/\theta}}, 0 \le t \le 1 \right) \convdis \left(\frac{1}{1- \Ex{\xi} }X_t\right)_{0 \le t \le 1}
	\end{align}
	in the space $\mathbb{D}([0,1], \ndR)$.
\end{corollary}
This was established by Kortchemski~\cite[Thm. 3]{MR3335012} for the case $\Omega=\ndN_0$. The proof of Corollary~\ref{co:yo} is by similar arguments.

\subsubsection{The remaining vertex outdegrees}
We order the outdegrees of the tree $\mT_n^\Omega$ in descending order $\Delta_1(\mT_n^\Omega) \ge \Delta_2(\mT_n^\Omega) \ge \ldots$. It was shown in \cite[Prop. 3.2]{MR3687241} that the maximal outdegree of the unconditioned Galton--Watson tree $\mT$ satisfies
\begin{align}
\label{eq:s0}
\Pr{\Delta(\mT) = n} &\sim \Pr{\xi=n}/(1 - \Ex{\xi}).
\end{align}
Similar as for the height in \eqref{eq:h1}, it follows that the second largest degree $\Delta_2(\mT_n^\Omega)$ satisfies
\begin{align}
\label{eq:s1}
\Delta_2(\mT_n^\Omega)  \atv \max(\Delta(\mT^1), \ldots, \Delta(\mT^{\Delta_{\langle n \rangle}})).
\end{align}
By extreme value statistics we obtain analogously as for~\eqref{co:height}:
\begin{corollary}
	\label{co:seconddeg}
	There is a slowly varying function $f_2(n)$ with
	\begin{align}
	\label{eq:seconddeg}
	\Delta_2(\mT_n^\Omega) / (f_2(n) n^{1/\alpha}) \convdis \mathrm{\text{Fr\'echet}}(\alpha).
	\end{align}
\end{corollary}
This generalizes the case $\Omega=\ndN_0$ established by Janson~\cite[Thm. 19.34]{MR2908619}) and  Kortchemski~\cite[Thm. 1]{MR3335012}. The slowly varying function $f_2$ may be set to
\begin{align}
\label{eq:dohhhh}
f_2(n) = n^{-1/\alpha}\inf\left\{m \ge 1 \,\,\Bigg\vert\,\, \Pr{\Delta(\mT) > m} \le \frac{\Pr{\xi \in \Omega}}{(1 - \Ex{\xi})n} \right\}.
\end{align}
By Karamata's theorem it holds that
\begin{align}
\label{eq:ddddd}
\Pr{\Delta(\mT) >n} \sim \frac{f(n)}{\alpha(1- \Ex{\xi})} n^{-\alpha} \sim \frac{\Pr{\xi > n}}{1 - \Ex{\xi}}.
\end{align}
So, we may just as well set
\begin{align}
\label{eq:forreal}
f_2(n) = n^{-1/\alpha}\inf\left\{m \ge 1 \,\,\Bigg\vert\,\, \Pr{\xi > m} \le \frac{\Pr{\xi \in \Omega}}{n} \right\}.
\end{align}

It is not clear how to evaluate this in general. However, for some special cases it is fairly easy: For example, if $f(n) \sim c$ for some constant $c>0$, then we may set
\begin{align}
f_2(n) = \left( \frac{c}{\alpha \Pr{\xi \in \Omega}} \right)^{1/\alpha}.
\end{align}
Or, if $f(n)$ is any slowly varying function and $1<\alpha<2$, Equation~\eqref{eq:defofg} entails that we may set
\begin{align}
f_2(n) = \left( \frac{1}{\Pr{\xi \in\Omega} \Gamma(1-\alpha)} \right)^{1/\alpha} g(n).
\end{align}

More generally, we may also treat the $i$th largest outdegree for $i \ge 2$:
\begin{theorem}
	\label{te:smalldeg}
	For any $i \ge 2$ and any fixed $x > 0$ it holds that
	\begin{align}
	\Prb{\frac{\Delta_i(\mT_n^\Omega)}{f_2(n)n^{1/\alpha}} \le x} \to \exp(-x^{-\alpha}) \sum_{s=0}^{i-1} \frac{x^{-\alpha s}}{s!}.
	\end{align}
	That is,
	\begin{align}
	\frac{\Delta_i(\mT_n^\Omega)}{f_2(n)n^{1/\alpha}} \convdis W_i
	\end{align}
	for a random variable $W_i$ having density
	\begin{align}
	\frac{\alpha  \exp(-x^{-\alpha })
		\left(x^{-\alpha }\right)^i}{x (i-1)!}
	\end{align}
	for $x>0$.
\end{theorem}
\begin{proof}
	Let us fix a constant $\delta>0$ with $1/\theta < \delta < 1$, and  set 
	\[
	u_n = \frac{1 - \Ex{\xi}}{\Pr{\xi \in \Omega}} n - n^\delta.
	\]
	From Equation~\eqref{eq:s0}, Theorem~\ref{te:main1},  Corollary~\ref{co:m1} and a time-reversal argument it follows that the largest outdegree  in the forest consisting of $F_0(\mT_n^\Omega)$ and $(F_i(\mT_n^\Omega))_{ u_n \le i \le \Delta(\mT_n^\Omega)}$ has order $O_p(n^{\delta/\alpha + o(1)})$. Hence it suffices to show that the $(i-1)$th largest outdegree in the forest $(F_i(\mT_n^\Omega))_{1 \le i \le u_n}$ admits $W_i$ as distributional limit after rescaling by $f_2(n) n^{1/\alpha}$.
	
	Corollary~\ref{co:m1} tells us that
	\[
	(F_i(\mT_n^\Omega))_{1 \le i \le u_n} \atv (\mT^i)_{1 \le i \le u_n},
	\]
	with $\mT^1, \mT^2, \ldots$ denoting independent copies of $\mT$. Let $\xi_1, \xi_2, $ denote independent copies of $\xi$ and set for all integers $k < 0$
	\[
	\tau_k = \inf\left\{d \ge 1 \,\,\Bigg\vert\,\, \sum_{i=1}^d (\xi_i -1) = k\right\}.
	\]
	The lexicographically ordered list of outdegrees in $(\mT^i)_{1 \le i \le u_n}$ is distributed like
	\[
	(\xi_1, \ldots, \xi_{\tau_{-u_n}}).
	\]
	Note that~\eqref{eq:uncond} entails
	\[
	\tau_{-u_n} = u_n / (1- \Ex{\xi}) + O_p(n^{1/\theta + o(1)}).
	\]
	In particular, $\tau_{-u_n} > n/\Pr{\xi \in \Omega} - n^{\delta}$ with a probability that tends to $1$ as $n \to \infty$.  It follows that the $(i-1)$th largest entry $M_n$ of the list $	(\xi_1, \ldots, \xi_{\tau_{-u_n}})$ is with high probability among the first $		n/\Pr{\xi \in \Omega} - n^{\delta}$ coordinates. That is, asymptotically distributed like the $(i-1)$th largest outdegree in $ \lfloor	n/\Pr{\xi \in \Omega} - n^{\delta} \rfloor$ independent copies of $\xi$. 	By extreme value statistics, see~\cite[Thm. 2.2.2]{MR691492}, and Equation~\eqref{eq:forreal} it follows that
	\[
	M_n / (f_2(n)n^{1/\alpha}) \convdis W_i.
	\]
	This completes the proof.
\end{proof}

\begin{remark}
	We proved Theorem~\ref{te:smalldeg} as a consequence of the main theorems. At least in the case $0 \in \Omega$, there is also a nice and short alternative approach:
	In Section~\ref{sec:cond} we describe for the case $0 \in \Omega$ how $\mT_n^\Omega$ may be sampled by taking a simply generated tree $\tilde{\mT}_n$ with $n$ vertices (with offspring distribution described in Equation~\eqref{eq:dec1}) and blowing up each vertex into an ancestry line by a process illustrated in Figure~\ref{fi:blowup}. This construction  goes back to Ehrenborg and M\'endez~\cite{MR1284403}, and was fruitfully applied in the probabilistic literature \cite{MR2135161,MR3335013, MR3164755,MR3227065}.  Applying results from \cite{MR2775110} and \cite[Cor. 2.7]{MR3335012} to the tree $\tilde{\mT}_n$ (compare with Equation~\eqref{eq:pre1} below) yields that the depth-first-search ordered list $\tilde{d}_1, \ldots, \tilde{d}_n$ satisfies the following: If $\tilde{j}_0$ denotes the smallest index with $\tilde{d}_{\tilde{j}_0} = \Delta(\tilde{\mT}_n)$, then 
	\begin{align}
	(\tilde{d}_{\tilde{j}_0}, \ldots, \tilde{d}_n, \tilde{d}_1, \ldots, \tilde{d}_{\tilde{j}_0 -1}) \atv (n - 1 - \tilde{\xi}_1 -  \ldots - \tilde{\xi}_{n-1}, \tilde{\xi}_1, \ldots, \tilde{\xi}_{n-1}).
	\end{align}
	Here $(\tilde{\xi}_i)_{i \ge 1}$ denote independent copies of the branching mechanism $\tilde{\xi}$ with probability generating function $\tilde{\phi}(z)$ stated in Equation~\eqref{eq:dec1}. The blow-up procedure transforms $\tilde{\xi}$ into a depth-first-search order respecting segment \begin{align}
	D := (\xi_1^{\Omega^c}, \ldots, \xi_L^{\Omega^c}, \xi^{\Omega})
	\end{align} of independent outdegrees. Here $(\xi_i^{\Omega^c})_{i \ge 1}$ denotes independent copies of $(\xi \mid \xi \in \Omega^c)$, $\xi^\Omega \eqdist (\xi \mid \xi \in \Omega)$, and $L$ an independent geometrically distributed integer with distribution 
	\begin{align}
	\Pr{L=k} = \Pr{\xi \in \Omega^c}^k/(1 - \Pr{\xi \in \Omega^c}).
	\end{align}
	(In case $\Omega = \ndN_0$ this means that $L=0$ is almost surely constant.) The description of the blowup of the first coordinate $ \sum_{i=1}^{n-1}(1 - \tilde{\xi}_i)$ is more delicate, and carried out in Lemma~\ref{le:capsule}. Let $(D_i)_{i \ge 1}$ be independent copies of $D$, and 		let $(L_i)_{i \ge 1}$ denote independent copies of $L$. We let $\frown$ denote a binary operator that concatenates any two given lists.  We obtain: 
	\begin{enumerate}[\qquad a)]
		\item 
		If $|\Omega^c| < \infty$, then 
		\begin{align}
		(d_{j_0+1}, \ldots, d_{ |\mT_n^\Omega|}, d_1, \ldots, d_{j_0 -1}) \atv D_1 \frown \ldots \frown D_{n-1} \frown( \xi_1^{\Omega^c}, \ldots, \xi_L^{\Omega^c}).
		\end{align}
		In particular,
		\begin{align}
		|\mT_n^\Omega| \atv \sum_{i=1}^n(1 + L_i) \eqdist n + \mathrm{NB}(n,\Pr{\xi \in \Omega}).
		\end{align}
		\item If $|\Omega| < \infty$, then 
		\begin{align}
		(d_{j_0+1}, \ldots, d_{ |\mT_n^\Omega|}, d_1, \ldots, d_{j_0 -1}) \atv D_1 \frown \ldots \frown D_{n}\frown( \xi_1^{\Omega^c}, \ldots, \xi_L^{\Omega^c}).
		\end{align}
		In particular,
		\begin{align}
		|\mT_n^\Omega| \atv \sum_{i=1}^{n+1}(1 + L_i) \eqdist n+1 + \mathrm{NB}(n+1,\Pr{\xi \in \Omega}).
		\end{align}
	\end{enumerate}
	From this we may directly deduce Theorem~\ref{te:smalldeg} using extreme value statistics.
\end{remark}

\subsection{Comparison with critical Galton--Watson trees}

Figures~\ref{fi:pic1}--\ref{fi:pic4} illustrate typical behaviour of large $n$-vertex Galton--Watson trees whose offspring distribution $\xi$ has a regularly varying density with index $-(\alpha+1)$. Each figure shows a  drawing of a simulation of the tree in the top left corner, and  the associated looptree (obtained by blowing up any vertex with outdegree $d$ into a cycle of circumference $d+1$, see~\cite{MR3286462}) in the top right corner. The bottom left corner shows the \L{}ukasiewicz path associated to the tree, and the bottom right corner the height process.  The colour gradient corresponds to the height of a vertex in the tree, and is used consistently in all four corners.

\begin{figure}[H]
	\centering
	\includegraphics[width=0.75\textwidth]{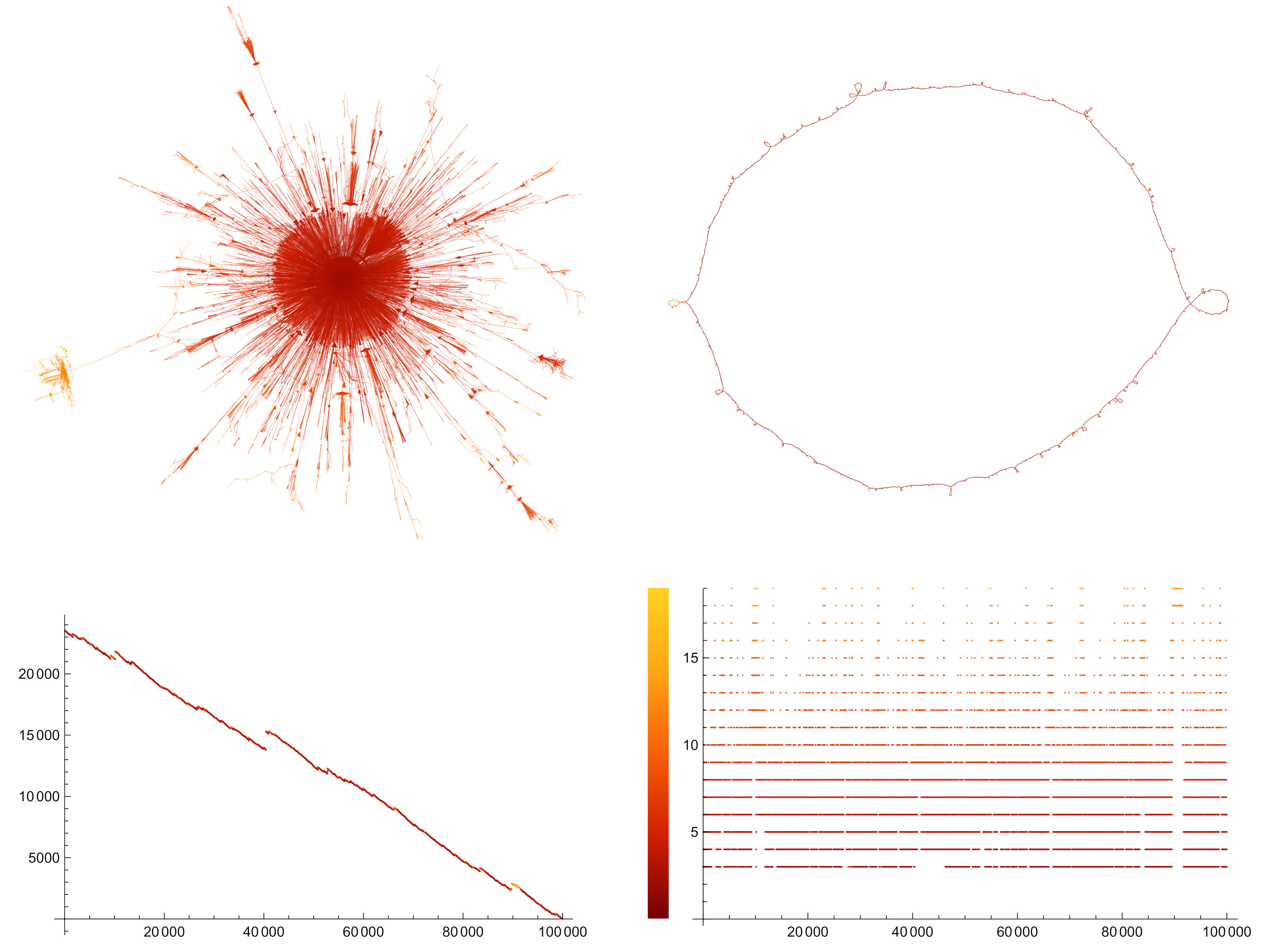}
	\caption{
		A subcritical Galton--Watson tree with 100k vertices. The offspring distribution $\xi$ was chosen to be of the form~\eqref{eq:xi} with $\alpha = 3/2$, $f(n)$ constant (except for  $f(0)$), and $\Ex{\xi} = 3/4$.
	}
	\label{fi:pic1}
\end{figure} 

The tree in Figure~\ref{fi:pic1} exhibits a unique vertex whose degree has order $(1-\Ex{\xi})n$, which is typical for the regime~\cite{MR2764126,MR2908619,MR3335012,MR3227065} of Theorem~\ref{te:main1}. A  similar condensation phenomenon has recently been shown to occur when $\xi$ is critical and lies in the domain of attraction of a Cauchy law~\cite{2018arXiv180410183K}, see Figure~\ref{fi:pic2} for an illustration. There the order of the maximum degree is $o(n)$, but varies regularly with index $1$, and is much larger than the second largest degree.

\begin{figure}[H]
	\centering
	\includegraphics[width=0.75\textwidth]{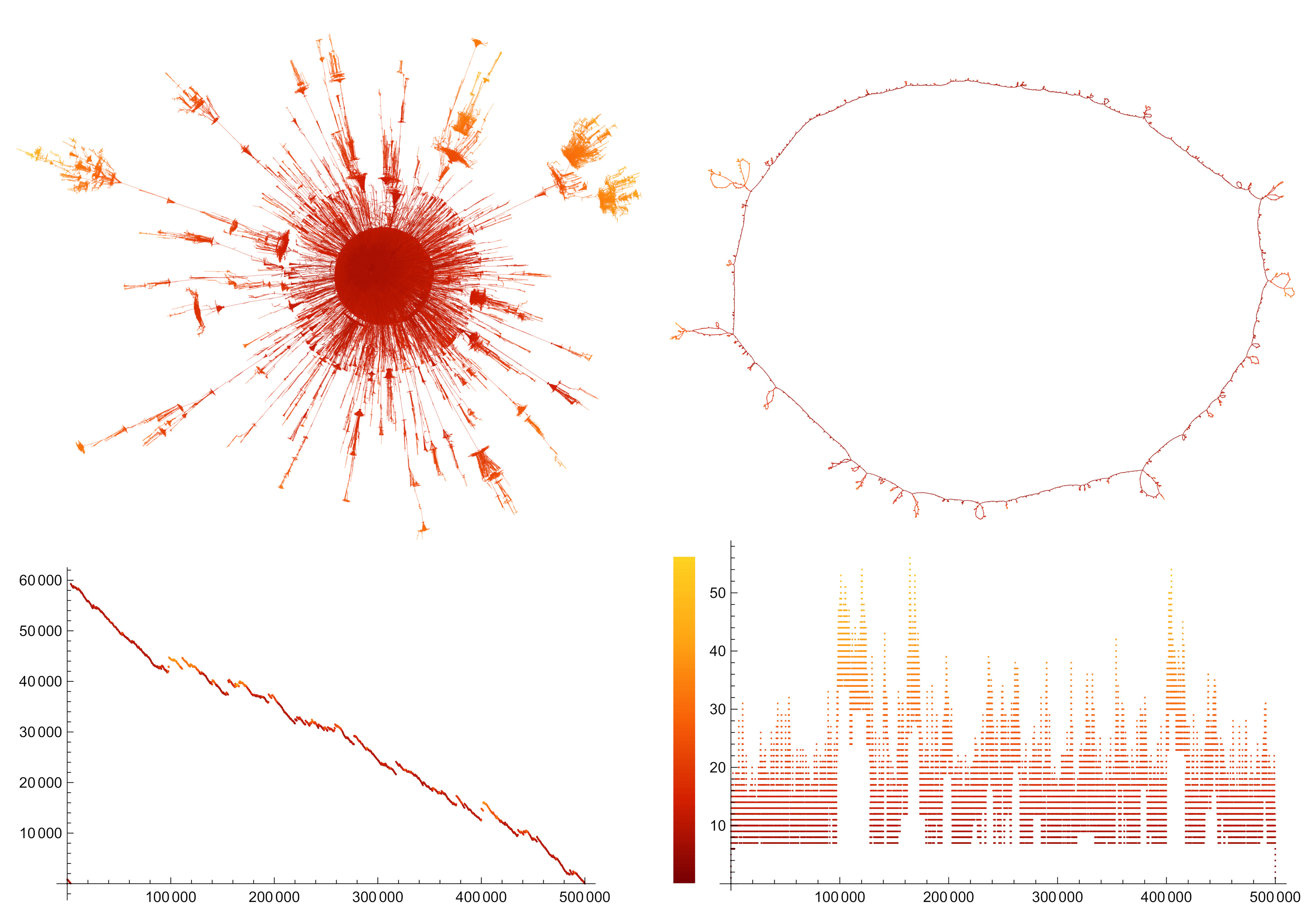}
	\caption{
		A critical Galton--Watson tree with 500k vertices. The offspring distribution was chosen to lie in the domain of attraction of the Cauchy law.
	}
	\label{fi:pic2}
\end{figure}

In the so called stable regime illustrated in Figure~\ref{fi:pic3}, $\xi$ is critical and lies in the domain of attraction of an $\alpha$-stable law for $1 < \alpha < 2$. There for each fixed $i\ge 1$ the order of the $i$th largest degree varies regularly with index $1/\alpha$~\cite{MR2908619,MR1964956,MR3286462}. Finally, the regime where $\xi$ is critical and lies in the domain of attraction of the normal law~\cite[Sec. 19]{MR2908619} is illustrated in Figure~\ref{fi:pic4}.

\begin{figure}[H]
	\centering
	\includegraphics[width=0.75\textwidth]{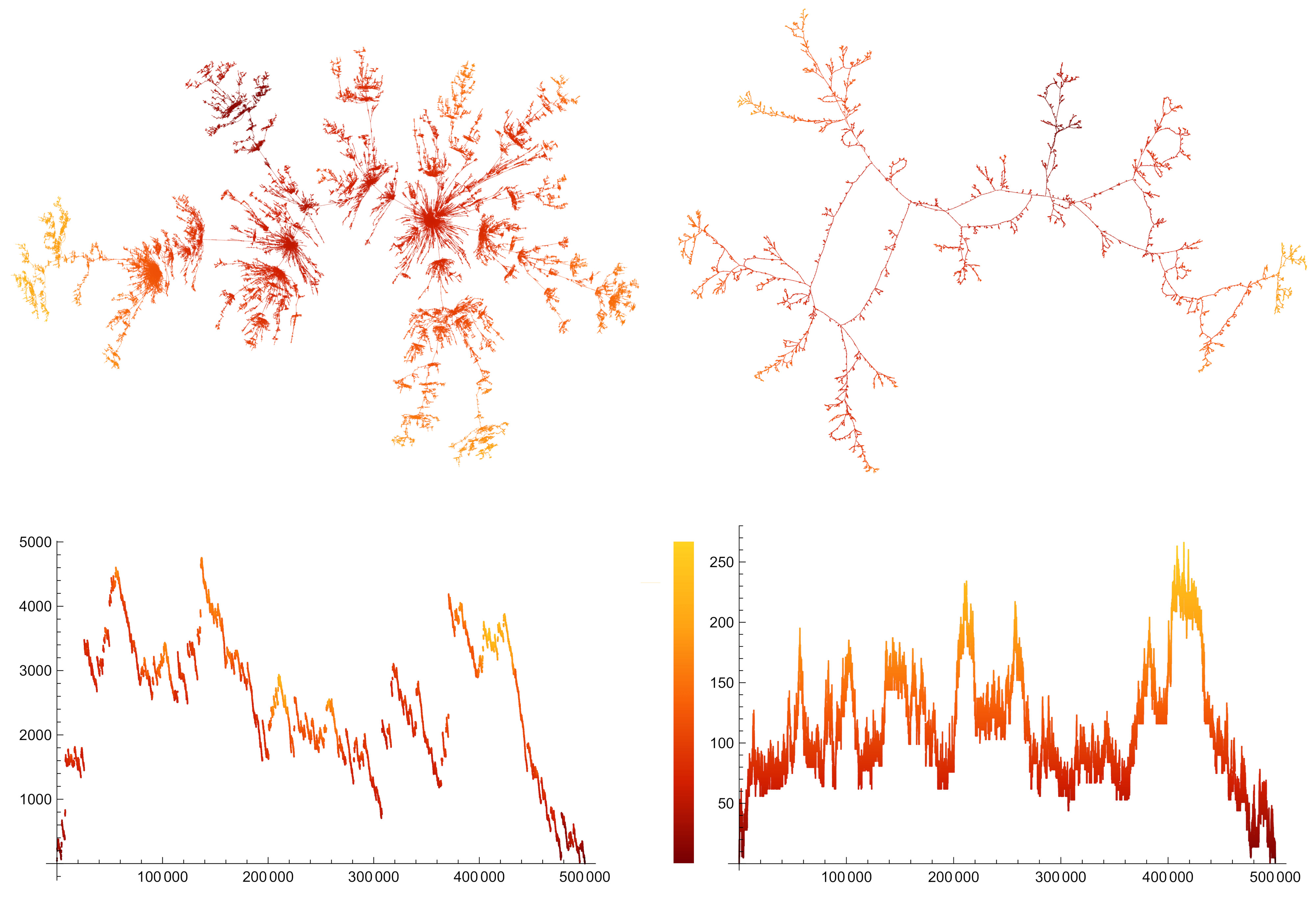}
	\caption{
		A critical Galton--Watson tree with 500k vertices. The offspring distribution was chosen to lie in the domain of attraction of the Airy law.
	}
	\label{fi:pic3}
\end{figure}  

\begin{figure}[H]
	\centering
	\includegraphics[width=0.75\textwidth]{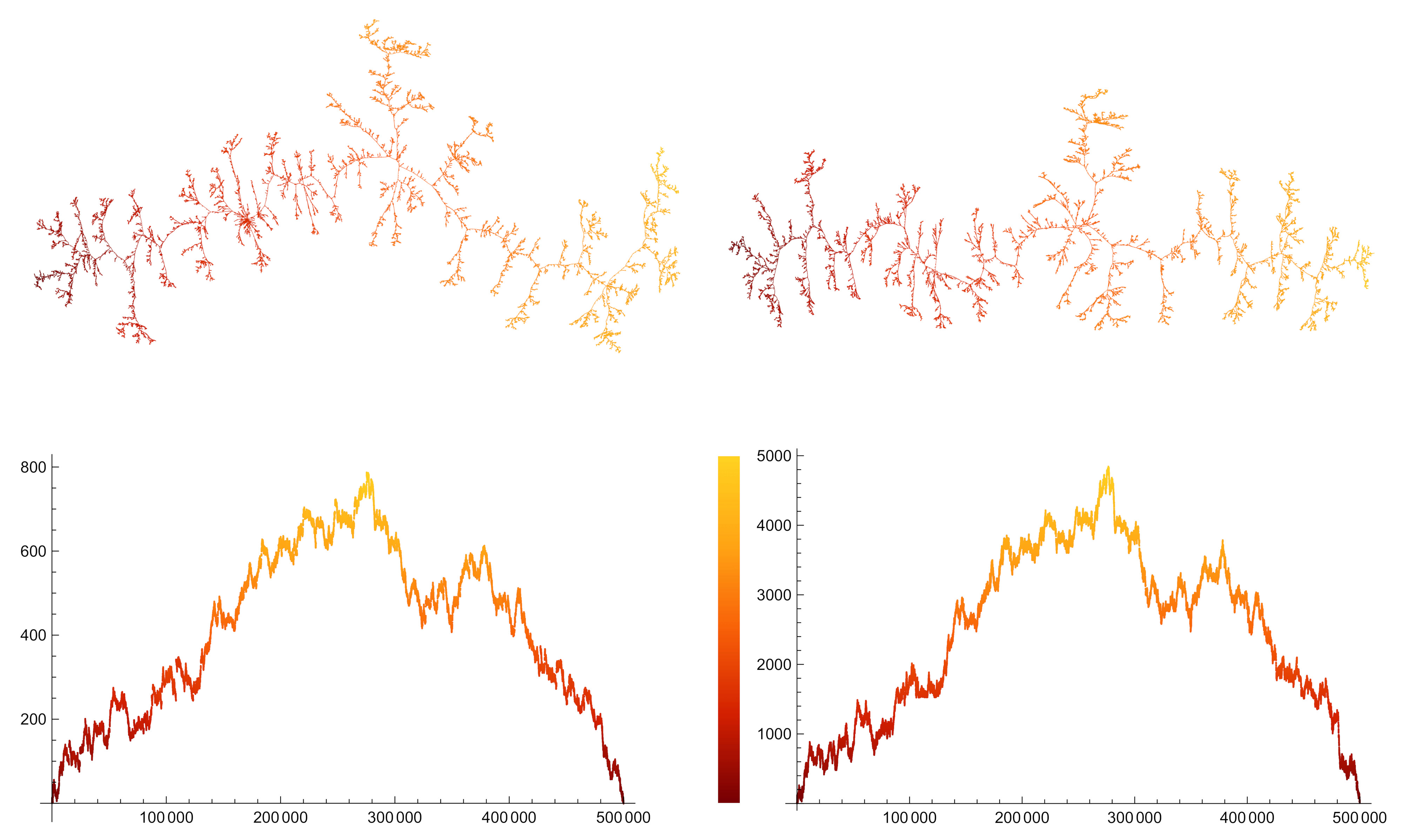}
	\caption{
		A critical Galton--Watson tree with 500k vertices. The offspring distribution was chosen to lie in the domain of attraction of the normal law.
	}
	\label{fi:pic4}
\end{figure}

\newpage
\section{Conditioning on the number of vertices}
\label{sec:cond1}

We start by establishing the limit theorems for the special case $\Omega=\ndN_0$ using results by Denisov, Dieker, and Shneer~\cite{MR2440928} and Armend{\'a}riz and Loulakis~\cite{MR2775110} on the big-jump domain for random walks.

\subsection{Plane trees correspond to cyclic shifts of balls-in-boxes configurations}
\label{sec:boxes}

A (planted) plane tree $T$ is a rooted unlabelled tree where each offspring set is endowed with a linear order. The outdegree of a vertex $v \in T$, denoted by $d_T^+(v)$, is its number of children. We let $\Delta(T)$ denote the maximal outdegree of $T$. The total number of vertices of $T$ is denoted by $|T|$.  Setting $n = |T|$, the tree $T$ is fully determined by the vector 
\[
(x_1, \ldots, x_n) = (d^+_T(v_1)-1, \ldots, d^+_T(v_n)-1),
\] with  $v_1, \ldots, v_n$ denoting the depth-first-search ordered list of vertices of $T$. The vector $(x_i)_{1 \le i \le n}$ satisfies $\sum_{i=1}^n x_i = -1$ and $\sum_{i=1}^k x_i \ge 0$ for all $ k < n$. The following result is classical:
\begin{lemma}[\cite{MR0138139}]
	\label{le:cyc}
	For any $r \ge 1$ and any vector $\mathbf{y} = (y_i)_{1 \le i \le n}$ of integers $y_i \ge -1$ with $\sum_{i=1}^n y_i = -r$ there exist precisely $r$ indices $i_0$ with the property that the cyclically shifted vector \[(\bar{y}_1, \ldots, \bar{y}_n) = (y_{i_0}, y_{i_0 +1}, \ldots, y_n, y_1, \ldots, y_{i_0-1})\] satisfies $\sum_{i=1}^{k} \bar{y}_i > -r$ for all $k < n$.
\end{lemma}
Hence for $r=1$ such a vector  $\mathbf{y}$ corresponds to a unique tree $T(\mathbf{y})$. The index $i_0$ is obtained by letting $k_0$ denote the smallest integer between $1$ and $n$ for which 
\[
\sum_{i=1}^{k_0} y_i = \min_{1 \le k \le n} \sum_{i=1}^{k} y_i,
\] and setting $i_0 = 1$ if $k_0 = n$, and $i_0 = k_0+1$ otherwise. 

\subsection{Non-generic simply generated trees and the big-jump domain}
If $(\xi_i)_{ i \ge 1}$ denote independent copies of $\xi$, then 
\begin{align}
\label{eq:teq}
\mT_n \eqdist T\left( (\xi_i-1)_{1 \le i \le n} \mid \xi_1 + \ldots + \xi_n = n-1 \right).
\end{align}
Suppose that $\Ex{\xi} < 1$ and \eqref{eq:xi} holds.  By results for the big-jump domain in random walk~\cite[Cor. 2.1]{MR2440928} it follows that
\begin{align}
\label{eq:gwt}
\Pr{|\mT|=n} = n^{-1} \Prb{\sum_{i=1}^{n} \xi_i = n-1} \sim \Pr{\xi = \lfloor(n-1)(1- \Ex{\xi})\rfloor} \sim \frac{f(n)}{ (n(1-\Ex{\xi}))^{1+\alpha}}.
\end{align}

Let $v_1, \ldots, v_n$ denote the depth-first-search ordered list of vertices of $\mT_n$, and set $d_i = d_{\mT_n}^+(v_i)-1$. (The depth-first-search order is often also referred to as the lexicographic order due to the usual embedding of plane trees as subtrees of the Ulam--Harris tree.) Let $1 \le j_0 \le n$ denote the smallest index such that the maximum outdegree of $\mT_n$ is attained at the corresponding vertex. It was observed in~\cite[Cor. 2.7]{MR3335012} using results from \cite{MR2775110} (compare with \cite[Thm. 19.34, (iii)]{MR2908619}) that
\begin{align}
\label{eq:pre1}
\lim_{n \to \infty} \sup_{A \subset \cB(\ndR^n)}\left| \Pr{ (d_{j_0}, \ldots, d_n, d_1, \ldots, d_{j_0 -1}) \in A} - \Pr{\mathbf{v}_n \in A} \right| = 0
\end{align}
for the vector
\begin{align*}
\mathbf{v}_n = \left(n - \xi_1 - \ldots - \xi_{n-1}, \xi_1, \ldots, \xi_{n-1}\right) - (1, \ldots, 1).
\end{align*}
Note that $\mathbf{v}_n$ does not have to correspond to a tree, since the first coordinate may be smaller than $-1$. In this case, we set $T(\mathbf{v}_n) = \diamond$ for some symbol $\diamond$ that is not contained in any other set under consideration in this paper. The probability for this event tends to zero as $n$ becomes large. Equation~\eqref{eq:pre1} implies that
\begin{align}
\label{eq:pre2}
\lim_{n \to \infty} \sup_{A \subset \cT_n \cup \{\diamond\}}\left| \Pr{\mT_n \in A} - \Pr{T(\mathbf{v}_n) \in A} \right| = 0
\end{align}
with $\cT_n$ denoting the finite set of all plane trees with $n$ vertices.

\subsection{Limits for the extremal degree of $\mT_n$}

Recall that $(X_t)_{t \ge 0}$ denotes the spectrally positive L\'evy process with Laplace exponent $\Ex{\exp(-\lambda X_t)} = \exp(t\lambda^\theta )$ for $\theta := \min(2,\alpha)$.  It is known that  the density $h$ of $X_1$ is positive,  uniformly continuous, and bounded on $\ndR$ (see ~\cite[Sec. XVII.6]{MR0270403} and \cite{MR1406564}). The classical local limit theorem~\cite[Thm. 4.2.1]{MR0322926} states that if $\xi$ lies in the domain of attraction of a $\theta$-stable law, then there is a slowly varying function $g$ such that the sums $S_n = \xi_1 + \ldots + \xi_n$ satisfy 
\begin{align}
\label{eq:llt}
\lim_{n \to \infty} \sup_{\ell \in \ndZ} \left|g(n) n^{1/\theta} \Pr{S_n = \ell} - h\left( \frac{\ell - n\Ex{\xi}}{g(n) n^{1/\theta}} \right) \right| = 0.
\end{align}
It was shown in~\cite[Thm. 1.10]{MR2946438} that the function $g$ may be chosen to satisfy Equation~\eqref{eq:defofg}. If assumption~\eqref{eq:xi} is satisfied, then  Equation~\eqref{eq:pre2} implies
\begin{align}
\label{eq:mdeg}
\Delta(\mT_n) \atv n - \xi_1 - \ldots - \xi_{n-1}\end{align}
This may be used (see \cite[Thm. 1]{MR3335012}) to deduce a central limit theorem
\begin{align}
\label{eq:mclt}
\frac{(1 - \Ex{\xi})n - \Delta(\mT_n)}{g(n)n^{1/\theta}} \convdis X_1.
\end{align}
Compare also with \cite[Thm. 19.34]{MR2908619}. We may strengthen \eqref{eq:mclt} to  a local limit theorem. This does not follow directly from~\eqref{eq:mdeg}, as we would require knowledge on the speed with which the total variation distance tends to zero. 
\begin{lemma}
	\label{le:maxllt}
	It holds that
	\[
	\Pr{\Delta(\mT_n) = \ell} = \frac{1}{g(n)n^{1/\theta}}\left(h\left(\frac{ (1-\Ex{\xi})n - \ell}{g(n)n^{1/\theta}}   \right) + o(1)\right)
	\]
	uniformly for all integers $\ell$.
\end{lemma}

\begin{proof}
	By Equation~\eqref{eq:teq} we know that $\Delta(\mT_n)$ is distributed like the maximum jump of the random walk $S_n$ conditioned to arrive at $n-1$. Hence
	\begin{align}
	\label{eq:mmd}
	\Pr{\Delta(\mT_n) = \ell} = \frac{\Pr{\max(\xi_1, \ldots, \xi_n) = \ell, S_n = n-1}}{\Pr{S_n = n-1 }}.
	\end{align}
	By \cite[Cor. 2.1]{MR2440928} it holds that
	\begin{align}
	\label{eq:nnd}
	\Pr{S_n = n-1 } \sim n \Pr{\xi =  \lfloor n(1-\Ex{\xi})\rfloor} \sim f(n) n^{-\alpha} (1 - \Ex{\xi})^{-\alpha-1}
	\end{align}
	It follows from Equations~\eqref{eq:mmd}, \eqref{eq:nnd} and the exponential bounds \cite[Lem. 2.1]{MR2440928} (applied to the centred random walk $S_n - n \Ex{\xi}$) that there is a constant $C>0$ such that
	\begin{align}
	\label{eq:lowpass}
	\Pr{\Delta(\mT_n) \le cg(n) n^{1/\theta}} \le C \exp\left( \frac{-n (1-\Ex{\xi})}{cg(n) n^{1/\theta}}\right) n^{\alpha} / f(n)
	\end{align}
	for all $c \ge 1$. Hence there is a constant $\epsilon_1>0$ such that it suffices to verify that Lemma~\ref{le:maxllt} holds uniformly for all $\ell \ge \epsilon_1 n/\log n$.
	
	Throughout the following we only consider values $\ell$ with $ \epsilon_1 n/\log n \le \ell \le n$. By Equations~\eqref{eq:mmd} and \eqref{eq:nnd} it follows that $g(n)n^{1/\theta} \Pr{\Delta(\mT_n) = \ell}$ equals
	\begin{multline}
	\label{eq:theexpr44}
	(1 + o(1))\frac{g(n)n^{1/\theta}}{ (1 - \Ex{\xi})^{-\alpha-1}f(n) n^{-\alpha}} \\\sum_{k \ge 1} \binom{n}{k} \Pr{\xi=\ell}^k \Prb{\max_{1 \le i \le n-k}\xi_i < \ell, S_{n-k} = n-1-k\ell}.
	\end{multline}
	Our next step is to discard all summands except for the first. Note that $S_{n-k} \ge 0$ implies that all summands with $k>n/\ell$ are equal to zero. Hence
	\begin{multline}
	\label{eq:boundme}
	\frac{g(n)n^{1/\theta}}{ f(n) n^{-\alpha}} \sum_{k \ge 2} \binom{n}{k} \Pr{\xi=\ell}^k \Prb{\max_{1 \le i \le n-k}\xi_i < \ell, S_{n-k} = n-1-k\ell} \\
	\le \frac{f(\ell)}{f(n)} \left(\frac{\ell}{n}\right)^{-\alpha-1} \sum_{2 \le k \le n/\ell} (n \Pr{\xi=\ell})^{k-1} \Pr{ S_{n-k} = n-1-k\ell} g(n) n^{1/\theta}.
	\end{multline}
	Note that $n/\ell \le \epsilon_1^{-1} \log n$ and hence $n \sim n-k$ uniformly for all summands. Since $h$ is bounded, it follows from the local limit theorem~\eqref{eq:llt} that $\Pr{ S_{n-k} = n-1-k\ell} g(n) n^{1/\theta}$ remains bounded uniformly for all $2 \le k \le n/\ell$ and $\ell \ge \epsilon_1 n/\log n$. 
	Hence the expression in~\eqref{eq:boundme} admits an upper bound of the form
	\begin{align}
	O(1)\frac{f(\ell)}{f(n)} (\log n)^{\alpha+1} \sum_{2 \le k \le n/\ell} (n \Pr{\xi=\ell})^{k-1}.
	\end{align}
	Moreover,  $n \Pr{\xi=\ell} = O(f(\ell)n^{-\alpha} (\log n)^{1 + \alpha})$ holds uniformly as well. Hence, using the Potter bounds,  the expression in~\eqref{eq:boundme} may be further bounded by
	\begin{align}
	\label{eq:Abound}
	\frac{f(\ell)}{f(n)}  (\log n)^{\alpha+1} O(f(\ell) n^{-\alpha} (\log n)^{1 + \alpha}) = o(1).
	\end{align}
	This verifies that $g(n)n^{1/\theta} \Pr{\Delta(\mT_n) = \ell}$ equals
	\begin{align}
	\label{eq:nstep}
	o(1) + (1+ o(1)) \frac{g(n)n^{1/\theta}}{ (1 - \Ex{\xi})^{-\alpha-1}f(n) n^{-\alpha}} n \Pr{\xi=\ell} \Prb{\max_{1 \le i \le n-1}\xi_i < \ell, S_{n-1} = n-1-\ell}
	\end{align}
	uniformly for all $\ell$ with $ \epsilon_1 n/\log n \le \ell \le n$.
	
	Let $0<\epsilon<1 - \Ex{\xi}$ be some constant. By \cite[Cor. 2.1]{MR2440928} (applied to the centred sum $S_n - n \Ex{\xi}$) it holds uniformly for all integers $\ell$ with $\epsilon_1 n/\log n \le \ell \le \epsilon n$ that
	\begin{align}
	\Pr{S_{n-1} = n-1-\ell} &\sim n \Pr{\xi = (n-1)(1- \Ex{\xi}) - \ell} \\&= O(n^{-\alpha}f(n)). \nonumber
	\end{align}
	Since $\alpha>1$ implies that $1/\theta < \alpha$, it follows that the expression in \eqref{eq:nstep} tends to zero uniformly for all $\ell$ in the restricted range. Thus it suffices to verify that Lemma~\ref{le:maxllt} holds uniformly for all $\ell$ with $\epsilon n \le \ell \le n$.
	
	Let $\ell \in [\epsilon n, n]$ be given. Our next step will be to get rid of the event $\max_{1 \le i \le n-1}\xi_i < \ell$ in the expression~\eqref{eq:nstep}. To this end, note that $\sup_{k \ge n} \Pr{\xi=k} = O(\Pr{\xi=n})$ implies that
	\begin{align}
	\Prb{\max_{1 \le i \le n-1}\xi_i \ge \ell, S_{n-1} = n-1-\ell} &\le n \sum_{i\ge \ell}\Pr{\xi=i} \Pr{S_{n-1}=n-1-\ell-i} 
	\\&\le O(n) \Pr{\xi=\ell}.\nonumber
	\end{align}
	Also,
	\begin{align}
	\label{eq:expre}
	\frac{n \Pr{\xi=\ell}}{(1 - \Ex{\xi})^{-\alpha-1}f(n) n^{-\alpha}} \sim \left( \frac{(1 - \Ex{\xi})n}{\ell} \right)^{1 + \alpha}
	\end{align}
	remains bounded for $\ell \ge \epsilon n$. Using again $1/\theta < \alpha$, this implies that the result of substituting $\max_{1 \le i \le n-1}\xi_i < \ell$ by $\max_{1 \le i \le n-1}\xi_i \ge \ell$ in expression~\eqref{eq:nstep} tends to zero. This shows that
	\begin{align}
	\label{eq:newstep}
	g(n)n^{1/\theta} \Pr{\Delta(\mT_n) = \ell} = o(1) + 
	\left( \frac{(1 - \Ex{\xi})n}{(1+o(1))\ell} \right)^{1 + \alpha}\Prb{ S_{n-1} = n-1-\ell}g(n)n^{1/\theta}
	\end{align}
	holds uniformly.
	
	The local limit theorem~\eqref{eq:llt} tells us that
	\begin{align}
	\Prb{ S_{n-1} = n-1-\ell}g(n)n^{1/\theta} = h\left( \frac{\ell - n(1 - \Ex{\xi})}{g(n) n^{1/\theta}} \right) + o(1).
	\end{align}
	Using that the function $h$ and the expression in~\eqref{eq:expre} are bounded, it follows from~\eqref{eq:newstep} that
	\begin{align}
	\label{eq:lstep}
	g(n)n^{1/\theta} \Pr{\Delta(\mT_n) = \ell} = o(1) + 
	\left( \frac{(1 - \Ex{\xi})n}{\ell} \right)^{1 + \alpha}h\left( \frac{\ell - n(1 - \Ex{\xi})}{g(n) n^{1/\theta}} \right).
	\end{align}
	Let $\epsilon_2 >0$ be small enough such that $1/\theta + \epsilon_2 < 1$. It holds that
	\begin{align}
	\sup_{\ell' \notin n(1 - \Ex{\xi}) \pm n^{1/\theta + \epsilon_2}, \ell' \ge \epsilon n} h\left( \frac{\ell' - n(1 - \Ex{\xi})}{g(n) n^{1/\theta}} \right) \to 0.
	\end{align}
	Consequently, it remains to verify that Lemma~\ref{le:maxllt} holds uniformly for $\ell \in n(1 - \Ex{\xi}) \pm n^{1/\theta + \epsilon_2}$. For $\ell$ ranging in this interval, Equation~\eqref{eq:lstep} yields
	\begin{align}
	g(n)n^{1/\theta} \Pr{\Delta(\mT_n) = \ell} = o(1) + 
	h\left( \frac{\ell - n(1 - \Ex{\xi})}{g(n) n^{1/\theta}} \right).
	\end{align}
	This completes the proof.
\end{proof}

For future use we remark on some deviation bounds:
\begin{proposition}
	\label{pro:largedev}
	\begin{enumerate}
		\item For any  $s>0$ we may select  $\epsilon>0$ small enough such that
		\begin{align}
		\label{eq:deva}
		\Prb{\Delta(\mT_n) \le \epsilon \frac{n}{\log n}} =  O(n^{-s}).
		\end{align}
		\item We may select $\epsilon>0$ small enough so that there exists a $\delta>0$ with
		\begin{align}
		\label{eq:devb}
		g(n)n^{1/\theta} \Prb{\Delta(\mT_n) \le \epsilon  n } = O(n^{-\delta})
		\end{align}
		as $n \to \infty$.
	\end{enumerate}
\end{proposition}

\begin{proof}[Proof of Proposition~\ref{pro:largedev}]
	Inequality~\eqref{eq:deva} follows directly from~\eqref{eq:lowpass}.  In order to check~\eqref{eq:devb}, it suffices to shows such a bound for~$\Prb{\epsilon_1 \frac{n}{\log n} \le \Delta(\mT_n) \le \epsilon n}$ for some  $\epsilon_1>0$ chosen sufficiently small according so that~\eqref{eq:deva} holds for $s=\alpha - 1/\theta$.
	
	The expression of $g(n)n^{1/\theta} \Pr{\Delta(\mT_n) = \ell}$ in~\eqref{eq:theexpr44} for $\epsilon_1 \frac{n}{\log n} \le \ell \le \epsilon n$ entails that
	\begin{align}
	g(n)n^{1/\theta} \Pr{\Delta(\mT_n) = \ell} \le A + B
	\end{align}
	with $A$ denoting the bound of~\eqref{eq:Abound}
	\begin{align}
	A  = O(n^{-\alpha + o(1)}),
	\end{align}
	and $B$ denoting the bound from~\eqref{eq:nstep}
	\begin{align}
	B &= O(n^{1/\theta  + \alpha + 1 + o(1)})  \Pr{\xi=\ell} \Prb{\max_{1 \le i \le n-1} \xi_i < \ell, S_{n-1} = n-1-\ell} \\
	&= O(n^{1/\theta + o(1)}) \Prb{\max_{1 \le i \le n-1} \xi_i < \ell, S_{n-1} = n-1-\ell}. \nonumber
	\end{align}
	It follows from the Inequality~\cite[(19.129)]{MR2908619} (generalized to admit regularly varying densities instead of asymptotic power laws) that for $\epsilon < \frac{\alpha-1}{\alpha+1}(1 - \Ex{\xi})$
	\begin{align}
	&\Prb{\max_{1 \le i \le n-1} \xi_i < \ell, S_{n-1} = n-1-\ell} \nonumber 
	\\&\qquad = \exp\left(-(\alpha+1) \frac{\log n}{n(1- \Ex{\xi})}( n-1-\ell - n \Ex{\xi} + o(n)) \right) 
	\\&\qquad\le n^{-(\alpha + 1)(1-\epsilon + o(1))}.\nonumber
	\end{align}
	Note that $\alpha>1$ implies that $1/\theta < \alpha$. Taking $\epsilon>0$ sufficiently small and summing over all integers $\ell$ with $\epsilon_1 \frac{n}{\log n} \le \ell \le \epsilon n$ it follows that
	\begin{align}
	\Prb{\epsilon_1 \frac{n}{\log n} \le \Delta(\mT_n) \le \epsilon n} &\le n \left( O(n^{-\alpha + o(1)}) + O(n^{1/\theta -(\alpha + 1)(1-\epsilon) + o(1)}) \right) \\
	& =O(n^{-\delta}) \nonumber
	\end{align}
	for some $\delta>0$.
\end{proof}

\subsection{The asymptotic shape of the random tree $\mT_n$} 

\begin{lemma}
	\label{le:limit}
	\begin{enumerate}
		\item 	Let $(t_n)_{n \ge 1}$ denote a sequence of integers with $t_n \to \infty$ and $t_n = o(n)$. The asymptotic equivalence
		\begin{align}
		\label{eq:lltoshow}
		\Prb{ F_0(T(\mathbf{v}_n)) = T^\bullet, (F_i(T(\mathbf{v}_n)))_{1 \le i \le k} = (T^i)_{1 \le i \le k} } \sim \Pr{\mT^\bullet = T^\bullet} \prod_{i=1}^k \Pr{\mT = T^i} 
		\end{align}	
		holds uniformly for all $k \ge 1$, all marked plane trees $T^\bullet \in \cT^\bullet$ and all ordered forests $(T^i)_{1 \le i \le k}$ of plane trees with a total number of vertices \begin{align}\label{eq:as0}
		|T^\bullet| + \sum_{i=1}^k |T^i| \le n - t_n.
		\end{align}
		\item 
		For any sequence of integers $(t_n)_{n \ge 1}$  with $t_n \to \infty$ and $t_n = o(n)$ 
		\begin{align}
		\label{eq:first}
		\left(F_0(\mT_n), (F_i(\mT_n))_{1 \le i \le \Delta(\mT_n) - t_n}  \right) \atv \left(\mT^\bullet, (\mT^i)_{1 \le i \le \Delta_{\langle n \rangle}- t_n}\right)
		\end{align}
		with
		\begin{align}
		\Delta_{\langle n \rangle} := \sup\left\{ d \ge 1 \,\,\bigg\rvert\,\, |\mT^\bullet| +  \sum_{i=1}^d |\mT^i| \le n\right\} < n.
		\end{align}
		The remaining $t_n$ fringe subtrees satisfy with high probability 
		\begin{align}
		\label{eq:second}
		\sum_{i = \Delta(\mT_n) - t_n }^{\Delta(\mT_n)} |F_i(\mT_n)| < \frac{2t_n}{1 - \Ex{\xi}}.
		\end{align}
	\end{enumerate}
\end{lemma}

\begin{proof}[Proof of Lemma~\ref{le:limit}]	 
	We start by verifying the equivalence~\eqref{eq:lltoshow}. Recall that in Section~\ref{sec:boxes} we discussed how a plane tree with $m \ge 1$ vertices corresponds to a sequence  $(x_i)_{1 \le i \le m}$ with $\sum_{i=1}^m x_i = -1$ and $\sum_{i=1}^\ell x_i \ge 0$ for all $\ell < m$. An ordered forest of plane trees corresponds to concatenations of such sequences. There is a unique way to cut $(\xi_1-1,  \ldots, \xi_{n-1}-1)$ into initial segments $\mathbf{x}_1, \ldots, \mathbf{x}_r$, each corresponding to a tree, and a single tail segment $\mathbf{y} = (y_i)_{1 \le i \le d}$ with $\sum_{i=1}^j y_i \ge 0$ for all $1 \le j \le d$. (For example, $\mathbf{x}_1$ corresponds to the tree $F_1(T(\mathbf{v}_n))$.) The segment $\mathbf{y}$ corresponds to the initial segment of the depth-first-search ordered list of vertex outdegrees of $F_0(T(\mathbf{v}_n))$ obtained by stopping right before visiting the lexicographically first vertex with maximum outdegree. Hence it encodes the outdegrees of the spine vertices (except for the marked vertex), and all vertices that lie to the left of the spine. It also encodes the precise location of the marked vertex. The sum $R:= \sum_{i=1}^d y_i$ tells us the quantity of direct offspring of spine vertices (except for the marked vertex) of $F_0(T(\mathbf{v}_n))$ that lie to the right of the spine. Hence $\mathbf{x}_1, \ldots, \mathbf{x}_{r-R}$ correspond to the fringe-subtrees dangling from the marked vertex in $T(\mathbf{v}_n)$, and $\mathbf{x}_{r-R+1}, \ldots, \mathbf{x}_r$ correspond to the fringe subtrees dangling from spine vertices (except for the marked vertex) to the right of the spine. Compare with the example in Figure~\ref{fi:example}.

	\begin{figure}[h]
		\centering
		\centering
		\includegraphics[width=1.0\textwidth]{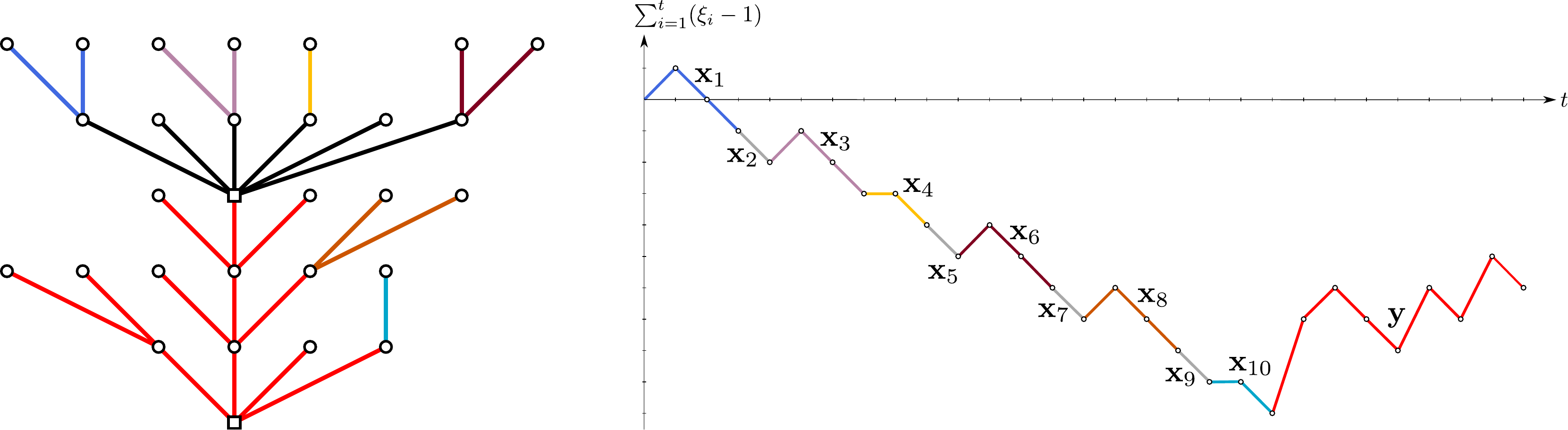}
		\caption{Decomposition of $(\xi_1-1, \ldots, \xi_{n-1} - 1)$.}
		\label{fi:example}
	\end{figure}

	So in order for the tail segment $\mathbf{x}_{r-R+1}, \ldots, \mathbf{x}_r, \mathbf{y}$ to encode the tree $T^\bullet$ it must hold that the concatenation of $\mathbf{y},-1,\mathbf{x}_r, \mathbf{x}_{r-1}, \ldots, \mathbf{x}_{r-R+1}$ is equal to the depth-first-search ordered list of outdegrees of $T^\bullet$. In order for $\mathbf{x}_1, \ldots, \mathbf{x}_k$ to encode $(T^i)_{1 \le i \le k}$ it must hold that $\mathbf{x}_i$ encodes $T^i$ for all $1 \le i \le k$. The only requirement for the middle segment $(\xi_{1+\sum_{i=1}^k |T^i| }, \ldots, \xi_{n - |T^\bullet|})$ is that it must correspond to a forest. Note that the middle segment has $s_n := n - |T^\bullet| - \sum_{i=1}^k |T^i|$ list entries and $s_n \ge t_n$ by assumption.
	Using Equation~\eqref{eq:tbullet} it follows that the probability
	\begin{multline}
	\label{eq:bigp}
	\Prb{ F_0(T(\mathbf{v}_n)) = T^\bullet, F_i(T(\mathbf{v}_n)) = T^i\text{ for } 1 \le i \le k} \\ =
	(1 - \Ex{\xi})^{-1} \Pr{\mT^\bullet = T^\bullet} \left(\prod_{i=1}^k \Pr{\mT = T^i}\right) \\\Pr{ (\xi_1-1, \ldots, \xi_{s_n}-1) \text{ corresponds to a forest} }.
	\end{multline}
	For any integer $s \ge 1$ the probability $\Pr{ (\xi_1-1, \ldots, \xi_{s}-1) \text{ corresponds to a forest} }$ is equal to the probability that $s$ is an arrival time for i.i.d. interarrival times distributed like $|\mT|$. As $\Ex{|\mT|} = (1 - \Ex{\xi})^{-1}$, it follows by the renewal theorem\footnote{The author thanks an anonymous referee for pointing out this elegant application of the renewal theorem to obtain Equation~\eqref{eq:renewal}.} that
	\begin{align}
	\label{eq:renewal}
	\lim_{s \to \infty} \Pr{ (\xi_1-1, \ldots, \xi_{s}-1) \text{ corresponds to a forest} } = 1 - \Ex{\xi}.
	\end{align}
	This completes the verification of~\eqref{eq:lltoshow}.

	Our next  step is to verify~\eqref{eq:first}.
	Recall that by Equation~\eqref{eq:pre2} the random tree $\mT_n$ has a total variational distance from the tree $T(\mathbf{v}_n)$ that tends to zero as $n$ becomes large. This allows us to work with  $T(\mathbf{v}_n)$ instead of~$\mT_n$.
	
	It follows from~\eqref{eq:gwt} and the classical local limit theorem that there is a slowly varying function $g_{\mT}$ such that
	\begin{align}
	\label{eq:llt42}
	\lim_{n \to \infty} \sup_{\ell \in \ndZ} \left|g_{\mT}(n) n^{1/\theta} \Prb{\sum_{i=1}^n|\mT^i| = \ell} - h\left( \frac{\ell - n\Ex{\xi}}{g_{\mT}(n) n^{1/\theta}} \right) \right| = 0.
	\end{align}
	As shown by Kortchemski~\cite[Lem. 2.10]{MR3335012}, the function $g_{\mT}$ may be chosen to satisfy for all $n \ge 1$
	\begin{align}
	\label{eq:gtilde}
	g_{\mT}(n) = \frac{g(n)}{(1 - \Ex{\xi})^{1 + 1/\theta}}.
	\end{align}
	
	Suppose that $(r_n)_{n \ge 1}$ is a sequence of positive integers satisfying  $r_n / g(n)n^{1/\theta} \to \infty$ and $r_n = o(n)$. As $\Delta(\mT_n) = (1- \Ex{\xi})n + O_p(g(n)n^{1/\theta})$, it follows that
	there is a sequence of integers $(t_n')_{n \ge 1}$  satisfying $t_n' \to \infty$ and $t_n' = o(n)$ such that $k_n := \lfloor (1-\Ex{\xi})n -r_n\rfloor$ satisfies
	\[
	|\mT^\bullet| + \sum_{i=1}^{k_n} |\mT^i| \le n - t_n'
	\]
	with probability tending to $1$ as $n\to \infty$.  Using Equation~\eqref{eq:lltoshow} (for $(t_n')_{n \ge 1}$), it follows that 	
	\begin{align}
	\label{eq:prel}
	\left(F_0(\mT_n), (F_i(\mT_n))_{1 \le i \le (1 - \Ex{\xi})n - r_n}  \right) \atv \left(\mT^\bullet, (\mT^i)_{1 \le i \le (1 - \Ex{\xi})n - r_n}\right).
	\end{align}


	Given $\epsilon>0$, it follows from Equation~\eqref{eq:prel} and a time-reversal argument that we may choose a constant $M>0$ large enough so that
	\begin{align}
	\label{eq:doh}
	\Prb{\left| \sum_{i = \Delta(T(\mathbf{v}_n)) - t_n }^{\Delta(T(\mathbf{v}_n))} |F_i(T(\mathbf{v}_n))| - \frac{t_n}{1 - \Ex{\xi}} \right| < M g(t_n) t_n^{1/\theta}} > 1 - \epsilon
	\end{align}
	for all large enough $n$. 
	
	Now, let $T^\bullet \in \cT^\bullet$ and let $(T^i)_{1 \le i \le k}$ be a sequence of plane trees such that $s_n = n - |T^\bullet| - \sum_{i=1}^k |T^i|$ satisfies
	\begin{align}
	\label{eq:assu1}
	\left|s_n -  \frac{t_n}{1 - \Ex{\xi}} \right| < M g(t_n) t_n^{1/\theta}.
	\end{align}
	Then
	\begin{multline}
	\label{eq:sp}
	\Prb{ F_0(T(\mathbf{v}_n)) = T^\bullet, (F_i(T(\mathbf{v}_n)))_{1 \le i \le \Delta(T(\mathbf{v}_n)) - t_n} 
		= (T^i)_{1 \le i \le k}} \\= \Prb{ F_0(T(\mathbf{v}_n)) = T^\bullet, (F_i(T(\mathbf{v}_n)))_{1 \le i \le k} = (T^i)_{1 \le i \le k}, \Delta(T(\mathbf{v}_n)) = k + t_n}.
	\end{multline}
	Arguing analogously as for Equation~\eqref{eq:bigp}, we obtain that the probability in ~\eqref{eq:sp} is given by
	\begin{multline}
	(1 - \Ex{\xi})^{-1} \Pr{\mT^\bullet = T^\bullet} \left(\prod_{i=1}^k \Pr{\mT = T^i}\right) \\\Pr{ (\xi_1-1, \ldots, \xi_{s_n}-1) \text{ corresponds to a forest with $t_n$ trees} }.
	\end{multline}
	Using Lemma~\ref{le:cyc}, Assumption~\eqref{eq:assu1}, and the local limit theorem~\eqref{eq:llt}, we obtain 
	\begin{align*}
	\Pr{ (\xi_1-1, \ldots, \xi_{s_n}-1) &\text{ corresponds to a forest with $t_n$ trees} } \\&= \Prb{ \sum_{1 \le i \le s_n} (\xi_i-1) =-t_n, \sum_{1 \le i \le j} (\xi_i-1) >-t_n \text{ for all $j < s_n$}} 
	\\&= \frac{t_n}{s_n} \Prb{ \sum_{1 \le i \le s_n} (\xi_i-1) = -t_n}
	\\&= \frac{1 - \Ex{\xi}}{g(s_n)s_n^{1/\theta}} \left( o(1) + h\left(\frac{t_n -s_n(1 - \Ex{\xi})}{g(s_n)s_n^{1/\theta}} \right)\right).
	\end{align*}
	The $o(1)$ term in this expression is uniform in $(T^\bullet, (T^i)_{1 \le i \le k})$.  Assumption~\eqref{eq:assu1} also entails that the value of the function $h$  in this expression is bounded away from zero. Summing up, we obtain
	\begin{multline}
	\label{eq:sp2}
	\Prb{ F_0(T(\mathbf{v}_n)) = T^\bullet, (F_i(T(\mathbf{v}_n)))_{1 \le i \le \Delta(T(\mathbf{v}_n)) - t_n} 
		= (T^i)_{1 \le i \le k}} \\= \Pr{\mT^\bullet = T^\bullet} \left(\prod_{i=1}^k \Pr{\mT = T^i}\right) \frac{1 }{g(s_n)s_n^{1/\theta}} \left( o(1) + h\left(\frac{t_n -s_n(1 - \Ex{\xi})}{g(s_n)s_n^{1/\theta}} \right)\right).
	\end{multline}
	On the other hand,
	\begin{multline}
	\Prb{\left(\mT^\bullet, (\mT^i)_{1 \le i \le \Delta_{\langle n \rangle}- t_n}\right) = (T^\bullet, (T^i)_{1 \le i \le k}) } 
	\\= \Pr{\mT^\bullet = T^\bullet} \left(\prod_{i=1}^k \Pr{\mT = T^i}\right) \Prb{ \sum_{i=1}^{t_n} |\mT^i| \le s_n <  \sum_{i=1}^{t_n+1} |\mT^i|}.
	\end{multline}
	Using~\eqref{eq:llt42},~\eqref{eq:gtilde},~\eqref{eq:assu1} and the fact that the function $h$ is uniformly continuous and bounded, we obtain
	\begin{align*}
	\Prb{ \sum_{i=1}^{t_n} |\mT^i| \le s_n <  \sum_{i=1}^{t_n+1} |\mT^i|} &= o\left(\frac{1}{g_{\mT}(t_n) t_n^{1/\theta}}\right) + \sum_{\ell = 0}^{\log t_n} \Prb{\sum_{i=1}^{t_n} |\mT^i| = s_n-\ell}\Pr{|\mT| > \ell } 
	\\&= \frac{1}{(1 - \Ex{\xi})g_{\mT}(t_n) t_n^{1/\theta}}\left( o(1) + h\left(\frac{t_n/(1- \Ex{\xi}) -s_n}{g_{\mT}(t_n)t_n^{1/\theta}} \right) \right) 
	\\&= \frac{1 }{g(s_n)s_n^{1/\theta}} \left( o(1) + h\left(\frac{t_n -s_n(1 - \Ex{\xi})}{g(s_n)s_n^{1/\theta}} \right)\right).
	\end{align*}
	Assumption~\eqref{eq:assu1} ensures that the $h$-term in this expression is bounded away from zero, uniformly  in $(T^\bullet, (T^i)_{1 \le i \le k})$. Summing up, it follows that
	\begin{multline}
	\Prb{ F_0(T(\mathbf{v}_n)) = T^\bullet, (F_i(T(\mathbf{v}_n)))_{1 \le i \le \Delta(T(\mathbf{v}_n)) - t_n} 	= (T^i)_{1 \le i \le k}} \\ = (1 + o(1)) 	\Prb{\left(\mT^\bullet, (\mT^i)_{1 \le i \le \Delta_{\langle n \rangle}- t_n}\right) = (T^\bullet, (T^i)_{1 \le i \le k}) },
	\end{multline}
	with a uniform $o(1)$ term. Since the constant $\epsilon$ in \eqref{eq:doh} (implicit in   Assumption~\eqref{eq:assu1}) may be chosen to be arbitrarily small, this verifies~\eqref{eq:first}.
	
	Finally, Equation~\eqref{eq:second} follows readily by Equation~\eqref{eq:prel}, a time-reversal argument, and Markov's inequality. 
\end{proof}

\section{The case $0 \in \Omega$}
\label{sec:cond}

We reduce the case of a general $\Omega$ containing $0$ to the special case $\Omega=\ndN_0$ via a combinatorial transformation. This construction  goes back to Ehrenborg and M\'endez~\cite{MR1284403} and is also known in the probabilistic literature, see Abraham and Delmas~\cite{MR3164755,MR3227065}, Minami~\cite{MR2135161}, and Rizzolo~\cite{MR3335013}. Further studies of related conditionings of Galton--Watson trees may be found in~\cite{MR2946438,MR2976559,MR3245291}.

Throughout this chapter we assume that $\Omega$ is a proper subset of $\ndN_0$ such that $0 \in \Omega$ and  either $\Omega$ or its complement $\Omega^{c} := \ndN_0 \setminus \Omega$ is finite. To each finite (planted) plane tree $T$ we may assign its weight $\omega(T) = \Pr{\mT = T}$. We let $L_\Omega(T)$ denote the number of vertices in $T$ whose outdegree lies in $\Omega$. The generating function \begin{align}\label{eq:ca}A(z) = \sum_T \omega(T) z^{L_\Omega(T)}\end{align} with the index $T$ ranging over all finite plane trees represents the combinatorial class $A$ of plane trees weighted by $\omega$ and indexed according to the number of vertices with outdegree in $\Omega$. \footnote{A (weighted) \emph{combinatorial class} consists of a countable set $S$ and a weight function $\gamma: S \to \ndR_{\ge 0}$. The class may be \emph{indexed} by a \emph{size function} $s: S \to \ndN_0$ and the corresponding \emph{generating series} may be formed by $\sum_{n \ge 0} \left(\sum_{m \in s^{-1}(\{n\})} \gamma(m)\right)z^n$ if all its coefficients are finite. } We set $\omega_k = \Pr{\xi=k}$ and for any subset $M \subset \ndN_0$ we set
\begin{align}
\label{eq:restrictions}
\phi_M(z) = \sum_{k \in M} \omega_k z^k.
\end{align}
Decomposing with respect to the outdegree of the root vertex readily yields
\begin{align}
\label{eq:decroot}
A(z) = z\phi_{\Omega}(A(z)) + \phi_{\Omega^c}(A(z)).
\end{align}

Since $0$ lies in $\Omega$  we may write
\begin{align}
\phi_{\Omega^c}(z) = z \phi_{\Omega^c}^*(z)
\end{align}
for some power series $ \phi_{\Omega^c}^*(z)$.  Hence Equation~\eqref{eq:decroot} becomes
\begin{align}
A(z) = z\phi_{\Omega}(A(z)) + A(z)\phi^*_{\Omega^c}(A(z)).
\end{align}
We may interpret this equation as follows. If the root vertex has outdegree in $\Omega$, then we have to account for it by a factor $z$ and attach the roots of a weighted forest  $\phi_\Omega(A(z))$. This accounts for the first summand. The second corresponds to the case where the outdegree of the root does not lie in $\Omega$. Here we take a root-vertex, attach to it as left-most offspring the root of a tree (counted by $A(z)$) and then add the root of a weighted forest $\phi^*_{\Omega^c}(A(z))$ as siblings to the right. If we are in the second case, then we may recurse this case-distinction at the left-most offspring of the root. In this way, we descend along the left-most offspring until we encounter for the first time a vertex with outdegree in $\Omega$. In this way we form an ordered list of $\phi^*_{\Omega^c}(A(z))$-forests, yielding
\begin{align}
\label{eq:mend}
A(z) = z\phi_{\Omega}(A(z))\frac{1}{1 - \phi^*_{\Omega^c}(A(z))}.
\end{align}
In other words,
\begin{align}
\label{eq:dec1}
A(z) = z\tilde{\phi}(A(z)) \qquad \text{with} \qquad  \tilde{\phi}(z) = \frac{\phi_\Omega(z)}{1 - \phi_{\Omega^c}^*(z)}.
\end{align}
In combinatorial language, decomposition~\eqref{eq:dec1} identifies the class $A$ as the class of $\tilde{\phi}$-enriched plane trees. We refer the reader to \cite{2016arXiv161202580S} and references given therein for details on the enriched trees viewpoint on random discrete structures.

We let $\tilde{\xi}$ denote a random non-negative integer with distribution given by the probability generating series $\tilde{\phi}$. We let $\tilde{\mT}$ denote a $\tilde{\xi}$-Galton--Watson tree and let $\tilde{\mT}_n = (\tilde{\mT} \mid |\tilde{\mT}| = n)$ denote the result of conditioning it to have $n$ vertices. For each $k \ge 0$ let $\mathfrak{B}_k$ denote the set of all vectors $(y, x_1, \ldots, x_\ell)$ with $\ell \ge 0$, $y \in \Omega$, $x_1, \ldots, x_\ell \in \Omega^c -1$, and $y + \sum_{i=1}^\ell x_i = k$. We let the weight of such a vector be given by
\begin{align}
\label{eq:weight}
\left([z^y]\phi_\Omega(z)\right) \prod_{i=1}^\ell [z^{x_i}] \phi_{\Omega^c}^*(z).
\end{align}
For each vertex $v \in \tilde{\mT}_n$ we independently select a vector $\beta_n(v) \in \mathfrak{B}_{d^+_{\tilde{\mT}_n}(v)}$ at random with probability \emph{proportional} to its weight. The pair $(\tilde{\mT}_n, \beta_n)$ is a $\tilde{\phi}$-enriched plane tree with $n$ vertices that by the decomposition \eqref{eq:dec1} corresponds to a plane tree that has precisely $n$ vertices with outdegree in the set $\Omega$. The correspondence goes by blowing up each vertex $v \in \tilde{\mT}_n$ into an ancestry line according to $\beta_n(v)$ as illustrated in Figure~\ref{fi:blowup}.

\begin{figure}[h]
	\centering
	\centering
	\includegraphics[width=0.5\textwidth]{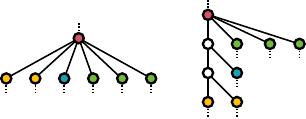}
	\caption{Blowing up a vertex $v$ (red) into an ancestry line segment for the case $\beta_n(v) = (y,x_1,x_2)$ with  $y=2$ (yellow), $x_1=1$ (blue), and $x_2=3$ (green).}
	\label{fi:blowup}
\end{figure}  

The blow-up of the random enriched plane tree $(\tilde{\mT}_n, \beta_n)$ is distributed like the random tree~$\mT_n^\Omega$. This may be verified  directly as by Rizzolo~\cite{MR3335013} or deduced from a general sampling principle~\cite[Lem. 6.1]{2016arXiv161202580S}. Note that $\Ex{\xi}<1$ implies that 
\begin{align}
\label{eq:f1}
\Ex{\tilde{\xi}} = \tilde{\phi}'(1) =  \frac{\phi_\Omega'(1) + (\phi_{\Omega^c}^*)'(1) }{1-\phi_{\Omega^c}(1)} = \frac{\phi'(1) - \phi_{\Omega^c}(1)}{1 - \phi_{\Omega^c}(1)} < 1
\end{align}
and
\begin{align}
\label{eq:ex}
1 - \Ex{\tilde{\xi}} = \frac{1 - \Ex{\xi}}{\Pr{\xi \in \Omega}}.
\end{align}
We assumed that either $\Omega$ or $\Omega^c$ is finite.  Using Equation~\eqref{eq:xi} together with $\phi_{\Omega^c}^*(1) = \phi_{\Omega^c}(1)< 1$ and {\cite[Thm. 4.8, 4.9, 4.30]{MR3097424}} it follows that
\begin{align}
\label{eq:f0}
[z^n] \tilde{\phi}(z) \sim \frac{1}{\Pr{\xi \in \Omega}} [z^n] \phi(z).
\end{align}
That is, the random variable $\tilde{\xi}$ has a regular varying density. Moreover, $\tilde{\xi}$ has  a finite variance if and only if $\xi$ has finite variance. In this case, we may use Equation~\eqref{eq:dec1} and a computer algebra system to compute
\begin{align}
\label{eq:fvar}
\Va{\tilde{\xi}} = \frac{\Ex{\xi^2}-1}{\Pr{\xi \in \Omega}} + \frac{(1- \Ex{\xi})(1 - \Ex{\xi} + 2 \Ex{\xi, \xi \in \Omega})}{\Pr{\xi \in \Omega}^2}.
\end{align}
We will require this expression for the variance later on when computing the slowly varying function $g_{\Omega}(n)$ that appears in Theorem~\ref{te:main1}. Note that Equations~\eqref{eq:gwt},~\eqref{eq:f0}, and~\eqref{eq:ex} entail
\begin{align}
\label{eq:lomegat1}
\Pr{L_\Omega(\mT) = n} &= \Pr{|\mT| = n} \\
&\sim \frac{f(n)}{\Pr{\xi \in \Omega}}\frac{1}{(n(1- \Ex{\tilde{\xi}}))^{1+\alpha}} \nonumber \\
& =\frac{f(n) \Pr{\xi \in \Omega}^\alpha}{(1 - \Ex{\xi})^{1+\alpha}} n^{-1-\alpha}. \nonumber
\end{align}

\begin{lemma}
	\label{le:capsule}
	Let $(Y^k, X_1^k, \ldots, X_{L_k}^k)$ be drawn from $\mathfrak{B}_k$ with probability proportional to the weights defined in \eqref{eq:weight}. We form the sequence
	\[
	(Y^k, X_1^k, \ldots, *, \ldots, X_{L_k}^k)
	\]
	by replacing the largest coefficient in the sequence $(Y^k, X_1^k, \ldots, X_{L_k}^k)$ with a $*$-placeholder. Let $L$, $X$, and $Y$ be  random independent integers with  distributions given by 
	\begin{align*}
	\Ex{z^L} = \frac{1 - \phi_{\Omega^c}^*(1)}{1 - z\phi_{\Omega^c}^*(1)}, \qquad \Ex{z^X} = \frac{\phi^*_{\Omega^c}(z)}{\phi^*_{\Omega^c}(1)}, \qquad \text{and} \qquad \Ex{z^Y} = \frac{\phi_{\Omega}(z)}{\phi_{\Omega}(1)}.
	\end{align*}
	Let $(X_i)_{i \ge 1}$ and $(X_i')_{i \ge 1}$ be independent copies of $X$, and let $L'$ be an independent copy of~$L$.  
	\begin{enumerate}[\qquad a)]
		\item 	If $\Omega^c$ is finite, then 
		\begin{align}
		\label{eq:lim1}
		(Y^k, X_1^k, \ldots, *, \ldots, X_{L_k}^k) \convdis (*, X_1, \ldots, X_L)
		\end{align}
		as $k$ tends to infinity.
		\item If $\Omega$ is finite, then
		\begin{align}
		\label{eq:lim2}
		(Y^k, X_1^k, \ldots, *, \ldots, X_{L_k}^k) \convdis (Y, X_1', \ldots, X_{L'}', *, X_1, \ldots, X_{L})
		\end{align}
		as $k$ tends to infinity.
		\item There are constants $\tilde{k}, C,c>0$  such that for all $k \ge \tilde{k}$ and $x \ge 0$ it holds that
		\begin{multline}
		\label{eq:bo}
		\Pr{\max(Y^{k+x}, X_1^{k+x}+1, \ldots, X_{L_{k+x}}^{k+x}+1)=k} \le \\C \frac{\Pr{\xi= k}\exp(-\frac{cx}{k})(\one_{\Pr{\xi=x}=0} + \Pr{\xi=x})}{\Pr{\xi = k+x}}.
		\end{multline}
	\end{enumerate}
\end{lemma}
\begin{proof}
	Claim a) is a  probabilistic version of the enumerative result~\eqref{eq:f0} and follows by standard arguments. Claim b) is the probabilistic version of the enumerative formula~\eqref{eq:f0} and may be justified using a general result for the asymptotic behaviour of random Gibbs partitions that exhibit a giant component~\cite[Thm. 3.1]{doi:10.1002/rsa.20771}.
	
	Claim c):  Suppose that $|\Omega| < \infty$. In this case  it holds by \cite[Thm. 4.30]{MR3097424} that
	\[
	\Pr{X_1 + \ldots  X_L = k} \sim \Ex{L} \Pr{X=k} 
	\]
	as $k$ becomes large. It follows that there are constants $C_2, k_0>0$, such that for all $k\ge k_0$ and  $x \ge 0$
	\begin{align*}
	&\Pr{\max(X_1, \ldots, X_L) = k \mid X_1 + \ldots + X_L = k+x } \\
	&\le \frac{\Pr{X=k}}{\Pr{X_1 + \ldots  +X_L = k+x}} \\&\quad\,\, \sum_{\ell\ge 1} \ell \Pr{L=\ell}  \Pr{X_1 + \ldots + X_{\ell -1} = x, \max(X_1, \ldots, X_{\ell -1}) \le k}  \\
	&\le C_2 \frac{\Pr{X= k}}{\Pr{X = k+x}}\sum_{\ell \ge 1 + \frac{x}{k} } \ell \Pr{L=\ell}\Pr{X_1 + \ldots + X_{\ell -1} = x}
	\end{align*}	
	Applying the bound \cite[Thm. 4.11]{MR3097424} yields that for any $\epsilon>0$ there are constants $c(\epsilon), \ell_0 >0$ such that for all $\ell \ge \ell_0$ \[\Pr{X_1 + \ldots + X_{\ell -1} = x} \le c(\epsilon)(1+ \epsilon)^\ell \Pr{X= x}.
	\]  Using that $L$ has finite exponential moments it follows that there are constants $c_1, C_3>0$ such that 
	\begin{align*}
	\Pr{\max(X_1, \ldots, X_L) = k \mid X_1 + \ldots + X_L = k+x } \le C_3 \frac{\Pr{X= k}\Pr{X=x} \exp(-c_1\frac{x}{k})}{\Pr{X = k+x}}.
	\end{align*}
	Since we are in the case $|\Omega|<\infty$ the random variable $Y$ has a deterministic upper bound, and it follows that there are positive constants $C_4, k_1>0$ with 
	\begin{align}
	\label{eq:bo2}
	\Pr{\max(Y^{k+x}, X_1^{k+x}+1, \ldots, X_{L_{k+x}}^{k+x}+1)=k} \le C_4 \frac{\Pr{\xi= k}\Pr{\xi=x} \exp(- c_1\frac{x}{k})}{\Pr{\xi = k+x}}
	\end{align}
	for all $k \ge k_1$ and $x \ge 0$.
	
	In the case $|\Omega^c|< \infty$ the $X_i$ are deterministically bounded and the sum $L + X_1 + \ldots + X_L$ has finite exponential moments. Hence, as $k \to \infty$
	\[
	\Pr{Y + X_1 + \ldots + X_L = k} \sim \Pr{Y=k}. 
	\]
	This implies that there  are  constants $k_2, C_5>0$ such that for all $k \ge k_2$ and $x \ge 0$
	\begin{align*}
	\Pr{\max(&Y^{k+x}, X_1^{k+x}+1, \ldots, X_{L_{k+x}}^{k+x}+1)=k} \nonumber \\
	\le&\,\, C_5 \Pr{Y=k+x}^{-1} \Pr{\max(Y, X_1 +1, \ldots, X_L + 1) = k, Y + X_1 + \ldots + X_L = k +x} \nonumber \\
	\le&\,\, C_5 \frac{\Pr{Y=k}}{\Pr{Y = k+x}} \Pr{X_1 + \ldots + X_L = x} 
	\end{align*}
	Using that $X_1 + \ldots + X_L$ has finite exponential moments it follows that there are constants $C_6, c_2>0$ such that
	\begin{align}
	\label{eq:bo1}
	\Pr{\max(&Y^{k+x}, X_1^{k+x}+1, \ldots, X_{L_{k+x}}^{k+x}+1)=k} \le C_6 \frac{\Pr{\xi=k} \exp(-c_2 x)}{\Pr{\xi = k+x}}
	\end{align}
	for all $k \ge k_2$ and $x \ge 0$.
	
	Combining the bounds~\eqref{eq:bo2} and \eqref{eq:bo1} yields Claim c).
\end{proof}

\begin{figure}[t]
	\centering
	\centering
	\includegraphics[width=1.0\textwidth]{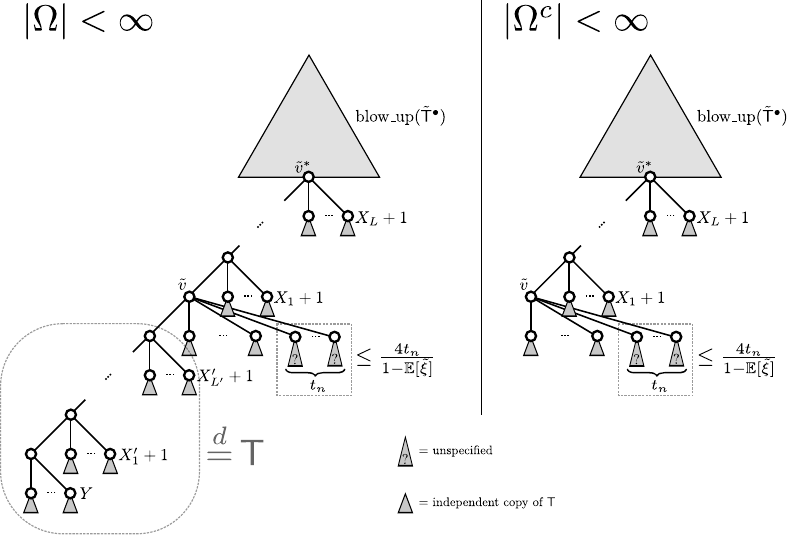}
	\caption{Illustration of the blow up of the tree $\tilde{\mT}_n$ in the two cases.}
	\label{fi:f0}
\end{figure}

\begin{proof}[Proof of Theorem~\ref{te:main2} for the case $0 \in \Omega$]
	By Equations~\eqref{eq:f0} and \eqref{eq:f1} we may apply Lemma~\ref{le:limit} to the tree $\tilde{\mT}_n$. Hence for any sequence of integers $(t_n)_{n \ge 1}$  with $t_n \to \infty$ it holds that
	\begin{align}
	\label{eq:dd}
	\left(F_0(\tilde{\mT}_n), (F_i(\tilde{\mT}_n))_{1 \le i \le \Delta(\tilde{\mT}_n) - t_n}, \one_{\sum_{i = \Delta(\tilde{\mT}_n) - t_n }^{\Delta(\tilde{\mT}_n)} |F_i(\tilde{\mT}_n)| \ge \frac{2t_n}{1 - \Ex{\tilde{\xi}}} } \right) \atv \left(\tilde{\mT}^\bullet, (\tilde{\mT}^i)_{1 \le i \le \tilde{\Delta}_{\langle n \rangle} - t_n},0\right),
	\end{align}
	with $(\tilde{\mT}_i)_{i \ge 1}$ denoting a family of independent $\tilde{\xi}$-Galton--Watson trees, $\tilde{\mT}^\bullet$ the analog of $\mT^\bullet$ that is constructed with $\tilde{\xi}$ instead of $\xi$, and
	\[
	\tilde{\Delta}_{\langle n \rangle} :=  \sup \left\{ d \ge 1 \,\,\bigg\rvert\,\, |\tilde{\mT}^\bullet| +  \sum_{i=1}^d |\tilde{\mT}^i| \le n  \right\}.
	\]

	We start with the case $|\Omega^c|< \infty$. By Lemma~\ref{le:capsule} it follows that the largest outdegree of $\mT_n^\Omega$ is with high probability equal to the giant component $Y^{\Delta(\tilde{\mT}_n)}$ in the blow-up of the  decoration \[
	\beta_n(\tilde{v}^*) = \left(Y^{\Delta(\tilde{\mT}_n)}, X_1^{\Delta(\tilde{\mT}_n)} , \ldots, X_{L_{\Delta(\tilde{\mT}_n)}}^{\Delta(\tilde{\mT}_n)}\right)
	\]
	with $\tilde{v}^*$ the (lexicographically first) vertex with maximum outdegree in~$\tilde{\mT}_n$. The small components admit the limit
	\begin{align*}
	(Y^{\Delta(\tilde{\mT}_n)}, X_1^{\Delta(\tilde{\mT}_n)}, \ldots, *, \ldots, X_{L_{\Delta(\tilde{\mT}_n)}}^{\Delta(\tilde{\mT}_n)}) \convdis (*, X_1, \ldots, X_L)
	\end{align*}

	Given $\Delta(\tilde{\mT}_n)$ the family of fringe subtrees $(F_i(\tilde{\mT}_n))_{1 \le i \le \Delta(\tilde{\mT}_n)}$ is conditionally exchangeable. Hence reordering the fringe subtrees in a suitable way and applying \eqref{eq:dd} yields that simultaneously the fringe subtrees dangling from the vertices belonging to small components in $\beta_n(\tilde{v}^*)$ and the first $Y^{\Delta(\tilde{\mT}_n)}-t_n$ fringe subtrees corresponding to the large component behave like $(\tilde{\mT}^i)_{1 \le i \le \tilde{\Delta}_{\langle n \rangle} - t_n}$. Compare with the right-hand side of Figure~\ref{fi:f0}.
	
	The limit~\eqref{eq:dd}  also tells us that the total number of vertices of the remaining $t_n$ fringe subtrees in $\tilde{\mT}_n$ is with high probability smaller than $2t_n / (1 - \Ex{\tilde{\xi}})$. When blowing up a tree we add additional vertices, but the size of any fringe subtree may at most double. This shows that  the size of the blow ups of the remaining $t_n$ fringe subtrees is with high probability smaller than $4t_n / (1 - \Ex{\tilde{\xi}})$.

	The limit of $F_0(\mT_n^\Omega)$ is determined by  $\tilde{\mT}^\bullet$ together with the small components of $\beta_n(\tilde{v}^*)$ and their fringe subtrees. 	
	Let us make this precise. 
	Note that the blow-up of $\tilde{\mT}$ with canonically chosen random local decorations is by construction distributed like  the $\xi$-Galton--Watson tree $\mT$, and the vertices of $\tilde{\mT}$ correspond bijectively to the vertices of $\mT$ with outdegree in the set $\Omega$. Let $\mS^\bullet$ denote the random marked tree constructed as follows. We start with the  blow-up of the tree $\tilde{\mT}^\bullet$ with canonically chosen random decorations. If $L=0$ we stop. Otherwise we add $X_L+1$ offspring vertices to the marked leaf. All except the first of these offspring vertices become the roots of independent copies of $\mT$. If $L=1$ we declare the first offspring vertex to be the new marked leaf and stop the construction. Otherwise the first offspring vertex becomes father of $X_{L-1} +1$ children. All but the first become roots of independent copies of $\mT$, and we proceed in the same manner until we are finished after $L$ steps in total.

	Using again that the number of vertices with outdegree in $\Omega$ in the blow-up correspond bijectively to the total number of original vertices, we may sum up what we have shown so far by
	\begin{align*}
	\left(F_0(\mT_n^\Omega), (F_i(\mT_n^\Omega))_{1 \le i \le \Delta(\mT_n^\Omega) - t_n}, \one_{\sum_{i = \Delta(\mT_n^\Omega) - t_n }^{\Delta(\mT_n^\Omega)} |F_i(\mT_n^\Omega)| \ge \frac{4t_n}{1 - \Ex{\tilde{\xi}}} } \right) \atv \left(\mS^\bullet, (\mT^i)_{1 \le i \le \Delta_{\langle n \rangle}'- t_n},0\right)
	\end{align*}
	with
	\begin{align*}
	\Delta_{\langle n \rangle}'  :=  \sup \left\{ d \ge 1 \,\,\bigg\rvert\,\, L_\Omega(\mS^\bullet) +  \sum_{i=1}^d L_\Omega(\mT^i) \le n  \right\}.
	\end{align*}
	The distribution of the random marked plane tree $\mS^\bullet$ agrees with the distribution of $\mT^\bullet$. This follows by a slightly tedious but inoffensive calculation from standard properties of size-biased geometric distributions.
	
	It remains to treat the case $|\Omega|< \infty$. This is analogous to the case $|\Omega^c| < \infty$, with the only difference being that we additionally have to take into account the small decorations $X_1', \ldots, X_{L'}'$. That is, we have to check that the circled fringe subtree on the left hand side of Figure~\ref{fi:f0} follows the distribution of the Galton--Watson tree $\mT$. But this is clear, since it is distributed like the  blow up of $\tilde{\mT}$ and hence like $\mT$.
\end{proof}

\begin{proof}[Proof of Theorem~\ref{te:main1} for the case $0 \in \Omega$]
	Equations~\eqref{eq:f0} and \eqref{eq:f1} allow us to apply Lemma~\ref{le:maxllt} to the tree $\tilde{\mT}_n$, yielding that there is a slowly varying function $\tilde{g}$ with
	\begin{align}
	\label{eq:base}
	\Pr{\Delta(\tilde{\mT}_n) = x} = \frac{1}{\tilde{g}(n)n^{1/\theta}}\left(h\left(\frac{ (1-\Ex{\tilde{\xi}})n - x}{\tilde{g}(n)n^{1/\theta}}   \right) + o(1)\right)
	\end{align}
	uniformly in $x \in \ndZ$. We are going   to show that
	\begin{align}
	\label{eq:step0}
	\sup_{1 \le \ell \le n} \left| \tilde{g}(n)n^{1/\theta}\Pr{\Delta(\mT_n^\Omega) =\ell} - h\left(\frac{ (1-\Ex{\tilde{\xi}})n - \ell}{\tilde{g}(n)n^{1/\theta}}   \right) \right| \to 0
	\end{align}
	as $n$ tends to infinity. As $h(t) \to 0$ as $|t| \to \infty$, $t \in \ndR$ this already implies that $\eqref{eq:step0}$ also holds uniformly for $\ell \in \ndZ$.
	
	Our first step in the verification~\eqref{eq:step0} is a lower bound on the maximum degree $\Delta(\mT_n^\Omega)$. By Equation~\eqref{eq:lowpass} it follows that
	\begin{align}
	\label{eq:m1}
	\tilde{g}(n) n^{1/\theta} \Pr{\Delta(\tilde{\mT}_n) \le n / \log^2 n} = o(1).
	\end{align}
	If we let $Z_n$ denote the size of the largest outdegree produced by blowing up the lexicographically first vertex $\tilde{v}$ with maximal outdegree in $\tilde{\mT}_n$, then it follows by \eqref{eq:m1}, \eqref{eq:base} and the fact that $h$ is bounded that
	\begin{align*}
	&\tilde{g}(n) n^{1/\theta} \Pr{Z_n \le n / \log^4 n} \\&= o(1) + \tilde{g}(n) n^{1/\theta}\sum_{n / \log^2 n \le r \le n} \Pr{\Delta(\tilde{\mT}_n) = r, Z_n \le n / \log^3 n} \\ 
	&\le o(1) + O(1) \sum_{n / \log^2 n \le r \le n} \Pr{ Z_n \le n / \log^4 n\mid \Delta(\tilde{\mT}_n) = r} \\
	&=  o(1) + O(1) \sum_{n / \log^2 n \le r \le n}  \Pr{\max(Y^{r}, X_1^{r}+1, \ldots, X_{L_{r}}^{r}+1)\le n / \log^4 n}.
	\end{align*}
	It follows easily by  Equation~\eqref{eq:bo} that this bound tends to zero. This shows that
	\begin{align}
	\tilde{g}(n) n^{1/\theta} \Pr{\Delta(\mT_n^\Omega) \le n / \log^4 n} = o(1).
	\end{align}
	
	Thus, it suffices to verify that ~\eqref{eq:step0} holds with $\ell$ ranging over the set $I_n$ of integers in the interval from $n / \log^4 n$ to $n$ instead. To this end, let $t_n$ be a sequence that tends to infinity and let $\ell \in I_n$ be an integer. Then
	\begin{align}
	\label{eq:master}
	\tilde{g}(n) n^{1/\theta} \Pr{\Delta(\mT_n^\Omega) =\ell} = R_{n,\ell} + S_{n,\ell}
	\end{align}
	with an error term $R_{n, \ell}$ and 
	\begin{align*}
	S_{n,\ell} =  \sum_{0 \le x \le t_n} \tilde{g}(n) n^{1/\theta}\Pr{\Delta(\tilde{\mT}_n) =\ell+x}\Pr{\max(Y^{\ell+x}, X_1^{\ell+x}+1, \ldots, X_{L_{\ell+x}}^{\ell+x}+1)=\ell} 
	\end{align*}
	the product of $\tilde{g}(n) n^{1/\theta}$ with the probability for the event that the largest outdegree in the blow-up of the lexicographically first vertex $\tilde{v}$ with maximal outdegree in $\tilde{\mT}_n$ is equal to $\ell$ and that $\Delta(\tilde{\mT}_n) - \ell \le t_n$. If this event fails but $\Delta(\mT_n^\Omega) =\ell$, then at least one of the following events must take place.
	\begin{enumerate}[\qquad  1)]
		\item  The maximal outdegree in the blow-up of the vertex  $\tilde{v}$ equals $\ell$ but $\Delta(\tilde{\mT}_n) - \ell > t_n$. We let $R_{n, \ell}(1)$ denote the product of $\tilde{g}(n) n^{1/\theta}$ with the probability for this event.
		\item At least two vertices of $\tilde{\mT}_n$ have outdegree at least $\ell$ and the blow-up of one of them produces a vertex with outdegree equal to $\ell$. The product of $\tilde{g}(n) n^{1/\theta}$ with the  probability for this event is denoted by $R_{n,\ell}(2)$.
	\end{enumerate}
	Hence
	\begin{align}
	\label{eq:xo}
	R_{n, \ell} \le R_{n,\ell}(1) + R_{n, \ell}(2).
	\end{align}
	We are going to verify that this bound tends to zero uniformly for all $\ell \in I_n$. Using  Inequality~\eqref{eq:bo} and Equality~\eqref{eq:base} it follows that 
	\begin{align*}
	R_{n,\ell}(1) &\le \sum_{t_n \le x \le n} \tilde{g}(n) n^{1/\theta}\Pr{\Delta(\tilde{\mT}_n) =\ell+x}\Pr{\max(Y^{\ell+x}, X_1^{\ell+x}+1, \ldots, X_{L_{\ell+x}}^{\ell+x}+1)=\ell} \\
	&\le C \sum_{t_n \le x \le n}\left(h\left(\frac{ (1-\Ex{\tilde{\xi}})n - \ell - x}{\tilde{g}(n)n^{1/\theta}}   \right) + o(1)\right)  \frac{\Pr{\xi= \ell}\Pr{\xi=x}}{\Pr{\xi = \ell+x}}\exp\left(-\frac{x}{\ell}\right) \\
	&\le O(1) \sum_{x \ge t_n} \Pr{\xi=x} \frac{f(\ell)}{f(\ell(1+\frac{x}{\ell}))}\left(1 + \frac{x}{\ell}\right)^{1 + \alpha}\exp\left(-\frac{x}{\ell}\right) \\
	&\le O(1)  \Pr{\xi\ge t_n} 
	\end{align*}
	and thus
	\begin{align}
	\label{eq:xo1}
	R_{n, \ell}(1) \to 0
	\end{align}
	uniformly for all $\ell \in I_n$.
	Here we have used that $h$ is bounded, that the $o(1)$ term tends to zero uniformly, and that the Potter bounds for slowly varying functions imply that for any $\epsilon>0$ there is a positive constant $C(\epsilon)$ with
	\begin{align*}
	\frac{f(\ell)}{f(\ell(1+\frac{x}{\ell}))} \le C(\epsilon) \left(1 + \frac{x}{\ell} \right)^\epsilon
	\end{align*}
	for all $\ell, x \ge 1$.
	
	In order to verify that the bound in \eqref{eq:xo} tends to zero it remains to show that $R_{n, \ell}(2)$ tends to zero. Let $(\tilde{\xi}_i)_{i \ge 1}$ be a family of independent copies of the offspring distribution $\tilde{\xi}$ and set $\tilde{S}_k = \tilde{\xi}_1 + \ldots + \tilde{\xi}_k$ for all $k$. Using Lemma~\ref{le:cyc} and Equation~\eqref{eq:teq} it follows that
	\begin{multline*}
	R_{n,\ell}(2) \le \frac{\tilde{g}(n)n^{1/\theta}}{\Pr{\tilde{S}_n = n-1}} n^2 \sum_{x \ge 0} \Pr{\tilde{\xi}=\ell +x}\Pr{\tilde{S}_{n-1} = n-1 - \ell -x, \tilde{\xi}_1 \ge \ell} \\ \Pr{\max(Y^{\ell+x}, X_1^{\ell+x}+1, \ldots, X_{L_{\ell+x}}^{\ell+x}+1)=\ell}.
	\end{multline*}
	By Inequality~\eqref{eq:bo}, the asymptotic expansion~\eqref{eq:f0}, and using \cite[Cor. 2.1]{MR2440928} in a similar way as in \eqref{eq:gwt} it follows that we may bound $R_{n, \ell}(2)$ by 
	\begin{align*} &\frac{O(1)\tilde{g}(n)n^{1/\theta+1}\Pr{\xi = \ell}}{  \Pr{\tilde{\xi} = \lfloor(n-1)(1- \Ex{\tilde{\xi}})\rfloor}} \sum_{x \ge 0} \Pr{\tilde{S}_{n-1} = n-1 - \ell -x, \tilde{\xi}_1 \ge \ell}  (\one_{\Pr{\xi=x}=0} + \Pr{\xi=x}).
	\end{align*}
	By the expansion~\eqref{eq:f0} there  are constants  $x', C'>0$ such that for all $x \ge x'$
	\[
	\one_{\Pr{\xi=x}=0} + \Pr{\xi=x} = \Pr{\xi = x} \le C'\Pr{\tilde{\xi}=x}.
	\]
	Using the local limit theorem~\eqref{eq:llt} (for $\tilde{S}_n$ instead of $S_n$) and the fact that the density $h$ is bounded it follows that
	\begin{align*}
	&R_{n,\ell}(2) \\
	&\quad\le O(1) \tilde{g}(n) n^{1/\theta +1} \left( \Pr{\tilde{S}_n = n-1 - \ell, \tilde{\xi}_1 \ge \ell} +  \Pr{\tilde{S}_{n-1} \in n-1 - \ell - [0,x'], \tilde{\xi}_1 \ge \ell} \right) \\
	&\quad\le O(1) \tilde{g}(n)n^{1/\theta + 1} \sum_{i \ge \ell} \Pr{\tilde{\xi} = i} \frac{O(1)}{\tilde{g}(n)n^{1/\theta}} \\
	&\quad\le O(1) n \Pr{\xi \ge \ell}
	\end{align*}
	and thus
	\begin{align}
	\label{eq:xo2}
	R_{n, \ell}(2) \to 0
	\end{align}
	uniformly for all $\ell \in I_n$. Combining Inequality~\eqref{eq:xo} and the limits \eqref{eq:xo1} and \eqref{eq:xo2} we obtain
	\begin{align}
	\label{eq:rnull}
	R_{n, \ell} \to 0.
	\end{align}
	
	We now turn our attention to $S_{n,\ell}$. 
	By the limits \eqref{eq:lim1} and \eqref{eq:lim2} it follows that there is a probability density $(p_x)_{x \ge 0}$ such that for each constant integer $x \ge 0 $ it holds that
	\begin{align*}
	\lim_{k \to \infty} \Pr{\max(Y^{k+x}, X_1^{k+x}+1, \ldots, X_{L_{k+x}}^{k+x}+1)=k} = p_x.
	\end{align*}
	Consequently, if we choose our sequence $t_n$  such that it tends  to infinity sufficiently slowly, then
	\begin{align}
	\label{eq:doo}
	\sum_{x=0}^{t_n}  \Pr{\max(Y^{\ell+x}, X_1^{\ell+x}+1, \ldots, X_{L_{\ell+x}}^{\ell+x}+1)=\ell} = 1 + o(1)
	\end{align}
	holds uniformly for all $\ell \in I_n$. We may additionally assume that  $t_n = o(\tilde{g}(n) n^{1/\theta})$. As $h$ is uniformly continuous this implies that
	\[
	h\left(\frac{ (1-\Ex{\tilde{\xi}})n - \ell - x}{\tilde{g}(n)n^{1/\theta}}   \right) = h\left(\frac{ (1-\Ex{\tilde{\xi}})n - \ell}{\tilde{g}(n)n^{1/\theta}}   \right) + o(1)
	\]
	holds uniformly for all $\ell \in I_n$ and $0 \le x \le t_n$.  Using the local limit theorem \eqref{eq:base} and Equation~\eqref{eq:doo} it follows that
	\begin{align*}
	S_{n,\ell} &=  \sum_{0 \le x \le t_n} \tilde{g}(n) n^{1/\theta}\Pr{\Delta(\tilde{\mT}_n) =\ell+x}\Pr{\max(Y^{\ell+x}, X_1^{\ell+x}+1, \ldots, X_{L_{\ell+x}}^{\ell+x}+1)=\ell} \\
	&=   \sum_{0 \le x \le t_n} \left(h\left(\frac{ (1-\Ex{\tilde{\xi}})n - \ell}{\tilde{g}(n)n^{1/\theta}}   \right) + o(1)\right)\Pr{\max(Y^{\ell+x}, X_1^{\ell+x}+1, \ldots, X_{L_{\ell+x}}^{\ell+x}+1)=\ell} \\
	&= h\left(\frac{ (1-\Ex{\tilde{\xi}})n - \ell}{\tilde{g}(n)n^{1/\theta}}   \right) + o(1)
	\end{align*}
	with uniform $o(1)$ terms.  By Equation~\eqref{eq:master} and \eqref{eq:rnull} it follows that
	\[
	\tilde{g}(n) n^{1/\theta} \Pr{\Delta(\mT_n^\Omega) =\ell} = h\left(\frac{ (1-\Ex{\tilde{\xi}})n - \ell}{\tilde{g}(n)n^{1/\theta}}   \right) + o(1)
	\]
	holds uniformly for $\ell \in I_n$. This completes the proof.
\end{proof}

We may compute the slowly varying function $g_\Omega(n)$, hence verifying~\eqref{eq:gslowly} for the case $0 \in \Omega$:
\begin{proposition}
	\label{pro:gomega}
	In the case $0 \in \Omega$, we may set
	\begin{align}
	g_\Omega(n) = \begin{cases}
	\sqrt{\frac{\Va{\tilde{\xi}}}{2}}, \quad &\Va{\xi}<\infty \\
	\frac{g(n)}{\Pr{\xi \in \Omega}^{1/\theta}}, \quad &\Va{\xi} = \infty.
	\end{cases}
	\end{align}
	The variance $\Va{\tilde{\xi}}$ was expressed in~\eqref{eq:fvar}:
	\[
	\Va{\tilde{\xi}} =  \frac{\Ex{\xi^2}-1}{\Pr{\xi \in \Omega}} + \frac{(1- \Ex{\xi})(1 - \Ex{\xi} + 2 \Ex{\xi, \xi \in \Omega})}{\Pr{\xi \in \Omega}^2}.
	\]
\end{proposition}
\begin{proof}
	In the previous proof we verified Theorem~\ref{te:main1} for $g_\Omega(n) = \tilde{g}(n)$.
	The function $\tilde{g}(n)$ stems from the local limit theorem~\ref{eq:llt} and Equation~\eqref{eq:defofg}, but for $\tilde{\xi}$ instead of $\xi$. 
	Using~\eqref{eq:defofg} the slowly varying function $\tilde{g}(n)$ may be chosen to satisfy 	$\tilde{g}(n) = \sqrt{\frac{\Va{\tilde{\xi}}}{2}}$ if $\Va{\tilde{\xi}}<\infty$, which is equivalent to $\Va{\xi} <\infty$. Suppose that $\Va{\xi} = \infty$. Using $\Pr{\tilde{\xi}=n} \sim \frac{1}{\Pr{\xi \in \Omega}} \Pr{\xi=n}$ from~\eqref{eq:f0}, it follows that if $1<\theta<2$
	\begin{align}
	\frac{ \inf\left\{ x \ge 0 \mid \Pr{\tilde{\xi} > x} \le \frac{1}{n} \right\}}{ \inf\left\{ x \ge 0 \mid \Pr{\xi > x} \le \frac{1}{n} \right\}} \to \frac{1}{\Pr{\xi \in \Omega}^{1/\theta}}.
	\end{align}
	Hence we set $\tilde{g}(n) = \frac{1}{\Pr{\xi \in \Omega}^{1/\theta}} g(n)$ in the this case. The remaining case $\Va{\xi}=\infty$ and $\theta=2$ follows by similar arguments.
\end{proof}

\begin{remark}
	\label{re:extend1}
	Throughout we assumed that either $\Omega$ or $\Omega^c$ is finite. Accordingly, the maximal outdegree of $\mT_n^\Omega$ either belongs with high probability to $\Omega^c$ or with high probability to $\Omega$. In our proofs we constructed the tree $\mT_n^\Omega$ via a bijection from the decorated tree $(\tilde{\mT}_n, \beta_n)$. The difference between the two cases is reflected by the asymptotic behaviour of the decoration  $\beta_n(\tilde{v}^*)$ of the lexicographically first vertex $\tilde{v}^*$ of $\tilde{\mT}_n$ with maximal outdegree. As ensured by Lemma~\ref{le:capsule}, according to whether $\Omega$ or $\Omega^c$ is finite we will either see a giant ``$\Omega$-component'' in $\beta_n(\tilde{v}^*)$ or a giant ``$\Omega^c$-component''. However, as we saw in the proofs of Theorems~\ref{te:main1} and~\ref{te:main2}, the implications for the tree $\mT_n^\Omega$ are the same in both cases.
	
	If $\Omega$ and $\Omega^c$ are both infinite, then we expect a mixture of these two behaviours to occur, with the end result for the tree $\mT_n^\Omega$ being the same. However, when making this formal, rather unpleasant periodicity issues have to be taken into account. Studying the asymptotic  behaviour of the coefficients $[z^k]\tilde{\phi}(z)$  of the probability generating function $\tilde{\phi}(z) = \phi_\Omega(z) / (1 - \phi_{\Omega^c}(z)/z)$ for $\tilde{\xi}$ requires us to partition the support of $\tilde{\xi}$ into multiple infinite subsets. In order to extend the proofs we would have to show versions of Equation~\eqref{eq:f0} and Lemma~\ref{le:capsule} for restrictions to theses subsets. We would also have to make sure that our main tools still apply to $\tilde{\xi}$, so that the proof of the case $\Omega=\ndN_0$ may be extended to this setting: The large deviation results of~\cite{MR2440928} already apply to $O$-regularly varying probability densities. The results of~\cite{MR2775110} assume $\Delta$-subexponentiality, so we would have to verify that they also apply to this setting.
\end{remark}

\section{The case $0 \notin \Omega$}

\label{sec:0not}

Throughout this section we assume that $\Omega$ is a proper subset of $\ndN_0$ satisfying  $0 \notin \Omega$.  We are going to use the same strategy as in the case $0 \in \Omega$: we generate the random tree $\mT_n^\Omega$ via a blow-up construction of some random decorated tree. Using Gibbs partition methods we will then show that the decoration of the largest degree in this second tree is likely to produce the vertex with largest degree in~$\mT_n^\Omega$. 


\subsection{Decomposition}

We recap the lifeline-tree decomposition by Rizzolo~\cite{MR3335013} that in the special case of plane trees extends the Ehrenborg--M\'endez bijection. 

We consider the generating series $A(z)$ of the class of weighted plane trees indexed by their number of vertices with outdegree in $\Omega$. We denote the restrictions $\phi_M(z) = \sum_{k \in M} \Pr{\xi = k} z^k$ of the generating series $\phi(z) = \Ex{z^\xi}$ to subsets $M \subset \ndN_0$. We set
\begin{align}
p := \Pr{L_{\Omega}(\mT) = 0}.
\end{align}
Our assumption $\Pr{\xi \in \Omega}>0$ ensures that $p<1$. As we are in the case $0 \notin \Omega$, it follows that $p>0$. Hence \begin{align}0<p<1.\end{align} 
Let us consider the subclass $A^*$ of $A$ of finite plane trees with at least one vertex having outdegree in $\Omega$. Its complement $A^{**}(z)$ consists of finite plane trees  whose outdegrees are required to belong to $\Omega^c$. Clearly
\begin{align}
A^{**}(z) = p
\end{align}
and
\begin{align}
\label{eq:d11}
A(z) = A^*(z) + p.
\end{align}
Given a tree from $A^*$, we may consider the path from the root to the lexicographically first vertex with outdegree in $\Omega$ and call it the spine of the tree. It consists of a sequence of vertices with outdegree in $\Omega^c$ and a single vertex with outdegree $\Omega$.  Any of these spine vertices with outdegree in $\Omega^c$ have the property, that their offspring to the left of the spine have only descendants with outdegree in $\Omega^c$. That is, they are roots of fringe subtrees belonging to $A^{**}$. Those to the right of the spine may have descendants of arbitrary kind and are hence roots of trees from $A$. Moreover, the offspring of the unique spine vertex with outdegree in $\Omega$ may also have arbitrary descendants. Hence
\begin{align}
A^*(z) = z \frac{\phi_{\Omega}(A(z))}{1 - \phi^*_{\Omega^c}(A(z)) }
\end{align}
with
\begin{align}
\phi^*_{\Omega^c}(z) := \sum_{k \in \Omega^c} \Pr{\xi=k} \sum_{i=0}^{k-1} p^i z^{k-i-1}.
\end{align}
The sum index $i$ corresponds to the number of vertices that lie to the left of the spine.   As $A(1)=1$ it follows from~\eqref{eq:d11} that $A^*(1) = 1 - p$.  Using~\eqref{eq:d11} it follows that
\begin{align}
\label{eq:dec2}
A^*(z)/(1-p) = z \varphi(A^*(z)/(1-p))
\end{align}
with
\begin{align}
\label{eq:fo}
\varphi(z) = (1-p)^{-1}\frac{\phi_{\Omega}(p + z(1-p))}{1 - \phi^*_{\Omega^c}(p +z(1-p)) }.
\end{align}
In combinatorial language,  decomposition~\eqref{eq:dec2} identifies the class $A^*$ as the class of $ \varphi$-enriched plane trees. That is, any tree from $A^*$ with $n$ vertices having outdegree in $\Omega$ corresponds in a bijective and weight-preserving manner to a plane tree with $n$ vertices where each vertex is endowed with a $\varphi$-decoration whose size matches the outdegree of the vertex. We refer the reader to \cite{2016arXiv161202580S} and references given therein for details on the enriched trees viewpoint on random discrete structures.

\subsection{Sampling procedure}

\label{sec:blowup}

Note that~\eqref{eq:dec2} entails that $\varphi(1) = 1$. Hence it is the probability generating function of some random non-negative integer $\zeta$. 
We let $\mA_n$ denote a $\zeta$-Galton--Watson tree conditioned on having $n$ vertices. Rizzolo~\cite{MR3335013} calls this tree the lifeline tree. The random tree $\mT_n^\Omega$ may be constructed from $\mA_n$ by a randomized blow-up procedure which we are going to describe in this section.

\begin{figure}[h]
	\centering
	\centering
	\includegraphics[width=0.8\textwidth]{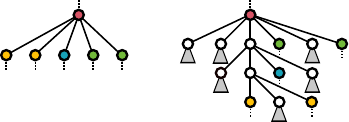}
	\caption{Blowing up a vertex $v$ (red) with $\ell= 2$, $\bm{x}_1 = (2,3)$, $\bm{x}_2 = (1,2)$, $y=3$, $m = \{1,3,4,6,8\}$ into a vertebrate. The grey triangles represent independent copies of $\mT^*$, a $\xi$-Galton--Watson tree conditioned on producing no vertex with outdegree in $\Omega$.}
	\label{fi:blowup2}
\end{figure}  

We let $L$ denote a geometric random variable with distribution given by
\begin{align}
\Ex{z^L} = \frac{1 - \phi_{\Omega^c}^*(1)}{1 - z\phi_{\Omega^c}^*(1)}.
\end{align}
Furthermore, we let $\bm{X} = (X(1), X(2))$ denote a random pair of non-negative integers with probability generating function
\begin{align}
\label{eq:def}
\Ex{z_1^{X(1)}z_2^{X(2)}} = \frac{1}{\phi^*_{\Omega^c}(1)} \sum_{k \in \Omega^c} \Pr{\xi=k} \sum_{i=0}^{k-1} p^i z_1^i  z_2^{k-i-1}.
\end{align}
We let $Y$ denote a random non-negative integer with distribution determined by
\begin{align}
\Ex{z^Y} = \frac{\phi_{\Omega}(z)}{\phi_{\Omega}(1)}.
\end{align}
We assume the three random variables $L$, $\bm{X}$, and $Y$ to be independent. We let $\bm{X}_i = (X_i(1), X_i(2))$ for $i \ge 1$ denote independent copies of $\bm{X}$. Equation~\eqref{eq:fo} entails that
\begin{align}
\label{eq:blowup}
\zeta \eqdist \mathrm{Bin}\left(Y +  \sum_{i=1}^L X_i(2),1-p\right).
\end{align}
Let $M$ denote a subset of $\{1, \ldots, Y + \sum_{i=1}^L X_i(2)\}$ drawn according to a binomial distribution with parameter $1-p$.  Equation~\eqref{eq:dec2} entails that we may generate a $\xi$-Galton--Watson tree $\mA^*$ conditioned on producing at least one vertex with outdegree belonging to $\Omega$ by ``blowing up'' a $\zeta$-Galton--Watson tree $\mA$. The sampling procedure draws for each vertex $v$ of $\mA$ an independent copy $(y, (\bm{x})_{1 \le i \le \ell}, m)$ of $(Y, (\bm{X})_{1 \le i \le L}, M)$ conditioned on
\[
d^+_{\mA}(v) = |M|.
\]
The vertex $v$ is then blown up as illustrated in Figure~\ref{fi:blowup2} into a vertebrate whose spine has length $\ell+1$. With $\bm{x}_i = (x_i(1), x_i(2))$ for each $1 \le i \le \ell$ the $i$th spine vertex receives $x_i(1)$ offspring vertices to the left of the spine, and $x_i(2)$ offspring vertices to the right of the spine. Each offspring to the left becomes the root of an independent copy of $\mT^*$, a $\xi$-Galton--Watson tree conditioned on producing no vertex with outdegree in $\Omega$. Each offspring to the right is coloured grey.  The tip of the spine receives $y$ grey offspring vertices. The set $m$ maybe interpreted as a subset of the  $y + \sum_{i=1}^{\ell} x_i(2)$ grey vertices in a canonical way. Its elements may be identified with the original offspring of $v$ in a canonical way. All remaining grey vertices become roots of independent copies of $\mT^*$.

We refer to the additionally sampled data for each vertex $v$ as its decoration. Of course, this blow-up operation may also be applied to any plane tree having at least one vertex with outdegree in $\Omega$. In particular, if $\mA_n$ denotes the result of conditioning the $\zeta$-Galton--Watson tree $\mA$ on producing $n$ vertices, then the blow-up of $\mA_n$ is distributed like $\mT_n^\Omega$. 

\subsection{Asymptotics of the blow-ups}

Our aim is to apply the case $\Omega=\ndN_0$ to the tree~$\mA_n$, yielding that $\mA_n$ exhibits a  vertex with degree fluctuating around a constant multiple of $n$. As \[\zeta \eqdist \mathrm{Bin}\left(Y +  \sum_{i=1}^L X_i(2),1-p\right)\] follows a binomial distribution with a random number of slots, we need some  estimates on the asymptotic behaviour of such compound random variables when we condition them to be large.
Proposition~\ref{pro:binomial} below ensures that when $ \mathrm{Bin}\left(Y +  \sum_{i=1}^L X_i(2),1-p\right)$ is large, then so is $Y +  \sum_{i=1}^L X_i(2)$, and  we have a local limit theorem available to determine the fluctuation. The question now is what happens when the sum $Y +  \sum_{i=1}^L X_i(2)$ becomes large. As we shall see further down below, there will be a single macroscopic summand accounting for the entire mass minus a stochastically bounded remainder.

We start with a description of the behaviour of a random non-negative integer $Z$ conditioned on the event $\mathrm{Bin}{(Z,q)} = n$ that $Z$ i.i.d. coin flips  yield head precisely $n$ times, assuming that each individual flip shows head with a fixed probability $0<q<1$.
\begin{proposition}
	\label{pro:binomial}
	Let $Z$ be a random non-negative integer such that $\Pr{Z=n}$ varies regularly. Let $0<q<1$ be given.
	It holds uniformly for all $k \in \ndZ$ that
	\begin{align}
	\label{eq:llt111}
	\Pr{Z = k \mid \mathrm{Bin}{(Z,q)} = n} = \frac{1}{\sqrt{n}} \left( o(1) + \frac{q}{\sqrt{2 \pi }\sqrt{1-q} } \exp\left(- \frac{1}{2} \left(\frac{k - n/q}{\sqrt{n}q^{-1}\sqrt{1-q}} \right)^2 \right) \right).
	\end{align}
	There are constants $C,c>0$ such that for all $n \ge 1$ and $\epsilon>0$
	\begin{align}
	\label{eq:devi}
	\Pr{ |Z - n/q| > \epsilon \sqrt{n} \mid \mathrm{Bin}(Z,q)=n} \le C (\exp(-cn) + n \exp(-c \epsilon^2)).
	\end{align}
	Moreover, 
	\begin{align}
	\label{eq:asymp1}
	\Pr{\mathrm{Bin}{(Z,q)} = n} \sim \frac{1}{q}\Pr{Z = \lfloor n/q \rfloor}.
	\end{align}
\end{proposition}
\begin{proof}
	Note that the event $\mathrm{Bin}(Z,q) = n$ entails $Z \ge n$.
	Using the Azuma--Hoeffding inequality, it follows that for any fixed $0<\delta<1$ there are constants $C,c>0$ such that for all integers $n \ge 1$
	\begin{align}
	\Pr{Z \notin (1 \pm \delta)n/q,\mathrm{Bin}(Z,q)= n} \le C\exp(-cn).
	\end{align}
	Using again the Azuma--Hoeffding inequality, it follows that for any $\epsilon>0$
	\begin{multline}
	\label{eq:aa1}
	\Pr{ |Z - n/q|>\epsilon \sqrt{n}, \mathrm{Bin}(Z,q) =n } \le \\ C \exp(-cn) + \sum_{\substack{k \in (1 \pm \delta) \frac{n}{q} \\ |n-qk|>\epsilon q \sqrt{n}}} \Pr{Z=k} 2 \exp\left(- \frac{|n-qk|^2}{2k} \right).
	\end{multline}	
	From this we obtain the crude bound
	\begin{align}
	\Pr{ |Z - n/q|>\epsilon \sqrt{n}, \mathrm{Bin}(Z,q) =n } &\le C \exp(-cn) + 2 \exp\left( -\frac{\epsilon^2 q^3}{ 2(1+ \delta)} \right). 
	\end{align}
	Using the strengthened local  limit theorem from~\cite[p. 79, P10]{MR0388547} and setting $\sigma = \sqrt{q(1-q)}$ and $x = \frac{n-kq}{\sigma \sqrt{k}}$ it follows that
	\begin{multline}
	\label{eq:dohdd}
	\Pr{\mathrm{Bin}(Z,q) = n} = \\ O(1) n^{-\Theta(\log n)} + \sum_{k \in n/q \pm \sqrt{n} \log n} \Pr{Z = k} \frac{1}{\sqrt{k}}\left( o\left(\frac{1}{\max\left(1, x^2\right)}\right)  + \frac{1}{\sqrt{2 \pi}\sigma} \exp\left(- x^2/2\right)\right). 
	\end{multline}
	For $k \in n/q \pm \sqrt{n} \log n$, it holds that $k \sim n/q$ and, since we assumed that the density of $Z$ varies regularly,
	\begin{align}
	\Pr{Z= k} \sim \Pr{Z= \lfloor n/q\rfloor}.
	\end{align}
	Using dominated convergence, it follows that
	\begin{align*}
	&\Pr{\mathrm{Bin}(Z,q) = n}  \\&= \frac{1}{q}\Pr{Z = \lfloor n/q \rfloor}\left( o(1) + \sum_{k \in n/q \pm \sqrt{n}\log n} \frac{1}{\sqrt{2 \pi n} \sigma / q^{3/2}} \exp\left(- \frac{1}{2} \left(\frac{k - n/q}{\sqrt{n}\sigma/q^{3/2}} \right)^2 \right) \right) \\
	&=\frac{1}{q}\Pr{Z = \lfloor n/q \rfloor}(1 + o(1)).
	\end{align*}
	This verifies~\eqref{eq:asymp1}. Combining this with Inequality~\eqref{eq:aa1} we may verify the bound~\eqref{eq:devi} too.
	
	In order to verify~\eqref{eq:llt111}, note that for if $k \notin n/q \pm  \sqrt{n} \log n$ it follows from~\eqref{eq:devi} that 
	\begin{align}
	\label{eq:strength1}
	\Pr{Z = k \mid \mathrm{Bin}{(Z,q)} = n} = O(n^{-\Theta(\log n)}).
	\end{align}
	Hence it suffices to verify~\eqref{eq:llt111} for  $k \in n/q \pm \sqrt{n}  \log n$. Arguing analogously as for~\eqref{eq:dohdd} it follows that uniformly for all $k$ in that interval
	\begin{align}
	\label{eq:strength2}
	&\Pr{Z = k \mid \mathrm{Bin}{(Z,q)} = n} \nonumber \\
	&= \frac{\Pr{Z=k}}{\Pr{\mathrm{Bin}{(Z,q)} = n}} \frac{1}{\sqrt{k}} \left( o\left(\frac{1}{\max\left(1, x^2\right)}\right) + \frac{1}{\sqrt{2 \pi}\sigma} \exp\left(- x^2/2\right) \right). \nonumber \\
	&=		
	\frac{1}{\sqrt{n}} \left( o\left(\frac{1}{\max\left(1, x^2\right)}\right) + \frac{1 + O(\log^2(n) /n)}{\sqrt{2 \pi } \sigma / q^{3/2}} \exp\left(- \frac{1}{2} \left(\frac{k - n/q}{\sqrt{n}\sigma/q^{3/2}} \right)^2 \right) \right) \nonumber \\
	&= \frac{1}{\sqrt{n}} \left( o\left(\frac{1}{\max\left(1, x^2\right)}\right) + \frac{1}{\sqrt{2 \pi } q^{-1}\sqrt{1-q} } \exp\left(- \frac{1}{2} \left(\frac{k - n/q}{\sqrt{n}q^{-1}\sqrt{1-q}} \right)^2 \right) \right). 
	\end{align}
	This verifies~\eqref{eq:llt111}.
\end{proof}

\begin{proposition}
	\label{pro:pro}
	Suppose that $\Omega$ is finite. Then
	\begin{align}
	\label{eq:jj1}
	\Pr{X(2) = n} \sim \frac{1}{(1-p)\phi^*_{\Omega^c}(1)}\Pr{\xi=n}.
	\end{align}
	as $n$ tends to infinity. Moreover,
	\begin{align}
	\label{eq:jj2}
	\left( X(1) \mid X(2) = n \right) \convdis \mathrm{Geom}(p)
	\end{align}
	for a geometrically distributed random variable $\mathrm{Geom}(p)$  that assumes an integer $i \ge 0$ with probability $p^i/(1-p)$. Even more, if $t_n$ denotes a sequence of positive integers satisfying $t_n = o(n)$, then
	\begin{align}
	\label{eq:jj3}
	\Pr{X(1) = i \mid X(2) = n} \sim p^i/(1-p)
	\end{align}
	holds uniformly for all $0 \le i \le t_n$. Finally, there is a  constant $C>0$ such that for all sufficiently large $n$ and all $x \ge 0$ 
	\begin{align}
	\label{eq:jj4}
	\Pr{X(1) \ge x , X(2) = n} \le C p^x \Pr{\xi= n}.
	\end{align}
\end{proposition}
\begin{proof}
	Let $i \ge 0$ be an integer. From~\eqref{eq:def} it follows that
	\begin{align}
	\label{eq:zo1}
	\Pr{X(1) = i, X(2) = n} = \frac{p^i}{\phi^*_{\Omega^c}(1)}  \Pr{\xi=n+i+1} \one_{n+i+1 \in \Omega^c}.
	\end{align}
	Since we assumed that $\Omega$ is finite, it holds that $n+i+1 \in \Omega^c$ for all  sufficiently large $n$, not depending of the value of $i$. Let $(t_n)_{n \ge 1}$ be a  sequence of positive integers  with $t_n = o(n)$ and $t_n \to \infty$. As $\xi$ has a regularly varying density it follows that
	\begin{align}
	\label{eq:doit1114}
	\sup_{x \in \ndZ, |x| \le t_n} \left| \frac{\Pr{\xi= n+x}}{\Pr{\xi=n}}-1\right| \to 0.
	\end{align}
	Hence
	\begin{align}
	\sum_{i=0}^{t_n} \Pr{X(1) = i, X(2) = n} \sim \frac{1}{(1-p)\phi^*_{\Omega^c}(1)}\Pr{\xi=n}.
	\end{align}
	Moreover, the Potter bounds imply that
	\begin{align}
	\label{eq:aaa}
	\sup_{x \in \ndZ, x \ge 0} \left| \frac{\Pr{\xi=n+x}}{\Pr{\xi=n}} \right| = O(1)
	\end{align}
	as $n\to \infty$. Thus
	\begin{align}
	\label{eq:almostthes}
	\frac{1}{\Pr{\xi=n}}\sum_{i > t_n} \Pr{X(1) = i, X(2) = n} = O(p^{t_n}).
	\end{align}
	It follows that
	\begin{align}
	\Pr{X(2)=n} \sim \frac{1}{(1-p)\phi^*_{\Omega^c}(1)}\Pr{\xi=n}.
	\end{align}
	This verifies Equation~\eqref{eq:jj1}. Combining~\eqref{eq:zo1},~\eqref{eq:doit1114} and~\eqref{eq:jj1} implies that
	\begin{align}
	\Pr{X(1)=i \mid X(2) =n} \sim p^i/(1-p)
	\end{align}
	holds uniformly for all $0 \le i \le t_n$.  This verifies~\eqref{eq:jj2} and~\eqref{eq:jj3}. Finally,~\eqref{eq:jj4} follows from~\eqref{eq:zo1} and~\eqref{eq:aaa},  analogously as in the verification of~\eqref{eq:almostthes}.
\end{proof}

Propositions~\ref{pro:binomial} and~\ref{pro:pro} enable us to determine the asymptotic behaviour of $\zeta$:

\begin{lemma}
	\label{le:regvar}
	As $n \to \infty$, 	
	\begin{align}
	\label{eq:step1}
	\Prb{Y +  \sum_{i=1}^L X_i(2) = n} \sim  \frac{\Pr{\xi=n}}{\Pr{\xi \in \Omega}}
	\end{align}
	Consequently, 
	\begin{align}
	\label{eq:zetadensity}
	\Pr{\zeta=n} \sim \frac{(1-p)^\alpha}{\Pr{\xi \in \Omega}} \Pr{\xi = n}.
	\end{align}
\end{lemma}
\begin{proof}
	Suppose that $\Omega^c$ is finite. Then $X(2)$ is bounded. As $L$ is light-tailed, it follows that $\sum_{i=1}^L X_i(2)$ is light-tailed. On the other hand, $Y$ is heavy-tailed.  Using \cite[Lem. 4.9]{MR3097424}, it follows that 
	\begin{align}
	\label{eq:foya1}
	\Prb{Y +  \sum_{i=1}^L X_i(2) = n} &\sim \Pr{Y=n}  \\&\sim \frac{1}{\phi_\Omega(1)}\Pr{\xi=n}. \nonumber
	\end{align}
	Suppose that $\Omega$ is finite. As $L$ is light-tailed, Proposition~\ref{pro:pro} allows us to apply~\cite[Thm. 4.30]{MR3097424}, yielding
	\begin{align}
	\label{eq:foya2}
	\Prb{\sum_{i=1}^L X_i(2) = n} &\sim \Ex{L}\Pr{X(2)=n} \\
	&\sim \frac{\Ex{L}}{(1-p)\phi^*_{\Omega^c}(1)}\Pr{\xi=n} \nonumber \\
	&\sim \frac{1}{(1-p)(1- \phi^*_{\Omega^c}(1))}\Pr{\xi=n}. \nonumber
	\end{align}
	Since $Y$ is now bounded, we may apply \cite[Lem. 4.9]{MR3097424} to obtain
	\begin{align}
	\label{eq:foyaddd}
	\Prb{Y +  \sum_{i=1}^L X_i(2) = n} &\sim \Prb{ \sum_{i=1}^L X_i(2) = n} \\
	&\sim \frac{1}{(1-p)(1- \phi^*_{\Omega^c}(1))}\Pr{\xi=n}. \nonumber
	\end{align}
	As $\varphi(1)=1$, it follows from~\eqref{eq:fo} that
	\begin{align}
	(1-p)(1- \phi_{\Omega^c}^*(1)) = \phi_\Omega(1).
	\end{align}
	Hence~\eqref{eq:foyaddd} simplifies to 
	\begin{align}
	\label{eq:foya3}
	\Prb{Y +  \sum_{i=1}^L X_i(2) = n}  \sim \frac{1}{\phi_{\Omega}(1)}\Pr{\xi=n}. \nonumber
	\end{align}
	This completes the verification of~\eqref{eq:step1}. Using Proposition~\ref{pro:binomial}, it follows that
	\begin{align}
	\Pr{\zeta=n} &\sim \frac{1}{1-p} \Prb{Y + \sum_{i=1}^L X_i(2) = \lfloor n/(1-p)\rfloor} \\
	&\sim \phi_\Omega(1)^{-1}\frac{1}{1-p} \Pr{\xi= \lfloor n/(1-p)\rfloor} \nonumber \\
	&\sim \frac{(1-p)^\alpha}{\phi_\Omega(1)} \Pr{\xi = n}. \nonumber
	\end{align}
	This completes the proof.
\end{proof}

\begin{lemma}
	\label{le:capsule0omega}
	For integers $k \ge 1$ we consider the list of integers
	\[
	(Y^{k}, \bm{X}^{k}_1, \ldots, \bm{X}_{L_k}^{k})  := \left( (Y, \bm{X}_1, \ldots, \bm{X}_L) \,\, \bigg\vert \,\, \mathrm{Bin}\left(Y + \sum_{i=1}^L X_i(2),1-p\right) = k\right).
	\]
	with $\bm{X}_i^{k} = (X_i^{k}(1),X_i^{k}(2))$ for all $1 \le i \le L_k$.
	We form the sequence
	\[
	P^{k} := (Y^{k}, X_1^{k}(1),X_1^{k}(2), \ldots,*,\ldots,X_{L_k}^{k}(1), X_{L_k}^{k}(2))
	\]
	by replacing the largest coefficient of the original list with a $*$-placeholder. Let $L'$ denote an independent copy of $L$ and let $\bm{X}'_i = (X_i'(1), X_i'(2))$, $i \ge 1$, denote independent copies of $\bm{X}$.  
	\begin{enumerate}[\qquad a)]
		\item 	If $\Omega^c$ is finite, then 
		\begin{align}
		\label{eq:limb1}
		P^{k} \convdis (*, X_1(1),X_1(2), \ldots, X_{L}(1), X_{L}(2))
		\end{align}
		as $k$ tends to infinity.
		\item If $\Omega$ is finite, then
		\begin{multline}
		\label{eq:limb2}
		P^{k}\convdis 	(Y, X_1(1),X_1(2), \ldots, X_{L}(1), X_{L}(2), \mathrm{Geom}(p),*, \\X'_1(1),X'_1(2), \ldots, X'_{L'}(1), X'_{L'}(2))
		\end{multline}
		as $k$ tends to infinity.
		\item There are constants $x_0, \tilde{k}, C,c>0$  such that for all $k \ge \tilde{k}$ and $x \in \ndZ$ with $\Pr{\xi = k+x}>0$  it holds that
		\begin{multline}
		\label{eq:bob}
		\mathbb{P}\Bigg(\max(Y,  X_1(1) + X_1(2)+1, \ldots, X_{L}(1) + X_{L}(2)+1)=k \,\,\Bigg\vert\,\, Y + \sum_{i=1}^L X_i(2) = k+x\Bigg) \\\le 
		C \frac{\Pr{\xi= k}}{\Pr{\xi = k+x}}\left( \one_{x\le x_0} \exp(-c|x|) + \one_{x> x_0} \exp\left(-\frac{cx}{k}\right)\Pr{\xi=x}  \right)
		\end{multline}
	\end{enumerate}
\end{lemma}
\begin{proof}
	Let us first consider the vectors 
	$
	(Y^{<k>}, \bm{X}^{<k>}_1, \ldots, \bm{X}_{L_{<k>}}^{<k>})
	$
	and $P^{<k>}$ that are defined analogously, but by conditioning on $Y + \sum_{i=1}^L X_i(2) = k$ instead. Thus $P^k$ is a mixture of $(P^{<i>})_{i \ge 0}$.
	
	\emph{Claim a):} If $\Omega^c$ is finite, then it follows analogously as for~		\eqref{eq:foya1} that
	\begin{align}
	\label{eq:c1}
	P^{<k>} \convdis (*, X_1(1),X_1(2), \ldots, X_{L}(1), X_{L}(2)).
	\end{align}
	Using Proposition~\ref{pro:binomial} it follows that
	\[
	P^{k} \convdis (*, X_1(1),X_1(2), \ldots, X_{L}(1), X_{L}(2)).
	\]
	
	\emph{Claim b)}  Suppose that $\Omega$ is finite. Consider the list \[(Y^{<k>}, X_1^{<k>}(2), \ldots, *, \ldots, X_{L_{<k>}}^{<k>}(2))\] obtained from~\[(Y^{<k>}, X_1^{<k>}(2), \ldots, X_{L_{<k>}}^{<k>}(2))\] by deleting the largest entry by a $*$-placeholder. Using a general result concerning the asymptotic behaviour of random Gibbs partitions that exhibit a giant component~\cite[Thm. 3.1]{doi:10.1002/rsa.20771} it follows that
	\[
	(Y^{<k>}, X_1^{<k>}(2), \ldots, *, \ldots, X_{L_{<k>}}^{<k>}(2)) \convdis (Y, X_1(2), \ldots, X_{L}(2),*,X'_1(2), \ldots, X'_{L'}(2)).
	\]
	Using Proposition~\ref{pro:pro} it follows that
	\begin{align}
	\label{eq:c2}
	P^{<k>} \convdis 	(Y, X_1(1),X_1(2), \ldots, X_{L}(1), X_{L}(2), \mathrm{Geom}(p),*,X'_1(1),X'_1(2), \ldots, X'_{L'}(1), X'_{L'}(2)).
	\end{align}
	Using Proposition~\ref{pro:binomial} we may deduce
	\[
	P^k \convdis 	(Y, X_1(1),X_1(2), \ldots, X_{L}(1), X_{L}(2), \mathrm{Geom}(p),*,X'_1(1),X'_1(2), \ldots, X'_{L'}(1), X'_{L'}(2))
	\]
	as $k$ tends to infinity.
	
	\emph{Claim c)} 
	To simplify notation, we set $Z := 	Y + X_1(2) + \ldots + X_L(2)$. We start with the case $|\Omega^c| < \infty$. Choose $\tilde{k} > \sup (\Omega^c)$ such that $\Pr{\xi=k}>0$ for all $k \ge \tilde{k}$. Using $X(1) + X(2) + 1 \in \Omega^c$, it follows that for $k \ge \tilde{k}$
	\begin{align*}
	\Pr{\max(Y,1 + X_1(1) + X_1(2), \ldots,1 + X_L(1) + X_L(2)) = k , Z  = k+x } =0
	\end{align*}
	if $x < 0$. For $x \ge 0$ it holds that
	\begin{align*}
	&\Pr{\max(Y,1 + X_1(1) + X_1(2), \ldots,1 + X_L(1) + X_L(2)) = k \mid Z  = k+x }  \\
	&= \Pr{Y=k \mid Z = k+x} \\
	&= \frac{\Pr{Y=k} \Prb{\sum_{i=1}^L X_i(2) = x}}{\Pr{Z=k+x}}.
	\end{align*}
	Using Lemma~\ref{le:regvar} and the fact that $\sum_{i=1}^L X_i(2)$ has finite exponential moments, it follows that there are constants $C_1,c_1>0$ with
	\begin{align*}
	\frac{\Pr{Y=k} \Prb{\sum_{i=1}^L X_i(2) = x}}{\Pr{Z=k+x}} \le  C_1	\frac{\Pr{\xi=k} \exp(-c_1 x)}{\Pr{\xi=k+x}} .
	\end{align*}
	
	It remains to treat the case $|\Omega|<\infty$. Chose $\tilde{k} > \sup(\Omega)$ such that $\Pr{\xi=k}>0$ for all $k \ge \tilde{k}$.  Using $Y \in \Omega$, it follows that  there is a constant $C_0>0$ such that for $k \ge \tilde{k}$
	\begin{align*}
	&\Pr{\max(Y,1 + X_1(1) + X_1(2), \ldots,1 + X_L(1) + X_L(2)) = k \mid Z  = k+x }  \\
	&\quad =  \Pr{\max(X_1(1) + X_1(2), \ldots, X_L(1) + X_L(2)) = k-1 \mid Z  = k+x } \\
	&\quad \le  \frac{\Pr{X(1) + X(2) +1 = k}}{\Pr{Z=k+x}} \\&\quad \quad \,\,\,\sum_{\ell \ge 1} \ell \Pr{L=\ell} 
	\Prb{\sum_{i=1}^{\ell -1} X_i(2) = x+1+X(1)- Y, \max_{1 \le i < \ell}(X_i(1) + X_i(2)) \le k-1} \\
	&\quad \le C_0  \frac{\Pr{\xi= k}}{\Pr{\xi=k+x}} \\&\quad \quad \,\,\,\sum_{\ell \ge 1} \ell \Pr{L=\ell} 
	\Prb{\sum_{i=1}^{\ell -1} X_i(2) = x+1+X(1)- Y, \max_{1 \le i < \ell}(X_i(1) + X_i(2)) \le k-1}. 
	\end{align*}
	Choose some constant $x_0>0$ such that $\Pr{\xi = s} >0$ for all integers $s > x_0  - \sup(\Omega)$. For $x<x_0$ (including explicitly negative values of $x$), the fact that $Y$ is bounded and $X(1)$ has finite exponential moments entails that the expression in the last line admits a bound of the form $C_2 \exp(-c_2|x|)$ for some constants $C_2, c_2>0$ that do not depend on $x$ or $k$.

	For $x > x_0$, note that all summands with $(\ell-1)(k-1) < x+1 - \sup(\Omega)$ equal zero. Hence it suffices to sum over integers $\ell$ with $\ell > x/k -1$. Applying the bound \cite[Thm. 4.11]{MR3097424}  and using the fact that $1 +X(1) - Y$ has finite exponential moments it follows that for any $\epsilon>0$ there are constants $C_3, c_3, \ell_0 >0$ such that for all $\ell \ge \ell_0$ 
	\begin{align*}
	\Prb{\sum_{i=1}^{\ell -1} X_i(2) = x+1+X(1)- Y} &\le c_3(1+ \epsilon)^\ell \Pr{X_1(2) = x+1+X(1)- Y} \\
	&\le C_3 (1 + \epsilon)^\ell \Pr{\xi = x}.
	\end{align*}
	Using that $L$ has finite exponential moments and taking $\epsilon>0$ small enough it follows that
	\begin{multline*}
	\sum_{\ell \ge 1} \ell \Pr{L=\ell} 
	\Prb{\sum_{i=1}^{\ell -1} X_i(2) = x+1+X(1)- Y, \max_{1 \le i < \ell}(X_i(1) + X_i(2)) \le k-1} \\
	C_4 \exp(-c_4 x/k) \Pr{\xi = x}.
	\end{multline*}
	
	Combining the two bounds for the cases $x \le x_0$ and $x>x_0$, Inequality~\eqref{eq:bob} follows. This concludes the proof.
\end{proof}



\subsection{Subcriticality}

Rizzolo~\cite[Lem. 6]{MR3335013} calculated moments of $\zeta$ using the representation
\begin{align}
\label{eq:yoooo}
Y + \sum_{i=1}^L X_i(2) \eqdist \left( 1 + \sum_{i=1}^{N(\xi)} (\xi_i-1) \,\,\bigg\vert\,\, N(\xi) \le \tau_{-1}(\xi) \right)
\end{align}
with
\begin{align}
N(\xi) = \inf\{k \ge 1 \mid \xi_k \in \Omega\}
\end{align}
and
\begin{align}
\tau_{-1}(\xi) = \inf \left\{k \ge 1 \,\,\bigg\vert\,\, \sum_{i=1}^k (\xi_i -1) = -1\right\}.
\end{align}
Specifically,~\cite[Lem. 6]{MR3335013} shows that if $\Ex{\xi}=1$ then $\Ex{\zeta}=1$. The arguments may be copied almost verbatim to calculate $\Ex{\zeta}$ if $\xi$ is subcritical:
\begin{proposition}
	\label{pro:subcrit}
	It holds that
	\begin{align}
	\Ex{\zeta} = 1 + (1-p)\frac{\Ex{\xi}-1}{\Pr{\xi \in \Omega}}<1.
	\end{align}
\end{proposition}
\begin{proof}
	First,~\eqref{eq:blowup} and Wald's first equation entails
	\begin{align}
	\label{eq:yoo}
	\Ex{\zeta} &= (1-p)\Exb{1 + \sum_{i=1}^{N(\xi)} (\xi_i-1) \,\,\bigg\vert\,\, N(\xi) \le \tau_{-1}(\xi)} \nonumber \\
	&= 1-p + \Exb{\sum_{i=1}^{N(\xi)} (\xi_i-1) , N(\xi) \le \tau_{-1}(\xi)}. 
	\end{align}
	Again by Wald's first equation, it holds that
	\begin{align}
	\Exb{\sum_{i=1}^{N(\xi)} (\xi_i-1)} = \Ex{N(\xi)}(\Ex{\xi} -1) 
	= 	\frac{\Ex{\xi}-1}{\Pr{\xi \in \Omega}}.
	\end{align}
	Using the strong Markov property of $(\xi_i)_{i \ge 1}$ at the stopping time $\tau_{-1}(\xi)$ it follows that
	\begin{align}
	\Exb{\sum_{i=1}^{N(\xi)} (\xi_i-1), N(\xi)>\tau_{-1}(\xi)} &= \Pr{N(\xi)>\tau_{-1}(\xi)}\left(-1 + \Exb{\sum_{i=1}^{N(\xi)} (\xi_i-1)}\right) \nonumber \\
	&= p\left(-1 + \frac{\Ex{\xi}-1}{\Pr{\xi \in \Omega}} \right).
	\end{align}
	Hence
	\begin{align}
	\Exb{\sum_{i=1}^{N(\xi)} (\xi_i-1), N(\xi) \le \tau_{-1}(\xi)}  &=  \Exb{\sum_{i=1}^{N(\xi)} (\xi_i-1)} - \Exb{\sum_{i=1}^{N(\xi)} (\xi_i-1), N(\xi)>\tau_{-1}(\xi)}  \nonumber \\
	&= p + (1-p)\frac{\Ex{\xi}-1}{\Pr{\xi \in \Omega}}.
	\end{align}
	Hence
	\begin{align}
	\label{eq:zeta}
	\Ex{\zeta} &= 1 + (1-p)\frac{\Ex{\xi}-1}{\Pr{\xi \in \Omega}}.
	\end{align}
	As $\Ex{\xi}<1$, it follows that $\Ex{\zeta}<1$.
\end{proof}

In~\cite[Lem. 6]{MR3335013} it was also shown that if $\Ex{\xi}=1$ and $\Va{\xi}<\infty$, then the variance of $\zeta$ equals $(1-p)^2\Va{\xi}/\Pr{\xi \in \Omega}$. Arguing similarly we may also compute $\Va{\zeta}$ if $\xi$ is subcritical instead:

\begin{proposition}
	\label{pro:var}
	The offspring distribution $\xi$ has finite variance if and only if $\zeta$ has finite variance. If this is the case, then
	\begin{align}
	\label{eq:vazeta}
	\Va{\zeta} =&  \frac{(1-p)^2(\Ex{\xi^2}-1)-(1-p)p(1-\Ex{\xi})}{\Pr{\xi \in \Omega}} \nonumber \\&+ \frac{(1-p)^2(1 -\Ex{\xi})( 1 - \Ex{\xi} - 2 \Ex{\xi, \xi \in \Omega})}{\Pr{\xi \in \Omega}^2} .
	\end{align}
\end{proposition}

We will require Proposition~\ref{pro:var} later in order to compute the slowly varying function $g_\Omega(n)$ of Theorem~\ref{te:main1}.
Note that if we formally set $p=0$, then expression for $\Va{\zeta}$ in~\eqref{eq:vazeta} simplifies to the expression of the variance in~\eqref{eq:fvar}. The latter was obtained  directly from the generating functions using a computer algebra system. Moreover, if we formally set $\Ex{\xi}=1$, then~\eqref{eq:vazeta} simplifies to the expression $(1-p)^2\Va{\xi}/\Pr{\xi \in \Omega}$ obtained in~\cite[Lem. 6]{MR3335013} in this case.

\begin{proof}[Proof of Proposition~\ref{pro:var}]
	We know from~\eqref{eq:zetadensity} that the variance of $\zeta$ is finite if and only if the variance of $\xi$ is finite
	
	Suppose that $\Va{\xi}<\infty$. For now, set $Z := Y + \sum_{i=1}^L X_i(2)$. Equations~\eqref{eq:blowup},~\eqref{eq:yoo}, xand~\eqref{eq:yoooo} entail that
	\begin{align}
	\label{eq:corona1}
	\Va{\zeta} &= \Ex{\zeta^2} - \Ex{\zeta}^2 \nonumber \\
	&= \Ex{Z} p (1-p) + \Ex{Z^2}(1-p)^2 - \Ex{\zeta}^2  \nonumber \\
	&= \Ex{\zeta} p + \Ex{Z^2}(1-p)^2 - \Ex{\zeta}^2.
	\end{align}
	From~\eqref{eq:yoooo} and~\eqref{eq:yoo} it follows that
	\begin{align}
	\Ex{Z^2} &= 1 + 2 \Exb{\sum_{i=1}^{N(\xi)} (\xi_i-1) \,\,\bigg\vert\,\, N(\xi) \le \tau_{-1}(\xi)} +\Exb{ \left(\sum_{i=1}^{N(\xi)} (\xi_i-1) \right)^2 \,\,\bigg\vert\,\, N(\xi) \le \tau_{-1}(\xi)} \\
	&= \frac{2\Ex{\zeta}}{1-p} -1 + \Exb{ \left(\sum_{i=1}^{N(\xi)} (\xi_i-1) \right)^2 \,\,\bigg\vert\,\, N(\xi) \le \tau_{-1}(\xi)}. \nonumber
	\end{align}
	Combining this with~\eqref{eq:corona1}, it follows that
	\begin{align}
	\label{eq:lasttime}
	\Va{\zeta} = (2-p)\Ex{\zeta} - (1-p)^2 - \Ex{\zeta}^2 + (1-p)\Exb{ \left(\sum_{i=1}^{N(\xi)} (\xi_i-1) \right)^2 , N(\xi) \le \tau_{-1}(\xi)}.
	\end{align}
	\eqref{eq:zeta}
	With $S_n = \sum_{i=1}^n \xi_i$ and Equation~\eqref{eq:zeta}, this simplifies to
	\begin{align}
	\label{eq:comeoooooon}
	\Va{\zeta} =& (1-p)\left( p + p\frac{1 - \Ex{\xi}}{\Pr{\xi \in \Omega}} -(1-p)\frac{(1- \Ex{\xi})^2}{\Pr{\xi \in \Omega}^2} \right) \nonumber \\ &+ (1-p)^2 \Exb{ (S_{N(\xi)} - N(\xi))^2 , N(\xi) \le \tau_{-1}(\xi)}
	\end{align}
	
	Our next step is to calculate the unconditioned expectation $\Exb{(S_{N(\xi)} - N(\xi))^2}$.
	Wald's second equation yields
	\begin{align}
	\label{eq:ted}
	\Exb{ \left(S_{N(\xi)} - \Ex{\xi}N(\xi) \right)^2 } = \Va{\xi} \Ex{N(\xi)} = \frac{\Va{\xi}}{\Pr{\xi \in \Omega}}.
	\end{align}
	Moreover, 
	\begin{multline}
	\label{eq:diff}
	\Exb{ \left(S_{N(\xi)} - N(\xi) \right)^2 } - \Exb{ \left(S_{N(\xi)} - \Ex{\xi}N(\xi) \right)^2 } \\= 2(\Ex{\xi}-1) \Ex{N(\xi) S_{N(\xi)} } + (1 - \Ex{\xi}^2) \Ex{N(\xi)^2}
	\end{multline}
	An elementary calculation shows that
	\begin{align}
	\label{eq:di1}
	\Ex{N(\xi)^2} = \frac{1 + \Pr{\xi \notin \Omega}}{\Pr{\xi \in \Omega}^2}.
	\end{align}
	Moreover, using the definition of $N(\xi)$,
	\begin{align}
	&\Ex{N(\xi) S_{N(\xi)} } \nonumber \\&= \Ex{N(\xi) S_{N(\xi)}, \xi_1 \in \Omega }  + \Ex{N(\xi) S_{N(\xi)}, \xi_1 \notin \Omega } \nonumber  \\
	&= \Ex{\xi, \xi \in \Omega} + \Pr{\xi \notin \Omega} \left(\Ex{N(\xi) S_{N(\xi)} } + \Ex{S_{N(\xi)}} \right) + \Ex{\xi, \xi \notin \Omega}\left(1 + \Ex{N(\xi)} \right)  
	\end{align}
	Using Wald's first equation, it follows that
	\begin{align}
	\label{eq:di2}
	\Pr{\xi \in \Omega}\Ex{N(\xi) S_{N(\xi)} } &= \Ex{\xi} + \Ex{N(\xi)} \Ex{\xi, \xi \notin \Omega} + \Ex{N(\xi)}\Ex{\xi} \Pr{\xi \notin \Omega} \\
	&= \Ex{\xi} + \frac{\Ex{\xi, \xi \notin \Omega}}{\Pr{\xi \in \Omega}} + \frac{\Ex{\xi} \Pr{\xi \notin \Omega}}{\Pr{\xi \in \Omega}} \nonumber \\
	&= \frac{\Ex{\xi} +\Ex{\xi, \xi \notin \Omega}}{\Pr{\xi \in \Omega}} \nonumber
	\end{align}
	Applying~\eqref{eq:di1} and~\eqref{eq:di2} to~\eqref{eq:diff} yields
	\begin{align}
	&\Exb{ \left(S_{N(\xi)} - N(\xi) \right)^2 } - \Exb{ \left(S_{N(\xi)} - \Ex{\xi}N(\xi) \right)^2 } \nonumber \\&\qquad\qquad= 2(\Ex{\xi}-1)\frac{\Ex{\xi} +\Ex{\xi, \xi \notin \Omega}}{\Pr{\xi \in \Omega}^2} + (1 - \Ex{\xi}^2) \frac{1 + \Pr{\xi \notin \Omega}}{\Pr{\xi \in \Omega}^2}  
	\end{align}
	Using~\eqref{eq:ted}, it follows that
	\begin{align}
	\label{eq:theexp}
	\Exb{(S_{N(\xi)} - N(\xi))^2} = \frac{\Ex{\xi^2}-1}{\Pr{\xi \in \Omega}} + \frac{2(\Ex{\xi}-1)^2 - 2(1-\Ex{\xi})\Ex{\xi, \xi \in \Omega}  }{\Pr{\xi \in \Omega}^2}.
	\end{align}
	Using the strong Markov property of $(\xi_i)_{i \ge 1}$ at the stopping time $\tau_{-1}(\xi)$ it follows that
	\begin{align}
	\Exb{ \left(\sum_{i=1}^{N(\xi)} (\xi_i-1) \right)^2 , N(\xi) > \tau_{-1}(\xi)} &= \Pr{N(\xi) > \tau_{-1}(\xi)} \Exb{ \left(-1 +\sum_{i=1}^{N(\xi)} (\xi_i-1) \right)^2 } \nonumber  \\ 
	&= p\left(1- 2	\frac{\Ex{\xi}-1}{\Pr{\xi \in \Omega}} + \Exb{(S_{N(\xi)} - N(\xi))^2} \right). 
	\end{align}
	Hence
	\begin{align}
	\Exb{ (S_{N(\xi)} - N(\xi))^2 , N(\xi) \le \tau_{-1}(\xi)} &= (1-p)\Exb{(S_{N(\xi)} - N(\xi))^2} - p \left( 1- 2	\frac{\Ex{\xi}-1}{\Pr{\xi \in \Omega}}\right).
	\end{align}
	Using Equation~\eqref{eq:theexp} it follows that
	\begin{multline}
	\label{eq:almooooost}
	\Exb{ (S_{N(\xi)} - N(\xi))^2 , N(\xi) \le \tau_{-1}(\xi)} \\= -p + \frac{(1-p)(\Ex{\xi^2}-1) - 2p(1 - \Ex{\xi})}{\Pr{\xi \in \Omega}} + (1-p)\frac{2(\Ex{\xi}-1)^2 - 2(1-\Ex{\xi})\Ex{\xi, \xi \in \Omega}  }{\Pr{\xi \in \Omega}^2} .
	\end{multline}
	
	Having established Equation~\eqref{eq:almooooost} we may now evaluate Equation~\eqref{eq:comeoooooon}, yielding
	\begin{align}
	\Va{\zeta} =&\, (1-p)\left( p + p\frac{1 - \Ex{\xi}}{\Pr{\xi \in \Omega}} -(1-p)\frac{(1- \Ex{\xi})^2}{\Pr{\xi \in \Omega}^2} \right) \nonumber ¸\\ &+ (1-p) \left( -p + \frac{(1-p)(\Ex{\xi^2}-1) - 2p(1 - \Ex{\xi})}{\Pr{\xi \in \Omega}}\right) \nonumber \\
	&+ (1-p)^2\frac{2(\Ex{\xi}-1)^2 - 2(1-\Ex{\xi})\Ex{\xi, \xi \in \Omega}  }{\Pr{\xi \in \Omega}^2}  \nonumber \\
	=&\,   \frac{(1-p)^2(\Ex{\xi^2}-1)-(1-p)p(1-\Ex{\xi})}{\Pr{\xi \in \Omega}} \nonumber \\&+ \frac{(1-p)^2(1 -\Ex{\xi})( 1 - \Ex{\xi} - 2 \Ex{\xi, \xi \in \Omega})}{\Pr{\xi \in \Omega}^2}  
	\end{align}
\end{proof}


Note that using~\eqref{eq:gwt}, Proposition~\ref{pro:subcrit} and Equation~\eqref{eq:zetadensity} it follows that
\begin{align}
\label{eq:lomegat2}
\Pr{L_\Omega(\mT) = n} &= \Pr{L_\Omega(\mT) >0}\Pr{|\mA| = n} \nonumber \\
&\sim  \frac{(1-p)^{\alpha+1}}{\Pr{\xi \in \Omega}} \frac{f(n)}{(n(1 - \Ex{\zeta}))^{1+\alpha}} \nonumber \\
&= \frac{\Pr{\xi \in \Omega}^\alpha}{(1- \Ex{\xi})^{1 + \alpha}} f(n) n^{-1-\alpha}. 
\end{align}

\subsection{Proof of the main theorems for the case $0 \notin \Omega$}
Throughout this section we assume that the random tree $\mT_n^\Omega$ is constructed from the tree $\mA_n$ by the  blow-up procedure described in Section~\ref{sec:blowup}.

\begin{proof}[Proof of Theorem~\ref{te:main2} for the case $0 \notin \Omega$]
	The offspring distribution $\zeta$ has a regularly varying density by Lemma~\ref{le:regvar} and satisfies $\Ex{\zeta}<1$ by Proposition~\ref{pro:subcrit}. As we already proved Theorem~\ref{te:main2} for the case $\Omega=\ndN_0$, it follows that for any sequence $(t_n)_{n \ge1}$ of positive integers satisfying $t_n = o(n)$ and $t_n \to \infty$
	\begin{align}
	\label{eq:ddnotin}
	\left(F_0(\mA_n), (F_i(\mA_n))_{1 \le i \le \Delta(\mA_n) - t_n}, \one_{\sum_{i = \Delta(\mA_n) - t_n }^{\Delta(\mA_n)} |F_i(\mA_n)| \ge \frac{2t_n}{1 - \Ex{\zeta}} } \right) \atv \left(\mA^\bullet, (\mA^i)_{1 \le i \le \Delta_{\langle n \rangle}^\mA - t_n},0\right),
	\end{align}
	with $(\mA_i)_{i \ge 1}$ denoting a family of independent $\zeta$-Galton--Watson trees, $\mA^\bullet$ the analogue of $\mT^\bullet$ that is constructed with $\zeta$ instead of $\xi$, and
	\begin{align}
	\Delta_{\langle n \rangle}^\mA :=  \sup \left\{ d \ge 1 \,\,\bigg\rvert\,\, |\mA^\bullet| +  \sum_{i=1}^d |\mA^i| \le n  \right\}.
	\end{align}
	Let $v$ denote the lexicographically first vertex with maximum outdegree in~$\mA_n$. The result of replacing the largest component from the decoration of $v$ by a $*$-placeholder is distributed like $P^{\Delta(\mA_n)}$, with $(P^k)_{k \ge 1}$ assumed to be independent from $\mA_n$. As $\Delta(\mA_n) \convdis \infty$, it follows from Equations~\eqref{eq:limb1} and~\eqref{eq:limb2} that 
	\begin{align}
	\label{eq:doh1}
	P^{\Delta(\mA_n)} \convdis (*, X_1(1),X_1(2), \ldots, X_{L}(1), X_{L}(2))
	\end{align}
	if $\Omega^c$ is finite, and 
	\begin{multline}
	\label{eq:doh2}
	P^{\Delta(\mA_n)} \convdis (Y, X_1(1),X_1(2), \ldots, X_{L}(1), X_{L}(2), \mathrm{Geom}(p),*,\\X'_1(1),X'_1(2), \ldots, X'_{L'}(1), X'_{L'}(2))
	\end{multline}
	if $\Omega$ is finite.
	
	Having~\eqref{eq:ddnotin},~\eqref{eq:doh1}, and~\eqref{eq:doh2} at hand, the proof may now be completed in an entirely analogous fashion as in the case $0 \in \Omega$.
\end{proof}

We are going to need the some preliminary observations before proceeding to prove Theorem~\ref{te:main1}. The offspring distribution $\zeta$ has a regularly varying density satisfies $\Ex{\zeta}<1$, see Lemma~\ref{le:regvar} and Proposition~\ref{pro:subcrit}. We already proved Theorem~\ref{te:main2} for the case $\Omega=\ndN_0$, hence there is a slowly varying function $g_A$ with
\begin{align}
\label{eq:base2}
\Pr{\Delta(\mA_n) = N} = \frac{1}{g_A(n)n^{1/\theta}}\left(h\left(\frac{ (1-\Ex{\zeta})n - N}{g_A(n)n^{1/\theta}}   \right) + o(1)\right)
\end{align}
uniformly for all $N \in \ndZ$. Using the asymptotic expression of $\Pr{\zeta=n}$ from Lemma~\ref{le:regvar} it follows analogously as in Proposition~\ref{pro:gomega} that
\begin{align}
\label{eq:gA}
g_A(n) = \begin{cases}
\sqrt{\frac{\Va{{\zeta}}}{2}}, \quad &\Va{\xi}<\infty \\
\frac{(1-p)g(n)}{\Pr{\xi \in \Omega}^{1/\theta}}, \quad &\Va{\xi} = \infty.
\end{cases}
\end{align}

\begin{lemma}
	\label{le:llt5}
	Assume that $(Y, \bm{X}_1, \ldots, \bm{X}_L) $ is independent from $\Delta(\mA_n)$ and set $Z := Y + \sum_{i=1}^L X_i(2)$. Set
	\begin{align}
	g_\Omega(n) = \begin{cases}
	\frac{1}{\sqrt{2}}\left(\frac{\Ex{\xi^2}-1}{\Pr{\xi \in \Omega}} + \frac{(1- \Ex{\xi})(1 - \Ex{\xi} + 2 \Ex{\xi, \xi \in \Omega})}{\Pr{\xi \in \Omega}^2}\right)^{1/2}, \quad &\Va{\xi}<\infty \\
	\frac{g(n)}{\Pr{\xi \in \Omega}^{1/\theta}}, \quad &\Va{\xi} = \infty.
	\end{cases}
	\end{align}
	Then
	\begin{align}
	\label{eq:toprove}
	\Pr{Z= \ell \mid \mathrm{Bin}(Z,1-p) = \Delta(\mA_n)} = \frac{1}{g_\Omega(n)n^{1/\theta}} \left( h\left(\frac{ \frac{1-\Ex{\xi}}{\Pr{\xi \in \Omega}}n - \ell}{g_\Omega(n)n^{1/\theta}}   \right) + o(1)\right)
	\end{align}
	uniformly for all integers $\ell \ge 0$.
\end{lemma}
\begin{proof}
	Clearly it holds that
	\begin{align}
	\Pr{Z= \ell \mid \mathrm{Bin}(Z,1-p) = \Delta(\mA_n)} = \sum_{N=1}^n \Pr{ \Delta(\mA_n)= N} \Pr{Z= \ell \mid \mathrm{Bin}(Z,1-p) = N}.
	\end{align}
	By Proposition~\ref{pro:largedev} it follows that there are constants $\epsilon,\delta>0$  such that
	\begin{align}
	\label{eq:deldev}
	n^{1/\theta}\Pr{\Delta(\mA_n) \le \epsilon n} = O(n^{-\delta}).
	\end{align}
	Hence
	\begin{align}
	\label{eq:lastinter}
	\Pr{Z= \ell \mid \mathrm{Bin}(Z,1-p) =&\, \Delta(\mA_n)} = O(n^{-\delta-1/\theta}) \\&+ \sum_{N=\lfloor \epsilon n \rfloor}^n \Pr{ \Delta(\mA_n)= N} \Pr{Z= \ell \mid \mathrm{Bin}(Z,1-p) = N}. \nonumber
	\end{align}
	Moreover, $\Pr{Z= \ell \mid \mathrm{Bin}(Z,1-p) = N} = 0$ whenever $\ell < N$. It follows that
	\[
	\Pr{Z= \ell \mid \mathrm{Bin}(Z,1-p) = \Delta(\mA_n)} = O(n^{-\delta - 1/\theta})
	\]
	uniformly for all $\ell \le \epsilon n$. Hence it suffices to verify~\eqref{eq:toprove} for $\ell > \epsilon n$.
	Set \begin{align}
	\label{eq:expforx}
	x = \frac{N-\ell (1-p)}{\sqrt{\ell}}\end{align}
	and
	\begin{align}
	I_\ell= \{ N \in \ndN \mid \epsilon n \le N \le n, |\ell - N/(1-p)| \le \sqrt{N} \log N \}.
	\end{align}
	For $\epsilon n \le N \le n$, it follows from~\eqref{eq:strength1} and~\eqref{eq:strength2} that
	\begin{align}
	\Pr{Z = \ell \mid \mathrm{Bin}{(Z,1-p)} = N} = O(N^{-\Theta(\log N)}) = O(n^{-\Theta(\log n)})
	\end{align}
	if $N \notin I_\ell$, and 
	\begin{multline}
	\Pr{Z = \ell \mid \mathrm{Bin}{(Z,1-p)} = N} = \\	\frac{1}{\sqrt{N}} \left( o\left(\frac{1}{\max\left(1, x^2\right)}\right) + \frac{1-p}{\sqrt{2 \pi p} } \exp\left(- \frac{(1-p)^2}{2p} \left(\frac{\ell - N/(1-p)}{\sqrt{N}} \right)^2 \right) \right)
	\end{multline}
	for $N \in I_\ell$. Using ~\eqref{eq:lastinter}, it follows that 
	\begin{align}
	\Pr{Z= \ell \mid \mathrm{Bin}(Z,1-p) = \Delta(\mA_n)} = A + B
	\end{align}
	with
	\begin{align}
	A =  O(n^{-\delta-1/\theta}) +  \sum_{N \in I_\ell} \Pr{\Delta(\mA_n) = N} 	\frac{1}{\sqrt{N}}  o\left(\frac{1}{\max\left(1, x^2\right)}\right),
	\end{align}
	and
	\begin{align}
	B =  \sum_{N \in I_\ell} \Pr{\Delta(\mA_n) = N} 	\frac{1}{\sqrt{N}} \frac{1-p}{\sqrt{2 \pi p} } \exp\left(- \frac{1}{2p} \left(\frac{(1-p)\ell - N}{\sqrt{N}} \right)^2 \right).
	\end{align}
	Note that $N \in I_\ell$ entails that $\ell < n(1+o(1))/(1-p)$. Hence~\eqref{eq:toprove} holds uniformly for $\ell \ge 2n/(1-p)$ as $n$ tends to infinity, and throughout the rest of the proof we may assume that
	\begin{align}
	\label{eq:assumonell}
	\epsilon n < \ell < 2n/(1-p).
	\end{align}
	Using~\eqref{eq:base2}, it follows that 
	\begin{align}
	A = o\left( \frac{1}{g_A(n)n^{1/\theta}} \right).
	\end{align} As for $B$, the asymptotic $N \sim (1-p)\ell$ entails that uniformly for all integers $\ell$ satisfying~\eqref{eq:assumonell}
	\begin{align}
	\label{eq:latestB}
	B = (1 + o(1)) \sum_{N \in I_\ell} \Pr{\Delta(\mA_n) = N} 	\frac{1}{\sqrt{\ell}} \frac{1-p}{\sqrt{2 \pi p(1-p)} } \exp\left(- \frac{1}{2p(1-p)} \left(\frac{(1-p)\ell - N}{\sqrt{\ell}} \right)^2 \right).
	\end{align}
	By~\eqref{eq:base2} it holds that
	\begin{align}
	\Pr{\Delta(\mA_n) = N} = \frac{1}{g_A(n)n^{1/\theta}}   &\left(h\left(\frac{ (1-\Ex{\zeta})n - N}{g_A(n)n^{1/\theta}}   \right) + o(1)\right).
	\end{align}
	For the remaining part of the proof, we have to argue differently according to whether $1 < \theta < 2$ or $\theta = 2$.
	
	Let us start with the case $\theta < 2$. Then, 
	\begin{align}
	\sqrt{n} \log n = o(g_A(n) n^{1/\theta}).
	\end{align} Hence it holds uniformly for all $\ell$ satisfying~\eqref{eq:assumonell} and $N \in I_\ell$ 
	\begin{align}
	h\left(\frac{ (1-\Ex{\zeta})n - N}{g_A(n)n^{1/\theta}}   \right) = 	h\left(\frac{ \frac{1-\Ex{\zeta}}{1-p}n - \ell}{(1-p)^{-1}g_A(n)n^{1/\theta}}   \right)
	\end{align}
	Applying this to~\eqref{eq:latestB} yields
	\begin{multline}
	B = (1+o(1))\frac{1-p}{g_A(n)n^{1/\theta}} \left( h\left(\frac{ \frac{1-\Ex{\zeta}}{1-p}n - \ell}{(1-p)^{-1}g_A(n)n^{1/\theta}}   \right) + o(1)\right) \\\sum_{N \in I_\ell} \frac{1}{\sqrt{\ell}} \frac{1}{\sqrt{2 \pi p(1-p)} } \exp\left(- \frac{1}{2p(1-p)} \left(\frac{(1-p)\ell - N}{\sqrt{\ell}} \right)^2 \right).
	\end{multline}
	Using dominated convergence and the fact that $h$ is bounded, this simplifies to 
	\begin{align}
	B = \frac{1-p}{g_A(n)n^{1/\theta}} \left( h\left(\frac{ \frac{1-\Ex{\zeta}}{1-p}n - \ell}{(1-p)^{-1}g_A(n)n^{1/\theta}}   \right) + o(1)\right).
	\end{align}
	Equation~\eqref{eq:toprove} now follows using the expression of $g_A(n)$ from Equation~\eqref{eq:gA} and the expression of $\Ex{\zeta}$ from Proposition~\ref{pro:subcrit}.
	
	It remains to treat the case $\theta=2$, where
	\begin{align}
	h(z) = \exp(-z^2/4)/\sqrt{4\pi}.
	\end{align}
	Using again dominated convergence and the fact that $h$ is bounded, it follows that
	\begin{align}
	B = \frac{1-p}{g_A(n)\sqrt{n}} \Bigg( o(1) + B' \Bigg).
	\end{align}
	for
	\begin{multline}
	\label{eq:goon}
	B' =  \sum_{N \in I_\ell} \frac{1}{\sqrt{4 \pi}} \exp\left( - \frac{1}{4}\left(\frac{(1- \Ex{\zeta})n - N}{g_A(n) \sqrt{n}} \right)^2 \right) \\	\frac{1}{\sqrt{\ell}} \frac{1}{\sqrt{2 \pi p(1-p)} } \exp\left(- \frac{1}{2p(1-p)} \left(\frac{(1-p)\ell - N}{\sqrt{\ell}} \right)^2 \right).
	\end{multline}
	Using dominated convergence, it follows that for any $\epsilon_1>0$ we may select a constant $M_1>0$ sufficiently large such that the sum of all summands in this expression of $B'$ for which $N \notin (1-p) \ell \pm M_1 \sqrt{n}$ is smaller than $\epsilon_1$. If $\Va{\xi}=\infty$ (or equivalently $\Va{\zeta}=\infty$), the slowly varying function $g_A(n)$ satisfies $g_A(n) \to \infty$ as $n$ becomes large.  Thus, using dominated convergence analogously as in the case $\theta < 2$, it follows that uniformly for all integers $\ell$ satisfying~\eqref{eq:assumonell}
	\begin{align}
	B' =  \frac{1}{\sqrt{4 \pi}} \exp \left( - \frac{1}{4}\left(\frac{ \frac{1-\Ex{\zeta}}{1-p}n - \ell}{(1-p)^{-1}g_A(n)\sqrt{n}}   \right)^2 \right) + R_{\ell, n}^{(1)}
	\end{align}
	for an error term $R_{\ell, n}^{(1)}$ satisfying $|R_{\ell,n}^{(1)}| < \epsilon_1$ for all large enough $n$. As $\epsilon_1 >0 $ was arbitrary, 	Equation~\eqref{eq:toprove} now follows using the expression of $g_A(n)$ from Equation~\eqref{eq:gA} and the expression of $\Ex{\zeta}$ from Proposition~\ref{pro:subcrit}.
	
	It remains to treat the case $\Va{\xi}=\infty$, where $
	g_A(n) = \sqrt{\Va{\zeta}/2}$ may be chosen to be constant.
	Using dominated convergence it follows
	that we may select a constant $M_2>0$ sufficiently large such that the sum of all summands in the expression~\eqref{eq:goon} of $B'$ for which $N \notin  (1- \Ex{\zeta})n \pm M_2 \sqrt{n}$ or $N \notin (1-p) \ell \pm M_1 \sqrt{n}$ is smaller than $2\epsilon_1$ (for large enough $n$). This entails also that  for large enough $n$ it holds that there is a constant $M_3 := (M_1 + M_2)/(1-p)$ such that
	$B' < 2\epsilon_1$ uniformly for all $\ell$ with $\ell \notin \frac{1 - \Ex{\zeta}}{1-p}n \pm M_3 \sqrt{n}$. Now, suppose that  $\ell \in \frac{1 - \Ex{\zeta}}{1-p}n \pm M_3 \sqrt{n}$. This allows us to write
	\begin{align}
	\ell = \frac{1- \Ex{\zeta}}{1-p}n + y_\ell g_A(n) \sqrt{n}(1-p)^{-1}
	\end{align}
	with $|y_\ell|$ bounded by some constant multiple of $M_3$. Furthermore, set
	\begin{align}
	z = \frac{N - (1-p)\ell }{\sqrt{n}}
	\end{align}
	and
	\begin{align}
	h_1(z) = \frac{1}{\sqrt{2 \pi p (1- \Ex{\zeta})}} \exp\left(- \frac{1}{2p(1- \Ex{\zeta})} z^2\right).
	\end{align}
	Thus, choosing the constants $M_i$ to be large enough, it follows that
	\begin{align}
	B' &= (1 + o(1))\sum_{N \in I_\ell} \frac{1}{\sqrt{n}} h(z/g_A(n) + y_\ell) h_1(z) \nonumber\\
	&= \int_{-\infty}^\infty  h(z/g_A(n) + y_\ell) h_1(z) \,\mathrm{d}z + R_{\ell,n}^{(2)}, 
	\end{align}
	for an error term $R_{\ell,n}^{(2)}$ satisfying $|R_{\ell,n}^{(2)}| <  3 \epsilon_1$ for large enough $n$. Now, \[y \mapsto \int_{-\infty}^\infty  h(z/g_A(n) + y_\ell) h_1(z) \,\mathrm{d}z\] is the density of \[A - B/g_A(n)\] with $A, B$   variables with distributions \[A \eqdist \cN(0,2)\] and \[B \eqdist \cN(0, p(1- \Ex{\zeta})).\] As \[A - B/g_A(n) \eqdist \cN(0, \sigma_C^2)\] for 
	\begin{align}
	\sigma_C^2 =  2 + p(1-\Ex{\zeta}) / g_A^2(n),
	\end{align} it follows that for all $y \in \ndR$
	\begin{align}
	\int_{-\infty}^\infty  h(z/g_A(n) + y) h_1(z) \,\mathrm{d}z &= \frac{1}{\sqrt{2\pi} \sigma_C} \exp\left(-\frac{y^2}{2 \sigma_C^2}\right) \nonumber\\
	&= \frac{\sqrt{2}}{\sigma_C}h\left( \frac{\sqrt{2}y}{\sigma_C}  \right).
	\end{align}
	As $\epsilon_1>0$ was arbitrary, it follows that
	\begin{multline}
	\Pr{Z= \ell \mid \mathrm{Bin}(Z,1-p) = \Delta(\mA_n)} = \\\frac{\sqrt{2}(1-p)}{\sigma_C g_A(n)n^{1/\theta}} \left( h\left(\frac{ \frac{1-\Ex{\zeta}}{1-p}n - \ell}{(1-p)^{-1} 2^{-1/2}\sigma_C g_A(n)n^{1/\theta}}   \right) + o(1)\right).
	\end{multline}
	The expression of $\Ex{\zeta}$ from Proposition~\ref{pro:subcrit} implies that
	\begin{align}
	\frac{1-\Ex{\zeta}}{1-p} = \frac{1- \Ex{\xi}}{\Pr{\xi \in \Omega}}.
	\end{align}
	By Equation~\eqref{eq:gA} we may set $g_A(n) = \sqrt{\frac{\Va{\zeta}}{2}}$ for all $n$. Hence, using the expression of $\Va{\zeta}$ from Proposition~\ref{pro:var}, it follows that
	\begin{align}
	\frac{\sigma_C g_A(n)}{\sqrt{2}(1-p)}  &=  \sqrt{\frac{\Va{\zeta}}{2}} \frac{1}{1-p} \sqrt{1 + \frac{p(1- \Ex{\zeta})}{\Va{\zeta}}}\nonumber \\
	&= \frac{1}{\sqrt{2}} \sqrt{ \frac{\Va{\zeta}}{(1-p)^2} + \frac{p(1-\Ex{\zeta})}{(1-p)^2}  } \nonumber \\
	&= \frac{1}{\sqrt{2}} \sqrt{ \frac{\Va{\zeta}}{(1-p)^2} + \frac{p(1-\Ex{\xi})}{(1-p)\Pr{\xi \in \Omega}}  } \nonumber \\
	&= \frac{1}{\sqrt{2}}\left(\frac{\Ex{\xi^2}-1}{\Pr{\xi \in \Omega}} + \frac{(1- \Ex{\xi})(1 - \Ex{\xi} + 2 \Ex{\xi, \xi \in \Omega})}{\Pr{\xi \in \Omega}^2}\right)^{1/2}. 
	\end{align}
	Equation~\eqref{eq:toprove} now follows.
\end{proof}

Let $v^*$ denote the lexicographically first vertex  of $\mA_n$ with maximal degree. There are two steps for proving Theorem~\ref{te:main1}. The first is to locate a vertex within the blow-up of $v^*$ whose outdegree satisfies a local limit theorem as in~\eqref{eq:toprove}. The second step is to discard all possibilities for a  larger outdegree to appear anywhere in the blow-up of the entire tree $\mA_n$.

Recall from Figure~\ref{fi:blowup2} that the blow-up creates a vertebrate. Hence we may distinguish the vertices on the spine and the vertices created from attaching independent copies of $\mT^*$. We let $D_n$ denote the largest outdegree of a \emph{spine} vertex created from the blow-up of the lexicographically first vertex of $\mA_n$ with outdegree $\Delta(\mA_n)$.

\begin{lemma}
	\label{eq:lltforD}
	Uniformly for all integers $\ell \ge 0$
	\begin{align}
	\Pr{D_n = \ell} = \frac{1}{g_\Omega(n)n^{1/\theta}} \left( h\left(\frac{ \frac{1-\Ex{\xi}}{\Pr{\xi \in \Omega}}n - \ell}{g_\Omega(n)n^{1/\theta}}   \right) + o(1)\right).
	\end{align}
\end{lemma}
\begin{proof}
	We continue to assume that $(Y, \bm{X}_1, \ldots, \bm{X}_L) $ is independent from $\Delta(\mA_n)$.  Let us set
	\begin{align}
	\label{eq:defD}
	D = \max(Y,  X_1(1) + X_1(2)+1, \ldots, X_{L}(1) + X_L(2)+1)
	\end{align}
	Then
	\begin{align}
	D_n \eqdist \left( D \mid \mathrm{Bin}(Z,1-p) = \Delta(\mA_n) \right).
	\end{align}	
	Using Lemma~\ref{le:llt5}, it follows that
	\begin{align}
	\label{eq:ss1}
	&g_{\Omega}(n)n^{1/\theta} \Pr{D_n = \ell} \nonumber \\
	&= g_{\Omega}(n)n^{1/\theta} \sum_{x \ge - \ell} \Pr{Z = \ell + x \mid \mathrm{Bin}(Z, 1-p) = \Delta(\mA_n)} \Pr{D= \ell \mid Z = \ell + x} \nonumber\\
	&=  \sum_{x \ge - \ell} \left( h\left(\frac{ \frac{1-\Ex{\xi}}{\Pr{\xi \in \Omega}}n - \ell-x}{g_\Omega(n)n^{1/\theta}}   \right) + o(1)\right) \Pr{D= \ell \mid Z = \ell + x}.
	\end{align}
	This expression tends to zero for any constant $\ell$. Hence there is some sequence $\ell_n \to \infty$  so that $g_{\Omega}(n)n^{1/\theta} \Pr{D_n = \ell} \to 0$ uniformly for $\ell \le \ell_n$. 
	
	Throughout the following, we assume $\ell \ge \ell_n$. 	We observed in the limits of Equations~\eqref{eq:c1} and~\eqref{eq:c2} the emergence of giant components with a stochastically bounded remainder that admits a distributional limit. Hence there is a random non-negative integer $R$ such that for any constant integer $x$
	\begin{align}
	\Pr{D = \ell \mid Z= \ell +x} \to \Pr{R=x}.
	\end{align}
	Using that $h$ is uniformly continuous and bounded, it follows that there is a sequence of integers $x_n$ with $x_n \to \infty$ so that
	\begin{multline}
	\label{eq:ss2}
	\sum_{|x| \le x_n} \left( h\left(\frac{ \frac{1-\Ex{\xi}}{\Pr{\xi \in \Omega}}n - \ell-x}{g_\Omega(n)n^{1/\theta}}   \right) + o(1)\right) \Pr{D= \ell \mid Z = \ell + x} =  \\h\left(\frac{ \frac{1-\Ex{\xi}}{\Pr{\xi \in \Omega}}n - \ell}{g_\Omega(n)n^{1/\theta}}   \right) + o(1).
	\end{multline}
	Moreover, using Inequality~\eqref{eq:bob}, it follows that there are constants $x', c>0$ with
	\begin{align}
	\label{eq:ss3}
	&\sum_{|x| > x_n, x \ge - \ell} \left( h\left(\frac{ \frac{1-\Ex{\xi}}{\Pr{\xi \in \Omega}}n - \ell-x}{g_\Omega(n)n^{1/\theta}}   \right) + o(1)\right) \Pr{D= \ell \mid Z = \ell + x} \nonumber \\
	&\quad \qquad = O(1) \sum_{|x| > x_n, x \ge -\ell} \Pr{D= \ell \mid Z = \ell + x} \nonumber \\
	&\quad \qquad \le O(1) \sum_{|x| > x_n, x \ge - \ell} \frac{\Pr{\xi= \ell}}{\Pr{\xi = \ell+x}}\left( \one_{x\le x'} \exp(-c|x|) + \one_{x> x'} \exp\left(-\frac{c|x|}{\ell}\right)\Pr{\xi=x}  \right) \nonumber \\
	&\quad \qquad = o(1). 
	\end{align}
	Combining~\eqref{eq:ss1},~\eqref{eq:ss2}, and~\eqref{eq:ss3}, it follows that
	\begin{align}
	g_{\Omega}(n)n^{1/\theta} \Pr{D_n = \ell} =  h\left(\frac{ \frac{1-\Ex{\xi}}{\Pr{\xi \in \Omega}}n - \ell}{g_\Omega(n)n^{1/\theta}}   \right) + o(1).
	\end{align}
\end{proof}

This concludes the first step. Now we prepare to eliminate the possibility for a vertex with larger degree than $D_n$ to appear in $\mT_n^\Omega$. First, we need some rough deviation bounds.

\begin{lemma}
	\label{le:devD}
	There are constants $C,\delta>0$ such that
	\begin{align*}
	\Pr{ D_n < n / \log^2 n \quad \mathrm{ or } \quad  D_n > Cn} = O(n^{-1/\theta - \delta}).
	\end{align*}
\end{lemma}
\begin{proof}
	By~\eqref{eq:deldev} there are constant $\epsilon,\delta>0$  such that
	\begin{align}
	n^{1/\theta}\Pr{\Delta(\mA_n) \le \epsilon n} = O(n^{-\delta}).
	\end{align}
	As $\Delta(\mA_n)<n$, it follows by~\eqref{eq:devi} that  for some constants $c_1,C_1>0$
	\begin{align}
	\Pr{ Z < c_1n \quad \mathrm{ or } \quad  Z > C_1n \mid \mathrm{Bin}(Z,1-p) = \Delta(\mA_n)} = O(n^{-\delta}).
	\end{align}
	With $D$ as in~\eqref{eq:defD}, it follows that  that uniformly for all  $d \in [c_1n, C_1n]$,
	\begin{align}
	\Pr{D \le d/\log^2 n \quad \mathrm{or} \quad D > 2 C_1 \mid Z = d} \le n^{-\Theta(\log n)}.
	\end{align}
	Thus
	\begin{align}
	\Pr{ D < n / \log^2 n \quad \mathrm{ or } \quad  D > 2 C_1 n \mid \mathrm{Bin}(Z,1-p) = \Delta(\mA_n) } = O(n^{-1/\theta - \delta}).
	\end{align}
\end{proof}

During the blow-up procedure that constructs $\mT_n^\Omega$ from $\mA_n$ we attach a random number $V_n$ of independent copies of $\mT^* \eqdist (\mT \mid L_\Omega(\mT) = 0)$. We need to ensure that it's sufficiently unlikely that any of these attached copies  contains a vertex with outdegree $\Delta(\mT_n^\Omega)$. The first step is to control $V_n$:

\begin{lemma}
	\label{le:bootstraplem}
	For any $\epsilon>0$
	\begin{align}
	\label{eq:sss}
	\Prb{V_n \notin n\left(\frac{p}{1-p} + \Ex{L}\Ex{X(1)} \right)(1 \pm \epsilon)} \le \exp(-\Theta(n)).
	\end{align}
\end{lemma}
\begin{proof}
	Recall the expression of $\zeta$ in~\eqref{eq:blowup}.	
	For each $k \ge 1$ let $(Y^{[k]}, \bm{X}_1^{[k]}, \ldots, \bm{X}^{[k]}_{L^{[k]}})$ denote an independent copy of $(Y, \bm{X}_1, \ldots, \bm{X}_{L})$, and set 
	\begin{align}
	\zeta_k = \mathrm{Bin}\left(Y^{[k]} + \sum_{i=1}^{L^{[k]}} X_i^{[k]}(2),1-p\right).
	\end{align}
	Thus, $(\zeta_k)_{k \ge 1}$ is a family of independent copies of $\zeta$. The conditioned $\zeta$-Galton--Watson tree $\mA_n$ corresponds naturally to $(\zeta_1, \ldots, \zeta_n)$ conditioned on the event
	\begin{align}
	\cE_n:= \left\{\min\left\{ m \ge 1 \,\,\Bigg\vert\,\, \sum_{k=1}^m(\zeta_k - 1) = -1\right\} = n \right\}.
	\end{align}
	Recall from Figure~\ref{fi:blowup2} that we construct the tree $\mT_n^\Omega$ from $\mA_n$ by a blow-up procedure, where each vertex is expanded into a vertebrate. All vertices to the left of the spine become roots of independent copies of $\mT^*$. Among the vertices to the right of the spine and the vertices dangling from the tip of the spine a random binomial subset with parameter $p$ becomes is selected, an each vertex from that subset becomes the root of an independent copy of $\mT^*$. Specifically, for the $k$th vertex $v_k$ (with outdegree $\zeta_k$) of $\mA_n$, the number of vertices to the left of the spine in the blow-up is given by
	\[
	\left( \sum_{i=1}^{L^{[k]}} X_i^{[k]}(1) \,\,\Bigg\vert\,\, \cE_n\right).
	\]
	The remaining number of vertices of the vertebrate in the blow-up of the $k$th vertex that become independent copies of $\mT^*$ may be expresssed by
	\[
	\left(Y^{[k]} + \sum_{i=1}^{L^{[k]}} X_i^{[k]}(2) - \zeta_k \,\,\Bigg\vert\,\, \cE_n \right).
	\]
	
	Let us write
	\begin{align}
	\label{eq:jjj1}
	V_n = V_n(1) +V_n(2)
	\end{align}
	with $V_n(1)$ the total number of vertices appearing to the left of the spine in all the blow-ups. Hence
	\begin{align}
	V_n(1) = \left( \sum_{k=1}^n \sum_{i=1}^{L^{[k]}} X_i^{[k]}(1) \,\,\Bigg\vert\,\, \cE_n\right).
	\end{align}
	Note that $\Pr{\cE_n} = \Pr{|\mA|=n}$  varies regularly with index $-1-\alpha$ by~\eqref{eq:lomegat2}.  Moreover, $\sum_{i=1}^L X_i(1)$ has finite exponential moments and $\Exb{\sum_{i=1}^L X_i(1)} = \Ex{L} \Ex{X(1)}$. It follows from large deviation inequalities for sums of i.i.d. light-tailed random variables that for any constant $\epsilon>0$ 
	\begin{align}
	\label{eq:jjj2}
	\Pr{ V_n(1) \notin  n\Ex{L} \Ex{X(1)} \pm \epsilon n} &\le \Pr{|\mA|=n}^{-1} \Prb{ \sum_{k=1}^n \sum_{i=1}^{L^{[k]}} X_i^{[k]}(1) \notin  n\Ex{L} \Ex{X(1)} \pm \epsilon n} \\
	& = \exp(- \Theta(n)). \nonumber
	\end{align} 
	Using $\sum_{i=1}^n d_{\mA_n}^+(v_i) = n-1$ we may write
	\begin{align}
	V_n(2) = \left( \sum_{k=1}^n\left( Y^{[k]} + \sum_{i=1}^{L^{[k]}} X_i^{[k]}(2) \right) - (n-1)  \,\,\Bigg\vert\,\, \cE_n \right).
	\end{align}
	The event $\cE_n$ entails that
	\[
	\mathrm{Bin}\left( \sum_{k=1}^n\left( Y^{[k]} + \sum_{i=1}^{L^{[k]}} X_i^{[k]}(2) \right), 1-p \right) = n-1.
	\]
	Using Proposition~\ref{pro:binomial}, it follows that for any $\epsilon>0$
	\begin{align}
	\Prb{ \sum_{k=1}^n\left( Y^{[k]} + \sum_{i=1}^{L^{[k]}} X_i^{[k]}(2) \right) \notin \frac{n}{1-p} \pm \epsilon n  \,\,\Bigg\vert\,\, \cE_n } \le \exp(-\Theta(n)).
	\end{align}
	Hence
	\begin{align}
	\label{eq:jjj3}
	\Prb{ V_n(2) \notin \frac{p}{1-p}n \pm \epsilon n } \le \exp(-\Theta(n)).
	\end{align}
	Inequality~\eqref{eq:sss} now follows by combining~\eqref{eq:jjj1},\eqref{eq:jjj2}, and \eqref{eq:jjj3}.
\end{proof}

\begin{lemma}
	\label{le:attached}
	Let $D_n^*$ denote the largest degree of any vertex of $\mT_n^\Omega$ that belongs to one of the $V_n$ attached copies of $\mT^*$. Then it holds uniformly for all integers $\ell \ge n / \log^2 n$ that
	\begin{align}
	\label{eq:tossshow}
	\Pr{D_n^* = \ell} = o\left( \frac{1}{g_{\Omega}(n) n^{1/\theta}}\right).
	\end{align}
\end{lemma}
\begin{proof}
	The statement is trivial if $\Omega^c$ is finite. Hence it suffices to treat the case where $\Omega$ is finite. Using \cite[Prop. 3.2]{MR3687241}, it follows  that the maximal degree of $\Delta(\mT^*)$ satisfies
	\begin{align}
	\label{eq:dm}
	\Pr{\Delta(\mT^*) = n} & \le \frac{\Pr{\Delta(\mT) = n}}{\Pr{L_\Omega(\mT) = 0}} \\
	&= \frac{1 + o(1)}{(1- \Ex{\xi})\Pr{L_\Omega(\mT) = 0}} \Pr{\xi=n}. \nonumber
	\end{align}
	Letting $(\mT^*_i)_{i \ge 1}$ denote independent copies of $\mT^*$, it holds that
	\begin{align}
	\label{eq:dohdoh}
	D_n^* \eqdist \max_{1 \le i \le V_n} \Delta(\mT^*_i).
	\end{align}
	By Lemma~\ref{le:bootstraplem} there are constants $0<c<C$ such that
	\[
	\Pr{V_n \notin [cn, Cn]} \le \exp(-\Theta(n)).
	\]
	Moreover, using Inequality~\eqref{eq:dm} it follows that uniformly for all $k \in [cn, Cn]$ and $\ell \ge n / \log^2 n$
	\begin{align}
	\Prb{ \max_{1 \le i \le k} \Delta(\mT^*_i) = \ell} &\le k \Pr{\Delta(\mT^*) = \ell} \\
	&\le k \frac{1 + o(1)}{(1- \Ex{\xi})\Pr{L_\Omega(\mT) = 0}} \Pr{\xi=\ell} \nonumber. 
	\end{align}
	Note that $\alpha>1$ implies $\alpha > 1/\theta$. Hence, using the Potter bounds it follows that
	\begin{align}
	g_{\Omega}(n) n^{1/\theta} 	\Prb{ \max_{1 \le i \le k} \Delta(\mT^*_i) = \ell} = O(1)g_{\Omega}(n) n^{1/\theta} k f(\ell)\ell^{-1-\alpha} = o(1).
	\end{align}
	By~\eqref{eq:dohdoh} this verifies~\eqref{eq:tossshow}.
\end{proof}

We are now ready to complete the proof of Theorem~\ref{te:main1}:

\begin{proof}[Proof of Theorem~\ref{te:main1} for the case $0 \notin \Omega$]
	By Lemma~\ref{le:devD} and the Potter bounds it follows that 
	\begin{align}
	\Pr{\Delta(\mT_n^\Omega) < n / \log^2n} = o\left( \frac{1}{g_{\Omega}(n) n^{1/\theta}}\right).
	\end{align}
	Hence it suffices to prove
	\begin{align}
	\label{eq:lastline}
	\Pr{\Delta(\mT_n^\Omega) = \ell} = \frac{1}{g_\Omega(n)n^{1/\theta}}\left(h\left(\frac{ \Pr{\xi \in \Omega}^{-1}(1-\Ex{\xi})n - \ell}{g_\Omega(n)n^{1/\theta}}   \right) + o(1)\right)
	\end{align}
	uniformly for all integers $\ell \ge n / \log^2n$. 
	
	It is elementary that
	\begin{align}
	\Pr{\Delta(\mT_n^\Omega) = \ell} = \Pr{D_n = \ell} - \Pr{D_n = \ell, \Delta(\mT_n^\Omega) > \ell} + \Pr{\Delta(\mT_n^\Omega)=\ell, D_n < \ell}.
	\end{align}
	Using Lemma~\ref{eq:lltforD} and 
	it follows that in order to verify~\eqref{eq:lastline} it is sufficient to show
	\begin{align}
	\label{eq:conquest1}
	\Pr{D_n = \ell, \Delta(\mT_n^\Omega) > \ell} = o\left( \frac{1}{g_{\Omega}(n) n^{1/\theta}}\right),
	\end{align}
	and
	\begin{align}
	\label{eq:conquest2}
	\Pr{\Delta(\mT_n^\Omega)=\ell, D_n < \ell} = o\left( \frac{1}{g_{\Omega}(n) n^{1/\theta}}\right).
	\end{align}
	
	Similarly as for $D_n$, we let $D_n^\circ$ denote the largest outdegree of a \emph{spine} vertex created from the blow-up of a vertex of $\mA_n$ \emph{except} the lexicographically first vertex with outdegree $\Delta(\mA_n)$. Hence 
	\begin{align}
	\Delta(\mT_n^\Omega) = \max(D_n, D_n^*, D_n^\circ).
	\end{align}
	
	Our next intermediate step is to reduce to reduce~\eqref{eq:conquest1} and~\eqref{eq:conquest2} to properties of $D_n^\circ$ and $D_n$ only.
	Using $\ell \ge n / \log^2n$, it follows from Lemma~\ref{le:attached} that
	\begin{align}
	\Pr{\Delta(\mT_n^\Omega)=\ell, D_n < \ell} \le \Pr{D_n^{\circ}= \ell, D_n < \ell} + o\left( \frac{1}{g_{\Omega}(n) n^{1/\theta}}\right).
	\end{align}
	This reduces~\eqref{eq:conquest2} to showing
	\begin{align}
	\label{eq:wonder2}
	\Pr{D_n^{\circ}= \ell, D_n < \ell} = o\left( \frac{1}{g_{\Omega}(n) n^{1/\theta}}\right).
	\end{align}
	As for~\eqref{eq:conquest1}, it is elementary that 
	\begin{align}
	\Pr{D_n = \ell, \Delta(\mT_n^\Omega) > \ell} \le \Pr{D_n = \ell, D_n^\circ > \ell} + \Pr{D_n = \ell, D_n^* > \ell}
	\end{align}
	Using  Lemma~\ref{eq:lltforD} it follows that for any $\epsilon>0$ we may select a constant $M>0$ large enough so that 
	\begin{align}
	g_{\Omega}(n) n^{1/\theta} \Pr{D_n = \ell, D_n^* > \ell} \le \epsilon
	\end{align}
	uniformly for all $\ell \notin  \Pr{\xi \in \Omega}^{-1}(1-\Ex{\xi})n \pm M g_\Omega(n)n^{1/\theta}$.  For all $\ell \in\Pr{\xi \in \Omega}^{-1}(1-\Ex{\xi})n \pm M g_\Omega(n)n^{1/\theta}$ it follows from Lemma~\ref{le:bootstraplem} and Lemma~\ref{eq:lltforD} that there are constants $0<c<C$ with
	\begin{align}
	g_{\Omega}(n) n^{1/\theta} \Pr{D_n = \ell, D_n^* > \ell} &= O(1) \frac{\Pr{ D_n^* > \ell , D_n = \ell}}{\Pr{D_n = \ell}} \\
	&= O(1) \frac{\Pr{ D_n^* > \ell , D_n = \ell, cn < V_n < C_n} + o(1)}{\Pr{D_n = \ell, cn < V_n < Cn}(1 + o(1))} \nonumber\\
	&= O(1) \Pr{ D_n^* > \ell \mid D_n = \ell, cn  < V_n < C n} + o(1) \nonumber 
	\end{align}
	Now, conditionally on $V_n$, it holds that $D_n^*$ is distributed like the maximum of $V_n$ independent copies of $\Delta(\mT^*)$. Using~\eqref{eq:dm}  it follows that uniformly for $\ell \in\Pr{\xi \in \Omega}^{-1}(1-\Ex{\xi})n \pm M g_\Omega(n)n^{1/\theta}$
	\begin{align}
	g_{\Omega}(n) n^{1/\theta} \Pr{D_n = \ell, D_n^* > \ell} = o(1).
	\end{align}
	As $\epsilon>0$ was arbitrary, it follows that
	\begin{align}
	\Pr{D_n = \ell, D_n^* > \ell} = o\left( \frac{1}{g_{\Omega}(n) n^{1/\theta}}\right)
	\end{align}
	uniformly for all $\ell \ge n / \log^2 n$. Hence in order to verify~\eqref{eq:conquest1} it suffices to show that
	\begin{align}
	\label{eq:wonder1}
	\Pr{D_n = \ell, D_n^\circ > \ell} = o\left( \frac{1}{g_{\Omega}(n) n^{1/\theta}}\right)
	\end{align}
	Summing up, we have reduced the task of proving Theorem~\ref{te:main1} to verifying that~\eqref{eq:wonder2} and~\eqref{eq:wonder1} hold uniformly for all $\ell \ge n / \log^2 n $, that is,
	\begin{align}
	\label{eq:singlewonder}
	\max(\Pr{D_n^{\circ}= \ell, D_n < \ell},	\Pr{D_n = \ell, D_n^\circ > \ell}) = o\left( \frac{1}{g_{\Omega}(n) n^{1/\theta}}\right).
	\end{align}

	Let  $1 \le d \le n-1$ be an integer with $\Pr{\zeta=d}>0$. For ease of notation, set
	\begin{align}
	t_n = n /\log^2 n.
	\end{align}
	Consider the event $\cE_1$ that
	\begin{align}
	\max(Y, X_1(1) + X_1(2) +1, \ldots,  X_L(1) + X_L(2) +1) \ge t_n, 
	\end{align}
	and the event $\cE_2$ that
	\begin{align}
	\mathrm{Bin}\left(Y + \sum_{i=1}^L X_i(2), 1-p\right) = d.
	\end{align}
	The sum $\sum_{i=1}^L X_i(1)$ has finite exponential moments. Hence
	\begin{align}
	\Prb{\cE_1, Y + \sum_{i=1}^L X_i(2) < t_n / 2} \le \Prb{1 + \sum_{i=1}^L X_i(1) \ge t_n/2} \le \exp(- \Theta(t_n)).
	\end{align}
	Consequently,
	\begin{align}
	\Prb{\cE_1, \cE_2, d < (1-p) t_n / 4} \le \exp(- \Theta(t_n)).
	\end{align}
	As $1 \le d \le n-1$ and $\Pr{\zeta=d}>0$, the probability $\Pr{\cE_2} = \Pr{\zeta=d}$ is bounded from below by some fixed polynomial in $n^{-1}$ that does not depend on $d$. For ease of notation, we set \begin{align}s_n = \lfloor (1-p) t_n / 4 \rfloor.\end{align} It follows that
	\begin{align}
	\Pr{ \cE_1 \mid \cE_2} \le \exp(-\Theta(t_n))
	\end{align}
	uniformly for all $1 \le d \le s_n$ with $\Pr{\zeta=d}>0$. 
	
	It follows that there is an upper bound of the form $\exp(-\Theta(t_n))$ for  the probability of the event that $\mA_n$ contains a vertex $v$ with outdegree less than $s_n$ so that the largest degree in the spine of the blow-up of $b$ is larger than $t_n$. In particular, there exists  a bound of the form  $\exp(-\Theta(t_n))$ for 	
	the probability that $D_n^\circ$  got produced in the spine of the blow-up of some vertex  of $\mA_n$ with outdegree less than $s_n$. Using~\eqref{eq:deldev}, it follows that
	\begin{align}
	\max(\Pr{D_n^{\circ}= \ell, D_n < \ell},	\Pr{D_n = \ell, D_n^\circ > \ell}) \le R_{n,\ell} +  o\left( \frac{1}{g_{\Omega}(n) n^{1/\theta}}\right).
	\end{align}
	for the probability $R_{n,\ell}$ that at least two vertices of $\mA_n$ have outdegree at least $s_n$ and the blow-up of one of them produces a spine vertex with outdegree $\ell$.
	
	This reduces the verification of~\eqref{eq:singlewonder} and hence also the completion of the proof of Theorem~\ref{te:main1} to proving
	\begin{align}
	\label{eq:comeoooon}
	R_{n, \ell} = o\left( \frac{1}{g_{\Omega}(n) n^{1/\theta}}\right)
	\end{align}
	uniformly for all integers $\ell \ge t_n$. 
	
	Let $(\zeta_i)_{i \ge 1}$ be a family of independent copies of the offspring distribution $\zeta$ and set $\tilde{S}_k = \zeta_1 + \ldots + \zeta_k$ for all $k$. Using Lemma~\ref{le:cyc} and Equations~\eqref{eq:teq},~\eqref{eq:gwt}, and~\eqref{eq:zetadensity} it follows that
	\begin{align}
	\label{eq:rnl1}
	R_{n,\ell} &\le \frac{1}{\Pr{\tilde{S}_n = n-1}} n^2 \sum_{y =  s_n}^n \Pr{\zeta=y}\Pr{\tilde{S}_{n-1} = n-1 - y, \zeta_1 \ge s_n} p_{\ell,y} \\
	&\le \frac{O(n)}{\Pr{\xi=n}} \sum_{y = s_n}^n \Pr{\xi=y}  \Pr{\tilde{S}_{n-1} = n-1 - y, \zeta_1 \ge s_n} p_{\ell,y}\nonumber
	\end{align}
	with $p_{\ell,y}$ denoting the conditional probability
	\begin{multline*}
	\mathbb{P} \Bigg(\max(Y, 1 + X_1(1) + X_1(2), \ldots, 1 + X_L(1) + X_L(2))=\ell \,\,\Bigg\vert\\\mathrm{Bin}\left( Y + \sum_{i=1}^L X_i(2) ,1-p\right) = y\Bigg).
	\end{multline*}
	Applying the local limit theorem for sums of independent copies of $\zeta$, it follows that uniformly in $\ell$ and $y$
	\begin{align}
	\Pr{\tilde{S}_{n-1} = n-1 - y, \zeta_1 \ge s_n} &= \sum_{i \ge s_n} \Pr{\zeta = i} \Pr{\tilde{S}_{n-2} = n-1 -y - i} \\
	&= \frac{O(1)}{g_A(n) n^{1/\theta}} \Pr{\zeta \ge s_n}\nonumber
	\end{align}
	As $\Pr{\zeta \ge s_n} = O(\Pr{\xi \ge s_n})$, it follows from~\eqref{eq:rnl1} that
	\begin{align}
	\label{eq:rnl2}
	R_{n,\ell} \le \frac{O(n) \Pr{\xi \ge s_n}}{g_A(n)n^{1/\theta} \Pr{\xi=n}}\sum_{y =  s_n }^n \Pr{\xi=y}   p_{\ell,y}.
	\end{align}
	As $ y \ge s_n = \Theta(n / \log n)$ and $y \le n$, we may apply Inequality~\eqref{eq:devi} and the local limit theorem in~\eqref{eq:llt111}, yielding 
	\begin{align}
	p_{\ell,y} =& O(\exp(- \Theta(\log^2 n))) \\ &+ \sum_{r \in \lfloor y/(1-p) \rfloor  \pm \sqrt{n} \log n}	  \Prb{Y + \sum_{i=1}^L X_i(2) = r \,\,\Bigg\vert\,\, \mathrm{Bin}\left( Y + \sum_{i=1}^L X_i(2) ,1-p\right) = y} q_{\ell,r} \nonumber\\
	=& O(\exp(- \Theta(\log^2 n))) + O(1/\sqrt{n}) \sum_{r \in y/(1-p) \pm \sqrt{n} \log n} q_{\ell,r} \nonumber
	\end{align}
	with
	\begin{align}
	q_{\ell,r} &:= \Prb{\max(Y, 1 + X_1(1) + X_1(2), \ldots, 1 + X_L(1) + X_L(2))=\ell \,\,\Bigg\vert\,\,  Y + \sum_{i=1}^L X_i(2) = r}.
	\end{align}
	It follows from Inequality~\eqref{eq:bob} that there are constants $c_0, r_0>0$ with
	\begin{align}
	q_{\ell,r} \le O(1) \frac{\Pr{\xi= \ell}}{\Pr{\xi = y}}\left( \one_{r-\ell\le r_0} \exp(-c_0|r-\ell|) + \one_{r-\ell> r_0} \Pr{\xi=r - \ell}  \right)
	\end{align}
	Thus,
	\begin{align}
	R_{n,\ell} \le & \,O(\exp(-\Theta(\log^2 n))) + \frac{O(n) \Pr{\xi \ge s_n} \Pr{\xi=\ell}}{g_A(n)n^{1/\theta} \Pr{\xi=n} \sqrt{n}}B_{\ell}
	\end{align}
	for
	\begin{align}	
	B_{\ell} &:= 
	\sum_{y =  s_n }^n \sum_{r \in \lfloor y/(1-p) \rfloor  \pm \sqrt{n} \log n}
	\left( \one_{r-\ell\le r_0} \exp(-c_0|r-\ell|) + \one_{r-\ell> r_0} \Pr{\xi=r - \ell}  \right) \\
	&\le \sum_{|s| \le  \sqrt{n} \log n} \sum_{y =  s_n }^n 
	\left(  \exp(-c_0| \lfloor y/(1-p) \rfloor + s - \ell |) + \Pr{\xi=| \lfloor y/(1-p) \rfloor + s - \ell |}  \right) \nonumber \\
	&= O(\sqrt{n} \log n) \nonumber
	\end{align}
	uniformly for all $\ell \ge t_n$. Using $\alpha>1$, $\ell \ge t_n$ and the Potter bounds it follows that
	\begin{align}
	\frac{O(n) \Pr{\xi \ge s_n} \Pr{\xi=\ell}}{g_A(n)n^{1/\theta} \Pr{\xi=n} \sqrt{n}}B_\ell  &=  \frac{O(n) \Pr{\xi \ge s_n} \Pr{\xi=\ell} \log n}{g_A(n)n^{1/\theta} \Pr{\xi=n} }\\
	&=  o\left( \frac{1}{g_{\Omega}(n) n^{1/\theta}}\right). \nonumber
	\end{align}
	This verifies~\eqref{eq:comeoooon} and completes the proof.
\end{proof}


\appendix

\section{Renewal Theory}

\label{sec:appendix}

Let $(Y_i)_{i \ge 1}$ and $(\bar{Y}_i)_{i \ge 1}$ denote independent copies of a non-negative random integer $Y$ that lies in the domain of attraction of a $\theta$-stable law for $1 < \theta \le 2$. Let $Y_0$ denote an arbitrary non-negative integer and let $\bar{Y}_0$ denote an independent copy of $Y_0$. We define the renewal process
\begin{align}
\tau_n = \sup\left\{k \ge 0 \bigg \rvert \sum_{i=0}^k Y_i \le n\right\} \in \{-\infty\} \cup \ndN_0.
\end{align}
Here we use that the convention that the supremum of the empty set equals $-\infty$. Thus 
\begin{align}
\Pr{\tau_n = -\infty} = \Pr{Y_0 > n}.
\end{align}

The present section is dedicated to the question, how much the sequence $(\bar{Y}_i)_{0 \le i \le \tau_n}$  deviates from $(Y_i)_{0 \le i \le \tau_n}$. Clearly there are some differences, as  $\sum_{i=0}^{\tau_n} \bar{Y}_i$ fluctuates  around $n$ at a larger order of magnitude than the sum $\sum_{i=0}^{\tau_n} Y_i$. However,  subfamilies may asymptotically become independent from each other and from $\tau_n$.

\subsection{Local limit theorems for first passage times}

Local limit theorems for first passage times have received attention in the literature, see for example~\cite{zbMATH06023979,zbMATH06666232} and references given therein. We collect some preparatory results.

Let $(X_t)_{t \ge 0}$ denote the spectrally positive L\'evy process with Laplace exponent $\Ex{\exp(-\lambda X_t)} = \exp(t\lambda^\theta )$.  $X_1$ has a density $h$ that is positive,  uniformly continuous, and bounded on $\ndR$ (see ~\cite[Sec. XVII.6]{MR0270403} and \cite{MR1406564}).  Let us check that the local limit theorem for $\sum_{i=1}^n Y_i $ still holds if the sum includes the additional term $Y_0$ that follows a different distribution than the rest.
\begin{proposition}
	\label{pro:extllt}
	There is a slowly varying function $g_Y$ such that 
	\begin{align}
	\label{eq:lltsecond}
	\lim_{n \to \infty} \sup_{\ell \ge 0} \left|g_Y(n) n^{1/\theta} \Prb{\sum_{i=0}^n Y_i = \ell} - h\left( \frac{\ell - n\Ex{Y}}{g_Y(n) n^{1/\theta}} \right) \right| = 0.
	\end{align}
\end{proposition}
\begin{proof}
	The classical local limit theorem~\cite[Thm. 4.2.1]{MR0322926} implies that there is a slowly varying function $g_Y$ such that 
	\begin{align*}
	\lim_{n \to \infty} \sup_{\ell \ge 0} \left|g_Y(n) n^{1/\theta} \Prb{\sum_{i=1}^n Y_i = \ell} - h\left( \frac{\ell - n\Ex{Y}}{g_Y(n) n^{1/\theta}} \right) \right| = 0.
	\end{align*}
	It follows that, with an $o(1)$ term that is uniform in $\ell \in \ndZ$,
	\begin{align*}
	g_Y(n)n^{1/\theta}\Prb{\sum_{i=0}^n Y_i = \ell} &= \sum_{k = 0}^\ell \Pr{Y_0 = k}\left( h\left( \frac{\ell-k  - n\Ex{Y}}{g_Y(n) n^{1/\theta}} \right) + o(1)\right) \\
	&= o(1) + \sum_{k = 0}^\ell \Pr{Y_0 = k} h\left( \frac{\ell-k  - n\Ex{Y}}{g_Y(n) n^{1/\theta}} \right) .
	\end{align*}
	For any $\epsilon>0$ we may select a constant integer $K>0$ such that $\Pr{Y_0 > K} < \epsilon$. Using that $h$ admits an upper bound $H>0$ and is uniformly continuous, it follows that
	\begin{align*}
	\left|g_Y(n) n^{1/\theta} \Prb{\sum_{i=0}^n Y_i = \ell} - h\left( \frac{\ell - n\Ex{Y}}{g_Y(n) n^{1/\theta}} \right) \right| \le o(1) + 2\epsilon H. 
	\end{align*}
	As this holds for arbitrary (but fixed) $\epsilon>0$, Equation~\eqref{eq:lltsecond} follows.
\end{proof}

The stopping time $\tau_n$ satisfies a similar local limit theorem for the remaining possible values.
\begin{proposition}
	\label{pro:fllt}
	Let $(\ell_n)_{ n \ge 1}$ denote a sequence of positive integers that satisfies $\ell_n \to \infty$ as $n$ becomes large. Then
	\begin{align}
	\label{eq:firstrllt}
	\sup_{\ell \ge \ell_n} \left|\Ex{Y}^{-1} g_Y(\ell) \ell^{1/\theta} \Pr{\tau_n = \ell} - h\left( \frac{n - \ell\Ex{Y}}{g_Y(\ell) \ell^{1/\theta}} \right) \right| = 0.
	\end{align}
\end{proposition}
\begin{proof}
	Using Proposition~\ref{pro:extllt}, it follows that, for a uniform $o(1)$ term,
	\begin{align*}
	g_Y(\ell)\ell^{1/\theta} \Pr{\tau_n= \ell} &= 	\sum_{k =0}^n g_Y(\ell)\ell^{1/\theta}  \Prb{\sum_{i=0}^\ell Y_i = n-k}\Pr{Y >k} \\ 
	&= \sum_{k=0}^n \left( h\left( \frac{n-k  - \ell\Ex{Y}}{g_Y(\ell) \ell^{1/\theta}} \right) + o(1)\right)\Pr{Y >k} \\
	&= o(1) + \sum_{k=0}^n  h\left( \frac{n-k  - \ell\Ex{Y}}{g_Y(\ell) \ell^{1/\theta}} \right)\Pr{Y >k}.
	\end{align*}
	As $\Ex{Y}< \infty$, for any $\epsilon>0$ we may select a constant $K>0$ with $\sum_{k >K} \Pr{Y>k}<\epsilon$. Using that the density $h$ has an upper bound $H>0$ and is uniformly continuous, it follows that uniformly for $\ell \ge \ell_n$
	\begin{align*}
	\left|g_Y(\ell) \ell^{1/\theta} \Pr{\tau_n = \ell} -  \Ex{Y} h\left( \frac{n - \ell\Ex{Y}}{g_Y(\ell) \ell^{1/\theta}} \right) \right| \le o(1) + 2 \epsilon H.
	\end{align*}
	As this holds for any fixed $\epsilon>0$, Equation~\eqref{eq:firstrllt} follows.
\end{proof}

\begin{corollary}
	\label{co:fllt}
	Let $\delta>0$ be given.
	Then 
	\begin{align}
	\label{eq:tov}
	\lim_{n \to \infty} \sup_{\ell \ge \delta n} \left|g_Y(n) n^{1/\theta}\Ex{Y}^{-1-1/ \theta} \Pr{\tau_n = \ell} - h\left( \frac{n/\Ex{Y} - \ell}{g_Y(n) n^{1/\theta}\Ex{Y}^{-1-1/ \theta}} \right) \right| = 0
	\end{align}
	and
	\begin{align}
	\label{eq:tow}
	\frac{n/\Ex{Y} - \tau_n}{g_Y(n)n^{1/\theta} \Ex{Y}^{-1-1/\theta}} \convdis X_1.
	\end{align}
\end{corollary}
\begin{proof}
	The central limit theorem in~\eqref{eq:tow} is a direct consequence of the local limit theorem~\eqref{eq:tov}, so it suffices to verify the latter.
	
	Let $M>0$ be a constant, and define the interval $I_{M,n} := n/\Ex{Y} \pm  M g_Y(n)n^{1/\theta}$. Then $g_Y(\ell)\ell^{1/\theta} \sim g_Y(n) n^{1/\theta} \Ex{Y}^{-1/\theta}$ uniformly for all integers $\ell \in \ndZ  \cap I_{M,n}$ as $n$ tends to infinity. By Proposition~\ref{pro:fllt}, it follows that~\eqref{eq:tov} holds uniformly for $\ell \in I_{M,n}$. As this holds for any fixed $M>0$, it follows that there exists a sequence $(M_n)_{n \ge 1}$ with $M_n \to \infty$ such that~\eqref{eq:tov} holds uniformly for $\ell \in I_{M_n,n}$. 
	
	It holds uniformly for $\ell \in \ndZ \setminus I_{M_n,n}$ that
	\[
	\left| \frac{n/\Ex{Y} - \ell}{g_Y(n) n^{1/\theta}\Ex{Y}^{-1-1/ \theta}}  \right| \to \infty
	\]
	and hence
	\[
	h\left( \frac{n/\Ex{Y} - \ell}{g_Y(n) n^{1/\theta}\Ex{Y}^{-1-1/ \theta}} \right) = o(1).
	\]
	Moreover it follows from Proposition~\ref{pro:fllt} that uniformly for $\ell \in \ndZ \cap ( [\delta n, \infty[\, \setminus\, I_{M_n,n})$
	\begin{align*}
	g_Y(n) n^{1/\theta} \Pr{\tau_n= \ell} &= \frac{g_Y(n)n^{1/\theta}}{g_Y(\ell)\ell^{1/\theta}} \left(o(1) + h\left( \frac{g_Y(n)n^{1/\theta}}{o(1)g_Y(\ell)\ell^{1/\theta}} \right) \right) \\
	&= o(1).
	\end{align*}
	Here we have used that the density $h$ is bounded, and that $\ell \ge \delta n$ ensures that $\frac{g_Y(n)n^{1/\theta}}{g_Y(\ell)\ell^{1/\theta}} = O(1)$. (In detail: if the argument of the $h$-function in this expressions has an absolute value that becomes large, then the entire expression tends to zero since $\frac{g_Y(n)n^{1/\theta}}{g_Y(\ell)\ell^{1/\theta}} = O(1)$. The only way for this case not to happen is when $\frac{g_Y(n)n^{1/\theta}}{g_Y(\ell)\ell^{1/\theta}}$ tends to zero, but in this case the fact that $h$ is bounded ensures that the entire expression tends to zero.)
\end{proof}

\subsection{Decoupling}

Recall that the family $(\bar{Y}_i)_{i \ge 0}$ is defined to be independent from $(Y_i)_{i \ge 0}$ and hence $\tau_n$. The following proposition ensures that the majority of coordinates of $({Y}_i)_{0 \le i \le \tau_n}$ become asymptotically independent from each other.
\begin{proposition}
	\label{pro:firstbit}
	Let $(u_n)_{n \ge 1}$ denote a sequence of positive integers such that \begin{align}
	\frac{u_n - n/\Ex{Y}}{g_Y(n)n^{1/\theta}} \to - \infty.
	\end{align}
	Then
	\begin{align}
	\label{eq:ap1}
	(Y_0, \ldots, Y_{\min(\tau_n, u_n)}) \atv (\bar{Y}_0, \ldots, \bar{Y}_{u_n}).
	\end{align}
\end{proposition}
\begin{proof}
	The central limit theorem for $\tau_n$ in~\eqref{eq:tow} implies that the probability for $u_n < \tau_n$ tends to $1$ as $n$ tends to infinity. This readily verifies~\eqref{eq:ap1}.
\end{proof}

Care has to be taken that although most coordinates of $(Y_i)_{0 \le i \le \tau_n}$ become asymptotically independent from each other, the dependence on $\tau_n$ perseveres. For this reason, only a weaker contiguousness relation holds:

\begin{lemma}
	\label{le:doit}
	Let $\epsilon>0$ and  $0<\delta<1/\Ex{Y}$ be given, and set $t_n = \lfloor \delta n \rfloor$. There are constants $0<c<C$ such that for all collections $\cE$ of finite sequences of integers
	\begin{align}
	\label{eq:ap2}
	c \Pr{(\bar{Y}_0, \ldots, \bar{Y}_{\tau_n - t_n}) \in \cE} - \epsilon &\le \Pr{ ( Y_0, \ldots, Y_{\tau_n - t_n}) \in \cE} \\&\le  C \Pr{(\bar{Y}_0, \ldots, \bar{Y}_{\tau_n - t_n}) \in \cE} + \epsilon. \nonumber
	\end{align}
\end{lemma}
Note that the vector $( Y_0, \ldots, Y_{\tau_n - t_n})$ determines $\tau_n$, since $t_n$ is deterministic.

\begin{proof}[Proof of Lemma~\ref{le:doit}]
	For all $M_1, M_2>0$ we consider the collection $\cE_{n, \delta, M_1,M_2}$ of finite sequences $\bm{y} = (y_0, \ldots, y_k)$, $k \ge 1$,  of non-negative integers satisfying the following properties:
	\begin{enumerate}
		\item $\Pr{Y_0 = y_0}>0$  and $\Pr{Y=y_i} >0$ for all integers $1 \le i \le k$.
		\item $\sum_{i=0}^{k} y_i \in k\Ex{Y} \pm M_1 g_Y(k) k^{1/\theta}$.
		\item $k+t_n \in n/\Ex{Y} \pm M_2 g_Y(n) n^{1/\theta}$.
	\end{enumerate}
	For any such sequence set $\ell = \sum_{i=0}^k y_i$. Then
	\begin{align}
	\label{eq:t1}
	\frac{\Pr{ ( Y_0, \ldots, Y_{\tau_n - t_n}) = \bm{y}}} {\Pr{(\bar{Y}_0, \ldots, \bar{Y}_{\tau_n - t_n}) = \bm{y}}} = \frac{\Pr{\tau_{n-\ell} = t_n }}{\Pr{\tau_n= k + t_n}}.
	\end{align}
	Our assumptions ensure that $t_n / g_Y(n)n^{1/\theta} \to \infty$ and hence $n - \ell \to \infty$. This allows us apply Proposition~\ref{pro:fllt} and Corollary~\ref{co:fllt}, yielding
	\begin{align}
	\frac{\Pr{\tau_{n-\ell} = t_n }}{\Pr{\tau_n= k + t_n}} &= \frac{g_Y(n) n^{1/\theta}}{\Ex{Y}^{1/\theta}g_Y(t_n)t_n^{1/\theta}} \frac{ h\left(\frac{n - \ell - t_n\Ex{Y}}{g_Y(t_n)t_n^{1/\theta}} \right) + o(1)}{h\left(\frac{n - (k+t_n)\Ex{Y}}{g_Y(n)n^{1/\theta} \Ex{Y}^{-1/\theta}} \right) + o(1)}.
	\end{align}
	If we write $k + t_n = n/\Ex{Y} + y g_Y(n)n^{1/\theta}$ with $|y| \le M_2$ and $\ell= k \Ex{Y} + x g_Y(k)k^{1/\theta}$ with $|x| \le M_1$, then
	\[
	\frac{n - (k+t_n)\Ex{Y}}{g_Y(n)n^{1/\theta} \Ex{Y}^{-1/\theta}} = -y \Ex{Y}^{1 + 1/\theta}
	\]
	and
	\[
	\frac{n - \ell - t_n\Ex{Y}}{g_Y(t_n)t_n^{1/\theta}} = \frac{-y g_Y(n) n^{1/\theta} \Ex{Y} -x g_Y(k)k^{1/\theta}}{g_Y(t_n)t_n^{1/\theta}}.
	\]
	Using $t_n = \lfloor \delta n \rfloor$, it follows that uniformly for all $\bm{y} \in \cE_{n, \delta, M_1,M_2}$
	\begin{align}
	\frac{1 + o(1)}{\Ex{Y}^{1/\theta}\delta^{1/\theta}} \frac{\inf_{|z| \le \delta^{-1/\theta}(M_2\Ex{Y} + M_1|\Ex{Y}^{-1} - \delta|^{1/\theta}) }h(z) }{\sup_{z \in \ndR} h(z)} &\le \frac{\Pr{ ( Y_0, \ldots, Y_{\tau_n - t_n}) = \bm{y}}} {\Pr{(\bar{Y}_0, \ldots, \bar{Y}_{\tau_n - t_n}) = \bm{y}}} \\ &\le \frac{1 + o(1)}{\Ex{Y}^{1/\theta} \delta^{1/\theta}} \frac{ \sup_{z\in \ndR} h(z)}{ \inf_{|z| \le M_2 \Ex{Y}^{1 + 1/\theta} } h(z)}. \nonumber
	\end{align}
	
	Corollary~\ref{co:fllt} and Proposition~\ref{pro:firstbit} ensure that for any $\epsilon>0$ we may select $M_1$ and $M_2$ sufficiently large so that $(Y_0, \ldots, Y_{\tau_n - t_n})$ and $(\bar{Y}_0, \ldots, \bar{Y}_{\tau_n - t_n})$ lie in $\cE_{n, \delta, M_1,M_2}$ with probability at least $1 - \epsilon$ for all sufficiently large $n$. This concludes the proof.
\end{proof}

A small portion of the coordinates becomes fully independent of each other \emph{and} of $\tau_n$:

\begin{lemma}
	\label{le:doit2}
	Let $(k_n)_{n \ge 1}$ denote a sequence of positive integers satsifying $k_n = o(n)$. Then
	\begin{align}
	(\tau_n, Y_0, \ldots, Y_{\min(\tau_n,k_n)}) \atv (\tau_n, \bar{Y}_0, \ldots, \bar{Y}_{k_n}).
	\end{align}
\end{lemma}
\begin{proof}
	For any numbers $M_1, M_2>0$ and any integer $k>0$ we may consider the collection $\cE_{M_1, M_2, k}$ of vectors $(t, y_0, \ldots, y_k)$ of positive integers satisfying the following properties:
	\begin{enumerate}
		\item $\Pr{Y_0 = y_0}>0$  and $\Pr{Y=y_i} >0$ for all integers $1 \le i \le k$.
		\item $\sum_{i=0}^{k} y_i \in k\Ex{Y} \pm M_1 g_Y(k) k^{1/\theta}$.
		\item $t \in n/\Ex{Y} \pm M_2 g_Y(n) n^{1/\theta}$.
	\end{enumerate}
	Setting $\ell= \sum_{i=0}^k y_i$, it holds that
	\begin{align}
	\label{eq:now}
	\frac{\Pr{ (\tau_n, Y_0, \ldots, Y_{k}) = \bm{y}}} {\Pr{(\tau_n, \bar{Y}_0, \ldots, \bar{Y}_{k}) = \bm{y}}} = \frac{\Pr{\tau_{n-\ell} = t -k }}{\Pr{\tau_n= t}}.
	\end{align}
	For $k = k_n$ our assumptions entail $n - \ell \to \infty$ and $t-k_n \to \infty$. This allows us to apply Proposition~\ref{pro:fllt} and Corollary~\ref{co:fllt}, yielding
	\begin{align}
	\frac{\Pr{\tau_{n-\ell} = t -k_n }}{\Pr{\tau_n= t}} = \frac{g_Y(n) n^{1/\theta}}{\Ex{Y}^{1/\theta}g_Y(t - k_n)(t-k_n)^{1/\theta}} \frac{ h\left(\frac{n - \ell - (t-k_n)\Ex{Y}}{g_Y(t - k_n)(t - k_n)^{1/\theta}} \right) + o(1)}{h\left(\frac{n - t\Ex{Y}}{g_Y(n)n^{1/\theta} \Ex{Y}^{-1/\theta}} \right) + o(1)}.
	\end{align}
	Since $k_n = o(n)$, it follows that $t - k_n \sim n/\Ex{Y}$ and $\ell = o(n)$, yielding
	\begin{align}
	\frac{\Pr{ (\tau_n, Y_0, \ldots, Y_{k}) = \bm{y}}} {\Pr{(\tau_n, \bar{Y}_0, \ldots, \bar{Y}_{k}) = \bm{y}}} = 1 + o(1)
	\end{align}
	uniformly for all $(t,y_0, \ldots, y_k) \in \cE_{M_1, M_2, k_n}$. 
	
	For any $\epsilon>0$ we may select $M_1,M_2>0$ large enough so that the vector \[(\tau_n, \bar{Y}_0, \ldots, \bar{Y}_{k_n})\] lies in the set $\cE_{M_1, M_2, k_n}$ with probability at least $1 - \epsilon$. This completes the proof.
\end{proof}


\ACKNO{I thank the referee and the associate editor for the thorough reading and  helpful suggestions, such as  the application of the renewal theorem in the proof of Lemma~\ref{le:limit}. I  also thank them for the encouragement to remove the condition $0 \in \Omega$.}


\end{document}